\documentclass[11pt,fleqn,reqno]{article}
\usepackage{amsmath}
\usepackage{amssymb}
\usepackage{amsfonts}
\usepackage{amsthm}
\usepackage{mathtools}
\usepackage{lscape} 
\usepackage{comment}
\mathtoolsset{
above-intertext-sep = -1ex
}

\usepackage[utf8]{inputenc}
\usepackage[T1]{fontenc}

\usepackage[sans]{dsfont}

\usepackage[left=2.5cm, right=2.5cm, top=2.5cm, bottom=2.5cm]{geometry}

\usepackage[%
cal=euler,
bb=fourier,
scr=zapfc
]
{mathalfa}

\usepackage{subcaption}
\usepackage{graphicx}
\usepackage{wrapfig}
\graphicspath{{./Figures/}}
\usepackage{acronym}
\usepackage{latexsym}
\usepackage{paralist}
\usepackage{wasysym}
\usepackage{xspace}
\usepackage{framed}
\usepackage{palatino,pxfonts}
\usepackage{authblk}
 \usepackage{booktabs}
\usepackage{adjustbox}
 \usepackage{makecell}
 \usepackage{rotating}
 \usepackage{multirow}
\usepackage{longtable}
\usepackage{graphicx}
\usepackage{array}
\usepackage[numbers,sort&compress]{natbib}

\usepackage[dvipsnames,svgnames]{xcolor}
\colorlet{MyBlue}{DodgerBlue!75!Black}
\colorlet{MyGreen}{DarkGreen!95!Black}

\usepackage{hyperref}
\hypersetup{
colorlinks=true,
linktocpage=true,
pdfstartview=FitH,
breaklinks=true,
pdfpagemode=UseNone,
pageanchor=true,
pdfpagemode=UseOutlines,
plainpages=false,
bookmarksnumbered,
bookmarksopen=false,
bookmarksopenlevel=1,
hypertexnames=true,
pdfhighlight=/O,
urlcolor=MyBlue!60!black,linkcolor=MyBlue!70!black,citecolor=DarkGreen!70!black, 
pdftitle={},
pdfauthor={},
pdfsubject={},
pdfkeywords={},
pdfcreator={pdfLaTeX},
pdfproducer={LaTeX with hyperref}
}

\numberwithin{equation}{section}  
\usepackage[sort&compress,capitalize,nameinlink]{cleveref}
\crefname{example}{Ex.}{Exs.}

\crefrangeformat{equation}{\upshape(#3#1#4)\textendash(#5#2#6)}

\newcommand{\eps}{\varepsilon}

\newcommand{\bigoh}{\mathcal{O}}

\DeclareMathOperator*{\argmin}{argmin}
\DeclareMathOperator*{\argmax}{argmax}

\DeclareMathOperator{\cl}{cl}

\DeclareMathOperator{\dom}{dom}

\DeclareMathOperator{\Int}{int}

\DeclareMathOperator{\tr}{tr}

\newcommand{\ca}{\mathtt{a}}

\newcommand{\ce}{\mathtt{e}}

\newcommand{\bA}{\mathbf{A}}

\newcommand{\bX}{\mathbf{X}}
\newcommand{\bb}{\mathbf{b}}

\newcommand{\UU}{\mathbf{U}}
\newcommand{\bx}{\mathbf{x}}
\newcommand{\by}{\mathbf{y}}

\newcommand{\bv}{\mathbf{v}}
\newcommand{\bu}{\mathbf{u}}

\newcommand{\bw}{\mathbf{w}}

\newcommand{\bI}{\mathbf{I}}
\newcommand{\bL}{\mathbf{L}}
\newcommand{\bP}{{\mathbf P}}
\newcommand{\bV}{{\mathbf V}}

\renewcommand{\emptyset}{\varnothing}

\newcommand{\setQ}{\mathsf{Q}}

\newcommand{\scrB}{\mathcal{B}}

\newcommand{\scrD}{\mathcal{D}}
\newcommand{\scrE}{\mathcal{E}}

\newcommand{\scrG}{\mathcal{G}}

\newcommand{\scrL}{\mathcal{L}}

\newcommand{\scrP}{\mathcal{P}}

\newcommand{\scrS}{\mathcal{S}}

\newcommand{\scrV}{\mathcal{V}}

\newcommand{\setC}{\mathsf{C}}

\newcommand{\setH}{\mathsf{H}}
\newcommand{\setE}{\mathsf{E}}
\newcommand{\setK}{\mathsf{K}}
\newcommand{\setX}{\mathsf{X}}

\newcommand{\0}{\mathbf{0}}
\newcommand{\1}{\mathbf{1}}
\newcommand{\Rn}{\R^n}

\newcommand{\R}{\mathbb{R}}

\newcommand{\N}{\mathbb{N}}
\newcommand{\K}{\mathbb{K}}
\newcommand{\symm}{\mathbb{S}}
\DeclareMathOperator{\NC}{\mathsf{NC}}

\newcommand{\bC}{{\mathbf{C}}}

\DeclareMathOperator{\Opt}{Opt}							

\DeclareMathOperator{\CG}{\texttt{CG}}

\DeclareMathOperator{\LCG}{\texttt{LCG}}
\DeclareMathOperator{\ICG}{\texttt{ICG}}

\DeclareMathOperator{\gap}{\mathsf{Gap}}

\usepackage{algorithm}
\usepackage{algpseudocode}

\DeclareMathOperator{\CGAL}{CGAL}
\theoremstyle{plain}
\newtheorem{theorem}{Theorem}

\newtheorem*{corollary*}{Corollary}
\newtheorem{lemma}[theorem]{Lemma}
\newtheorem{proposition}[theorem]{Proposition}

\theoremstyle{definition}
\newtheorem{definition}[theorem]{Definition}
\newtheorem*{definition*}{Definition}
\newtheorem{assumption}{Assumption}

\newcommand{\ms}[1]{{\color{black}#1}}
\definecolor{applegreen}{rgb}{0.55, 0.71, 0.0}
\newcommand{\gi}[1]{{\color{black}#1}}
\newcommand{\sh}[1]{{\color{black}#1}}
\newcommand{\pdv}[1]{{\color{black} #1}}
\theoremstyle{remark}
\newtheorem{remark}{Remark}
\newtheorem*{remark*}{Remark}
\newtheorem*{notation*}{Notational remark}
\newtheorem{example}{Example}

\numberwithin{theorem}{section}
\numberwithin{remark}{section}
\numberwithin{example}{section}

\DeclarePairedDelimiter{\abs}{\lvert}{\rvert}

\DeclarePairedDelimiter{\inner}{\langle}{\rangle}
\DeclarePairedDelimiter{\norm}{\lVert}{\rVert}

\newacro{ASFW}{Away-Step Frank Wolfe}
\newacroplural{VI}[VIs]{variational inequalities}
\newacro{iid}[i.i.d.]{independent and identically distributed}
\newacro{FOM}{First-order method}
\newacro{LMO}{Linear Mnimization Oracle}
\newacro{LLOO}{Local Linear minimization oracle}
\newacro{FW}{Frank-Wolfe}
\newacro{CG}{Conditional Gradient}
\newacro{SC}{self-concordant}
\newacro{GSC}{generalized self-concordant}
\newacro{CndG}{Conditional Gradient}
\newacro{ILMO}{Inexact Linear Minimization Oracle}
\DeclareMathOperator{\SDPT3}{\texttt{SDPT3}}

\title{A conditional gradient homotopy method with applications to Semidefinite Programming}

\date{\today}

\author[1]{\small Pavel Dvurechensky}
\author[4]{\small Gabriele Iommazzo}
\author[2]{\small Shimrit Shtern}
\author[3]{\small Mathias Staudigl\thanks{Corresponding Author}}

\affil[1]{\footnotesize Weierstrass Institute for Applied Analysis and Stochastics, Mohrenstr. 39, 10117 Berlin, Germany\\
(\href{mailto:Pavel.Dvurechensky@wias-berlin.de}{Pavel.Dvurechensky@wias-berlin.de})}
\affil[2]{\footnotesize Faculty of Data and Decision Sciences, Technion - Israel Institute of Technology, Haifa, Israel\\
(\href{mailto:shimrits@technion.ac.il}{shimrits@technion.ac.il})}
\affil[3]{\footnotesize Department of Mathematics, 68159 Mannheim, B6, Germany\\
(\href{mathias.staudigl@uni-mannheim.de}{mathias.staudigl@uni-mannheim.de})}
\affil[4]{\footnotesize Zuse Institute Berlin, Berlin, Germany\\
(\href{iommazzo@zib.de}{iommazzo@zib.de})}

\begin{document}
\maketitle

\begin{abstract}
We propose a new homotopy-based conditional gradient method for solving convex optimization problems with a large number of simple conic constraints.
Instances of this template naturally appear in semidefinite programming problems arising as convex relaxations of combinatorial optimization problems.
Our method is a double-loop algorithm in which the conic constraint is treated via a self-concordant barrier, and the inner loop employs a conditional gradient algorithm to approximate the analytic central path, while the outer loop updates the accuracy imposed on the temporal solution and the homotopy parameter.
Our theoretical iteration complexity is competitive when confronted to state-of-the-art SDP solvers, with the decisive advantage of cheap projection-free subroutines. Preliminary numerical experiments are provided for illustrating the practical performance of the method.

\end{abstract}

\section{Introduction}
\label{sec:intro}
%
In this paper we investigate a new algorithm for solving a large class of convex optimization problems with conic constraints of the following form: 
\begin{equation}\label{eq:Opt}\tag{P}
 \min_{x} g(x) \quad \text{s.t. } x\in\setX, \scrP(x)\in\setK\subseteq\setH.
\end{equation}
The objective function $g:\setE\to(-\infty,\infty]$ is assumed to be closed, convex, and proper, $\setX\subseteq\setE$ is a compact convex set, $\scrP:\setE\to\setH$ is an affine mapping between two finite-dimensional Euclidean vector spaces $\setE$ and $\setH$, and $\setK$ is a closed convex pointed cone. Problem \eqref{eq:Opt} is sufficiently generic to cover many optimization settings considered in the literature. 
\begin{example}[Packing SDP]
The model template \eqref{eq:Opt} covers the large class of packing semidefinite programs (SDPs) \cite{Iyengar:2010wf,ElbMakNaj22}. In this problem class we have $\setE=\symm^{n}$ -- the space of symmetric $n\times n$ matrices, a collection of positive semidefinite input matrices $\bA_{1},\ldots,\bA_{m}\in\symm^{n}_{+}$, with constraints given by $\scrP(\bX)=[1-\tr(\bA_{1}\bX);\ldots;1-\tr(\bA_{m}\bX)]^{\top}$ and $\setK=\R^{m}_{+}$, coupled with the objective function $g(\bX)=\tr(C\bX)$ and the set $\setX=\{\bX\in\symm^{n}_{+}\vert \tr(\bX)\leq\rho\}$. Special instances of packing SDPs include the MaxCut SDP relaxation of \cite{Goemans:1995aa}, the Lovasz-$\theta$ function SDP, among many others; see \cite{IyengarSIOPT11}. 
\end{example}
\begin{example}[Covering SDP]
\label{ex:covering}
A dual formulation of the packing SDP is the covering problem \cite{ElbMakNaj22}. In its normalized version, the problem reads as 
\begin{align*}
 & \min_{x}\sum_{i=1}^{m}x_{i} \quad 
 \text{s.t.:}  \sum_{i=1}^{m}x_{i}\bA_{i}-\bI\succeq 0, x\in\R^{m}_{+}, 
\end{align*}
where $\bA_{1},\ldots,\bA_{m}\in\symm^{n}_{+}$ are given input matrices. This can be seen as a special case of problem \eqref{eq:Opt} with $g(x)=\sum_{i=1}^{m}x_{i},\scrP(x)=\sum_{i=1}^{m}x_{i}\bA_{i}-\bI,\setK=\symm^{n}_{+},\setH=\symm^{n},\setE=\R^{m}$, and the constraint set $\setX=\{x\in\R^{m}_{+}\vert\sum_{i=1}^{m}x_{i}\leq \rho\}$ for some large $\rho$ added to the problem for the compactification purposes.
\end{example}
\begin{example}[Relative entropy programs]
Relative entropy programs are conic optimization problems in which a linear functional of a decision variable is minimized subject to linear constraints and conic constraints specified by a \emph{relative entropy cone}. The relative entropy cone is defined for triples $(u,v,w)\in\R^{n}\times\R^{n}\times\R^{n}$ via Cartesian products of the following elementary cone 
$$
\setK^{(1)}_{\text{RE}}=\{(\nu,\lambda,\delta)\in\R_{+}\times\R_{+}\times\R\vert -\nu\log\left(\frac{\lambda}{\nu}\right)\leq \delta\}.
$$
As the function $\R^{2}_{++}\ni (\nu,\lambda )\mapsto-\nu\log(\lambda/\nu)$ is the perspective transform of the negative logarithm, we know that it is a convex function \cite{BoyVan04}. Hence,  $\setK^{(1)}_{\text{RE}}$ is a convex cone. The relative entropy cone is accordingly defined as $\setK=\prod_{i=1}^{n}\setK^{(1)}_{\text{RE}}$. This set appears in many important applications, containing estimation of dynamical systems and the optimization of quantum channels; see \cite{chandrasekaran2017relative} for an overview. 
\end{example}
\begin{example}[Relative Entropy entanglement]
\label{ex:REE}
The relative entropy of entanglement (REE) is an important measure for entanglement for a given quantum state in quantum information theory \cite{Nielsen:2010aa}. \sh{Obtaining accurate bounds on the REE is a major problem in quantum information theory.} In \cite{Zinchenko:2010aa} the following problem is used to find a lower bound to the REE:
\begin{align*}
 & \min_{x} g(x)=\tr[\rho(\log(\rho)-\log(x))]\\
&\text{s.t.: } x\in \mathbb{H}^{n}_{+},\tr(x)=1,\; \scrP(x)\in\mathbb{H}^{n}_{+}, 
 \end{align*}
 where $\mathbb{H}^{n}_{+}$ is the cone of $n\times n$ Hermitian positive semi-definite matrices, and $\scrP(x)$ is the so-called partial transpose operator, which is a linear symmetric operator (see \cite{Horodecki:1996aa}). $\rho\in\mathbb{H}^{n}_{+}$ is a given quantum reference state. We obtain a problem of the form \eqref{eq:Opt}, upon the identification $\setE=\mathbb{H}^{n}$ the space of $n\times n$ Hermitian matrices, and $\setK=\mathbb{H}^{n}_{+}$, and $\setX=\{x\in\setE\vert x\in\mathbb{H}^{n}_{+},\tr(x)=1\}$. 
\end{example}
The model template \eqref{eq:Opt} can be solved with moderate accuracy using fast proximal gradient methods, like FISTA and related acceleration tricks \cite{FOMSurvey}. 
In this formalism, the constraints imposed on the optimization problem have to be included in proximal operators which essentially means that one can calculate a projection onto the feasible set. Unfortunately, the computation of proximal operators can impose a significant burden and even lead to intractability of gradient-based methods. This is in particular the case in SDPs with a large number of constraints and decision variables, such as in convex relaxations of NP-hard combinatorial optimization problems \cite{LauRen05}. 
Other prominent examples in machine learning are $k$-means clustering \cite{PenWei07}, $k$-nearest neighbor classification \cite{Weinberger:2009to}, sparse PCA \cite{dAspremont:2008wa}, kernel learning \cite{Lanckriet:2004aa}, among many others. 
In many of these mentioned applications, it is not only the size of the problem that challenges the application of off-the-shelve solvers, but also the desire to obtain sparse solutions. 
As a result, the \ac{CndG} method \cite{FraWol56} has seen renewed interest in large-scale convex programming, since it \sh{only requires computing a \ac{LMO} at} each iteration. Compared to proximal methods, \ac{CndG} features significantly reduced computational costs (e.g. when $\setX$ is a spectrahedron), tractability and interpretability (e.g. it generates solutions as a combination of atoms of $\setX)$. 

\paragraph{Our approach}
In this paper we propose a \ac{CndG} framework for solving \eqref{eq:Opt} with rigorous iteration complexity guarantees. 
Our approach retains the simplicity of \ac{CndG}, but allows to disentangle the different sources of complexity describing the problem. 
Specifically, we study the performance of a new \emph{homotopy/path-following method}, which iteratively approaches a solution of problem \eqref{eq:Opt}. 
To develop the path-following argument we start by decomposing the feasible set into two components. 
We assume that $\setX$ is a compact convex set which admits efficient applications of \ac{CndG}. 
The challenging constraints are embodied by the membership condition $\scrP(x)\in\setK$. 
These are conic restrictions on the solution, and in typical applications of our method we have to manage a lot of them. 
Our approach deals with these constraints via a \emph{logarithmically homogeneous barrier} for the cone $\setK$.
\cite{NesNem94} introduced this class of functions in connection with polynomial-time interior point methods. 
It is a well-known fact that any closed convex cone admits a logarithmically homogeneous barrier (the universal barrier) \cite[Thm. 2.5.1]{NesNem94}.  
A central assumption of our approach is that we have a practical representation of a logarithmically homogenous barrier for the cone $\setK$.
As many restrictions appearing in applications are linear or PSD inequalities, this assumption is rather mild. Exploiting the announced decomposition, we approximate the target problem \eqref{eq:Opt} by a family of parametric penalty-based composite convex optimization problems formulated in terms of the \emph{potential function}
\begin{equation}\label{eq:potential}
\min_{x\in\setX}\{V_{t}(x):=\frac{1}{t}F(x)+g(x)\}.
\end{equation}
This potential function involves a path-following/homotopy parameter $t>0$, and $F$ is a barrier function defined over the set $\scrP^{-1}(\setK)$. 
By minimizing the function $V_{t}$ over the set $\setX$ for a sequence of increasing values of $t$, we can trace the \emph{analytic central path} $z^{\ast}(t)$ of \eqref{eq:potential} as it converges to a solution of \eqref{eq:Opt}. The practical implementation of this path-following scheme is achieved via a double-loop algorithm in which the inner loop performs a number of \ac{CndG} updating steps on the potential function $V_{t}$, and the outer-loop readjusts the parameter $t$ as well as the desired tolerance imposed on the \ac{CndG}-solver. Importantly, unlike existing primal-dual methods for \eqref{eq:Opt} (e.g. \cite{YurTroFerUdeCev21}), our approach generates a sequence of feasible points for the original problem \eqref{eq:Opt}. 

\paragraph{Comparison to the literature}
Exact homotopy/path-following methods were developed in statistics and signal processing in the context of the $\ell_{1}$-regularized least squares problem (LASSO) \cite{TibTay11} to compute the entire regularization path. 
Relatedly, approximate homotopy methods have been studied in this context as well \cite{HalYinZha08,WenYinGol10}, and superior empirical performance has been reported for seeking sparse solutions. 
\cite{XiaoTong2013} is the first reference which investigates the complexity of a proximal gradient homotopy method for the LASSO problem. To our best knowledge, our paper provides the first complexity analysis of a homotopy \ac{CndG} algorithm to compute an approximate solution close to the analytic central path for the optimization template \eqref{eq:Opt}. Our analysis is heavily influenced by recent papers \cite{ZhaFre20} and \cite{Zhao:2023aa}. They study a generalized \ac{CndG} algorithm designed for solving a composite convex optimization problem, in which the smooth part is given by a \emph{logarithmically homogeneous barrier}. Unlike \cite{ZhaFre20}, we analyze a more complicated dynamic problem in which the objective changes as we update the penalty parameter, and we propose important practical extensions including line-search and inexact-\ac{LMO} versions of the \ac{CndG} algorithm for reducing the potential function \eqref{eq:potential}. Furthermore, we analyze the complexity of the whole path-following process in order to approximate a solution to the initial non-penalized problem \eqref{eq:Opt}.

\paragraph{Main Results} 
One of the challenging questions is how to design a penalty parameter updating schedule, and we develop a practical answer to this question. Specifically, our contributions can be summarized as follows: 
\begin{itemize}
\item We introduce a simple \ac{CndG} framework for solving problem \eqref{eq:Opt} and prove that it achieves a $\tilde{O}(\eps^{-2})$ iteration complexity, where $\tilde{O}(\cdot)$ hides polylogarithmic factors (Theorem \ref{th:Complexity}).
\item We provide two extensions for the inner-loop \ac{CndG} algorithm solving the potential minimization problem \eqref{eq:potential}: a line-search and an inexact-\ac{LMO} version. For both, we show that the complexity of the inner and the outer loop are the same as in the basic variant up to constant factors.
\item We present practically relevant applications of our framework, including instances of SDPs arising from convex relaxations of NP-hard combinatorial optimization problems, and provide promising results of the numerical experiments. 
\end{itemize}
Direct application of existing \ac{CndG} frameworks \cite{YurFerLocCev18,YurTroFerUdeCev21} for solving \eqref{eq:Opt} would require projection onto $\setK$, which may be complicated if, e.g., $\setK$ is a PSD cone. Moreover, these algorithms do not produce feasible iterates but rather approximate solutions with primal optimality and dual feasibility guarantees. For specific instances of SDPs, the theoretical complexity achieved by our method matches or even improves the state-of-the-art iteration complexity results reported in the literature. For packing and covering types of SDPs, the state-of-the-art upper complexity bound reported in \cite{ElbMakNaj22} is worse than ours since their algorithm has complexity $\tilde{O}(C_{1}\eps^{-2.5}+C_{2}\eps^{-2})$, where $C_{1},C_{2}$ are constants depending on the problem data\footnote{The notation $\tilde{O}$ hides polylogarithmic factors}. 
For SDPs with linear equality constraints and spectrahedron constraints, the primal-dual method CGAL \cite{YurTroFerUdeCev21} produces \emph{approximately optimal} and  \emph{approximately feasible} points in $O(\eps^{-2})$ iterations. Our method is purely primal, generating feasible points anytime. At the same time, our method shares the same theoretical iteration complexity as CGAL and can deal with additional conic constraints embodied by the membership restriction $\scrP(x)\in\setK$ \sh{without use of projection}. We expect that this added feature will allow us to apply our method to handle challenging scientific computing questions connected to the quantum information theory, such as bounding the relative entropy of entanglement of quantum states (cf. Example \ref{ex:REE} and \cite{Zinchenko:2010aa,Fawzi:2019aa,Faybusovich:2020aa}). 

\paragraph{Organization of the paper}
This paper is organized as follows: Section \ref{sec:prelims} introduces the notation and terminology we use throughout this paper. Section \ref{sec:algo} presents the basic algorithm under consideration in this paper. In section \ref{sec:extensions} we describe an inexact version of our basic homotopy method. \sh{Section \ref{sec:complexityLMO} contains a detailed complexity analysis of our scheme when an exact \ac{LMO} is available. The corresponding statements under our inexact \ac{LMO} are performed in Appendix \ref{app:Inexact}. Section \ref{sec:numerics}} reports the performance of our method in solving relevant examples of SDPs and gives some first comparisons with existing approaches. In particular, numerical results on the mixing time problem of a finite Markov Chain, \sh{and synthetic examples aimed at testing the robustness to scaling} are reported there. Appendix \ref{sec:exp} contains further results on the numerical experiments. 

\section{Notation and Preliminaries}
\label{sec:prelims}
%
Let $\setE$ be a finite-dimensional Euclidean vector space, $\setE^{\ast}$ its dual space, which is formed by all linear functions on $\setE$. 
The value of function $s\in\setE^{\ast}$ at $x\in\setE$ is denoted by $\inner{s,x}$. 
Let $B:\setE\to\setE^{\ast}$ be a positive definite self-adjoint operator. Define the primal and dual norms 
$$
\norm{h}_{B}=\inner{Bh,h}^{1/2},\text{ and }\norm{s}_{B}^{\ast}=\inner{s,B^{-1}s}^{1/2}
$$
for $h\in\setE$ and $s\in\setE^{\ast}$, respectively. \pdv{If $B$ is the identity operator, we simplify notation to $\norm{h}$ and $\norm{s}_{\ast}$ respectively for the primal and the dual norm.} For a linear operator $A:\setE\to\setE^{\ast}$ we define the operator norm
$$
\norm{A}=\max_{h\in\setE:\norm{h}\leq 1}\norm{Ah}_{\ast}.
$$
For a differentiable function $f(x)$ with $\dom(f)\subseteq\setE$, we denote by $f'(x)\in\setE^{\ast}$ its gradient, by $f''(x):\setE\to\setE^{\ast}$ its Hessian, and by $f'''(x)$ its third derivative.  \pdv{Accordingly, given directions $h_1,h_2,h_3 \in \setE$ the first, second, and third directional derivatives of $f$ are denoted respectively as $f'(x)[h_1]$,$f''(x)[h_1,h_2]$,$f'''(x)[h_1,h_2,h_3]$. }

\subsection{Self-Concordant Functions}
\label{sec:SCB}
Let $\setQ$ be an open and convex subset of $\setE$. A function $f:\setE\to(-\infty,\infty]$ is \emph{\ac{SC}} on $\setQ$ if $f\in\bC^{3}(\setQ)$ and for all $x\in\setQ,h\in\setE$, we have 
$\abs{f'''(x)[h,h,h]}\leq 2(f''(x)[h,h])^{3/2}.$ In case where $\cl(\setQ)$ is a \pdv{closed, convex, and pointed (i.e., contains no line) cone},  we have more structure. We call $f$ a \emph{$\nu$-canonical barrier} for $\setQ$, denoted by $f\in\scrB_{\nu}(\setQ)$, if it is \ac{SC} and 
\[
\forall (x,t)\in\setQ\times(0,\infty):\quad f(tx)=f(x)-\nu\log(t).
\]
From \citep[][Prop. 2.3.4]{NesNem94}, the following properties are satisfied by a $\nu$-canonical barrier for each $x\in\setQ,h\in\setE,t>0$:
\begin{align}
&f'(tx)=t^{-1}f'(x),\\
&f'(x)[h]=-f''(x)[x,h],\\
&f'(x)[x]=-\nu,\\
&f''(x)[x,x]=\nu.
\end{align}
We define the local norm $\norm{h}_{f''(x)}:=(f''(x)[h,h])^{1/2}$ for all $x\in\dom(f)$ and $h\in\setE$. 
\ac{SC} functions are in general not Lipschitz smooth. Still, we have access to a version of a descent lemma of the following form \cite[][Thm. 5.1.9]{Nes18}:
\begin{equation}\label{eq:descent}
f(x+h)\leq f(x)+f'(x)[h]+\omega_{\ast}(\norm{h}_{f''(x)})
\end{equation}
for all $x\in\dom(f)$ and $h\in\setE$ such that $\norm{h}_{f''(x)}<1$, where $\omega_{\ast}(t)=-t-\log(1-t)$. It also holds true that \cite[][Thm. 5.1.8]{Nes18}
\begin{equation}\label{eq:SC-convex}
f(x+h)\geq f(x)+f'(x)[h]+\omega(\norm{h}_{f''(x)})
\end{equation}
for all $x\in\dom(f)$ and $h\in\setE$ such that $x+h\in \dom(f)$, where $\omega(t)=t-\log(1+t)$ for $t\geq 0$.

A classical and useful bound is \cite[][Lemma 5.1.5]{Nes18}:
\begin{equation}\label{eq:boundomega}
\frac{t^{2}}{2-t}\leq\omega_{\ast}(t)\leq\frac{t^{2}}{2(1-t)}\quad\forall t\in[0,1).
\end{equation}

\subsection{The Optimization Problem}
\label{sec:opt}
The following assumptions are made for the rest of this paper.
\begin{assumption}\label{ass:g}
The function $g:\setE\to(-\infty,\infty]$ is closed and convex and continuous on $\dom(g)$. The quantity  $\Omega_{g}:=\max_{x,y\in\dom(g)\cap\setX}\abs{g(x)-g(y)}$ is finite. 
\end{assumption}
\begin{assumption}\label{ass:barrier}
$\setK\subset\setH$ is a closed convex \pdv{pointed} cone with $\Int(\setK)\neq\emptyset$, admitting a $\nu$-canonical barrier $f\in\scrB_{\nu}(\setK)$. 
\end{assumption}
The next assumption transports the barrier setup from the codomain $\setK$ to the domain $\scrP^{-1}(\setK)$. This is a common operation in the framework of the "barrier calculus" developed in \cite[][Section 5.1]{NesNem94}.
\begin{assumption}
The map $\scrP:\setE\to\setH$ is linear, and $F(x):=f(\scrP(x))$ is a $\nu$-canonical barrier on the cone $\scrP^{-1}(\setK)$, i.e. $F\in\scrB_{\nu}(\scrP^{-1}(\setK))$. 
\end{assumption}
Note that $\dom(F)=\Int\scrP^{-1}(\setK)$.
\begin{example}
Considering the normalised covering problem presented in Example \ref{ex:covering}. 
\cite{ElbMakNaj22} solve this problem via a logarithmic potential function method \cite{NesNem94}, involving the logarithmically homogeneous barrier $f(\bX)=\log\det(\bX)$ for $\bX\in\symm^{n}_{+}$.  
This is a typical choice in Newton-type methods to impose the semi-definiteness constraint. However, in our projection-free environment, we use the power of the linear minimisation oracle to obtain search directions which leaves the cone of positive semidefinite matrices invariant. Instead, we employ barrier functions to incorporate the additional linear constraints in \eqref{eq:Opt}. Hence, we set $F(x)=f(\scrP(x))=\log\det(\sum_{i=1}^{m}x_{i}\bA_{i}-\bI)$ to absorb the constraint $\scrP(x)\in\symm^{n}_{+}$. In particular, this frees us from matrix inversions, and related computationally intensive steps coming with Newton and interior-point methods.
\end{example}
\begin{assumption}\label{ass:slater}
$\setX$ is a nonempty compact convex set in $\setE$, and $\setC:=\setX\cap\dom(F)\cap\dom(g)\neq\emptyset$.
\end{assumption}
Let $\Opt:=\min\{g(x)\vert x\in\setX,\scrP(x)\in\setK\}$.  Thanks to assumptions \ref{ass:barrier} and \ref{ass:slater}, $\Opt$ is attained. Our goal is to find an $\eps$-solution of problem \eqref{eq:Opt}, defined  as follows. 
\begin{definition}\label{def:eps}
Given a tolerance $\eps>0$, we say that $z^{\ast}_{\eps}$ is an $\eps$-solution for \eqref{eq:Opt} if \sh{$z^{\ast}_{\eps}$ is a \emph{feasible} solution of problem \eqref{eq:Opt} with objective at most $\eps$ higher than $\Opt$, that is,}
$$
z^{\ast}_{\eps}\in\scrP^{-1}(\setK)\cap\setX\text{ and } g(z^{\ast}_{\eps})-\Opt\leq \eps.
$$
\end{definition} 

Given $t>0$, define 
\begin{equation}\label{eq:Optt}
\Opt(t):=\min_{x\in\setX}\{V_{t}(x)=\frac{1}{t}F(x)+g(x)\},\text{ and }z^{\ast}(t)\in\{x\in\setX\vert V_{t}(x)=\Opt(t)\}.
\end{equation}
The following Lemma shows that the path $t\mapsto z^{\ast}(t)$ traces a trajectory in the interior of the feasible set which can be used to approximate a solution of the original problem \eqref{eq:Opt}, provided the penalty parameter $t$ is chosen large enough. 
\begin{lemma}\label{lem:pathfollowing}
For all $t>0$, it holds that $z^{\ast}(t)\in\setC$. In particular, 
\begin{equation}
g(z^{\ast}(t))-\Opt\leq \frac{\nu}{t}.
\end{equation}
\end{lemma}
\begin{proof}
Since $F\in\scrB_{\nu}(\scrP^{-1}(\setK))$, it follows immediately that $z^{\ast}(t)\in\dom(F)\cap\setX$. By Fermat's principle we have 
$$
0\in \frac{1}{t}F'(z^{\ast}(t))+\partial g(z^{\ast}(t))+\NC_{\setX}(z^{\ast}(t)).
$$
Convexity and \cite[][Thm. 5.3.7]{Nes18} implies that for all $z\in\dom(F)\cap\setX$, we have
$$
g(z)\geq g(z^{\ast}(t))-\frac{1}{t}F'(z^{\ast}(t))[z-z^{\ast}(t)]\geq g(z^{\ast}(t))-\frac{\nu}{t}.
$$
Let $z^{\ast}$ be a feasible point point satisfying $g(z^{\ast})=\Opt$. Choosing a sequence $\{z^{j}\}_{j\in\N}\subset\dom(F)\cap\setX$ with $z^{j}\to z^{\ast}$, the continuity on $\dom(g)$ gives 
$$
\Opt\geq g(z^{\ast}(t))-\frac{\nu}{t}.
$$
\end{proof}

\section{Algorithm}
\label{sec:algo}
%
In this section we first describe a new \ac{CndG} method for solving general conic constrained convex optimization problems of the form \eqref{eq:Optt}, for a fixed $t>0$. This procedure serves as the inner loop of our homotopy method. The path-following strategy in the outer loop is then explained in Section \ref{sec:path-following}.

\subsection{The Proposed \ac{CndG} Solver}
The computational efficiency of our approach relies on the availability of a suitably defined minimisation oracle: 
\begin{assumption}\label{ass:LMO}
For any $c\in\setE^{\ast}$, the auxiliary problem
\begin{equation}\label{eq:LMO}
\scrL_{g}(c):=\argmin_{s\in\setX}\{\inner{c,s}+g(s)\}
\end{equation}
is easily solvable. 
\end{assumption}

We can compute the gradient and Hessian of $F$ as 
\[
F'(x)=f'(\scrP(x))\scrP'(x),\quad F''(x)[s,t]=f''(\scrP(x))[\scrP(s),\scrP(t)]\quad \forall x\in\Int\scrP^{-1}(\setK).
\]
This means that we obtain a local norm on $\setH$ given by 
\[
\norm{w}_{x}\pdv{:=}F''(x)[w,w]^{1/2}\qquad \forall (x,w)\in\Int\scrP^{-1}(\setK)\times\setH.
\]
Note that in order to evaluate the local norm we do not need to compute the full Hessian $F''(x)$. It only requires a directional derivative, which is potentially easy to do numerically.\\ 

A \acf{LMO} is defined as an oracle producing the vector field 
\begin{equation}\label{eq:s}
s_{t}(x)\in\scrL_{g}(t^{-1}F'(x)). 
\end{equation}
Note that our analysis does not rely on a specific tie-breaking rule, so any member of the set $\scrL_{g}(t^{-1}F'(x))$ will be acceptable. 

To measure solution accuracy and overall algorithmic progress, we introduce two \emph{merit functions}: 
\begin{align}
\gap_{t}(x)&:=t^{-1}F'(x)[x-s_{t}(x)]+g(x)-g(s_{t}(x)),\text{ and }\\
\Delta_{t}(x)&:=V_{t}(x)-V_{t}(z^{\ast}(t)).
\end{align} 
Note that $\gap_{t}(x)\geq 0$ and $\Delta_{t}(x)\geq 0$ for all $x\in\dom(F)$. Moreover, convexity together with the definition of the point $z^{\ast}(t)$, gives 
\begin{align*}
0&\leq \Delta_{t}(x)=t^{-1}[F(x)-F(z^{\ast}(t))]+g(x)-g(z^{\ast}(t))\\
&\leq t^{-1}F'(x)[x-z^{\ast}(t)]+g(x)-g(z^{\ast}(t))\\
&\leq t^{-1}F'(x)[x-s_{t}(x)]+g(x)-g(s_{t}(x))=\gap_{t}(x).
\end{align*}
Hence, 
\begin{equation}\label{eq:relationGap}
\gap_{t}(x)\geq \Delta_{t}(x)\qquad\forall x\in\dom(F)\cap\dom(g). 
\end{equation}

Define 
\begin{equation}
\ce_{t}(x):=\norm{s_{t}(x)-x}_{x}.
\end{equation}
Then, for \ms{$\alpha\in[0,\min\{1,1/\ce_{t}(x)\})$}, we get from \eqref{eq:descent}
$$
F(x+\alpha(s_{t}(x)-x))\leq F(x)+\alpha F'(x)[s_{t}(x)-x]+\omega_{\ast}(\alpha\ce_{t}(x)).
$$
Together with the convexity of $g$, this implies 
\begin{equation}\label{eq:upperV}
V_{t}(x+\alpha(s_{t}(x)-x))\leq V_{t}(x)-\alpha\gap_{t}(x)+t^{-1}\omega_{\ast}(\alpha\ce_{t}(x))=\ms{V_{t}(x)-\eta_{t,x}(\alpha)},
\end{equation}
\ms{where 
$$
\eta_{t,x}(\alpha):=\alpha\gap_{t}(x)-t^{-1}\omega_{\ast}(\alpha\ce_{t}(x))=\alpha\gap_{t}(x)+t^{-1}\left(\alpha\ce_{t}(x)+\log(1-\alpha\ce_{t}(x))\right).
$$
Optimising this expression w.r.t $\alpha\in[0,1]$, we obtain the analytic step-size policy} 
\begin{equation}\label{eq:alpha}
\alpha_{t}(x):=\min\left\{1,\frac{t\gap_{t}(x)}{\ce_{t}(x)(\ce_{t}(x)+t\gap_{t}(x))}\right\}\in[0,1/\ce_{t}(x)).
\end{equation}
Equipped with this step strategy, procedure $\CG(x^{0},\eps,t)$, described in Algorithm \ref{alg:CndG}, constructs a sequence $\{x^{k}_{t}\}_{k\geq 0}$ which produces an approximately-optimal solution in terms of the merit function $\gap_{t}(\cdot)$ and the potential function gap $\Delta_{t}$. Specifically, our main results on the performance of Algorithm \ref{alg:CndG} are the following iteration complexity estimates, which are proven in Section \ref{sec:complexityLMO}.
\begin{algorithm}[t]
 \caption{$\CG(x^{0},\eps,t)$}
 \label{alg:CndG}
 \begin{algorithmic}
\State {\bfseries Input: } $(x^{0},t)\in\setC\times(0,\infty)$ initial state; $\eps>0$ accuracy level
\For{$k=0,1,\ldots$}
 \If{$\gap_{t}(x^{k})>\eps$}
  \State Obtain $s^{k}=s_{t}(x^{k})$ defined in \eqref{eq:s}.
  \State $\alpha_{k}=\alpha_{t}(x^{k})$ defined in \eqref{eq:alpha}.
  \State Set $x^{k+1}=x^{k}+\alpha_{k}(s^{k}-x^{k})$.
 \Else
  \State \ms{STOP. Return iteration index $k$ and last iterate $x^{k}$.}
\EndIf
\EndFor
\end{algorithmic}
\end{algorithm}

\begin{proposition}\label{prop:gap}
Given $\eta,t>0$, Algorithm $\CG(x^{0}_{t},\eta,t)$ requires at most
\begin{equation}\label{eq:R}
R(x^{0}_{t},\eta,t):= \lceil 5.3(\nu+t\Delta_{t}(x^{0}_{t})+t\Omega_{g})\log(10.6t\Delta_{t}(x^{0}_{t}))\rceil+\lceil \frac{24}{t\eta}(\nu+t\Omega_{g})^{2}\rceil,
\end{equation}
iterations in order to reach a point $x^{k}_{t}$ satisfying  $\gap_{t}(x^{k}_{t})\leq\eta$ for the first time.
\end{proposition}
\begin{proposition}\label{prop:Delta}
Given $\eta,t>0$. Algorithm $\CG(x^{0}_{t},\eta,t)$ requires at most 
\begin{equation}\label{eq:N}
N(x^{0}_{t},\eta,t):=\lceil 5.3(\nu+t\Delta_{t}(x^{0}_{t})+t\Omega_{g})\log(10.6t\Delta_{t}(x^{0}_{t}))\rceil+\lceil 12(\nu+t\Omega_{g})^{2}\left(\frac{1}{t\eta}-\frac{1}{t\Delta_{t}(x^{0}_{t})}\right)_{+}\rceil
\end{equation}
iterations, in order to reach a point $x^{k}_{t}$ satisfying $\Delta_{t}(x^{k}_{t})\leq \eta$ for the first time.
\end{proposition}

\subsection{Updating the Homotopy Parameter}
\label{sec:path-following}
Our analysis so far focuses on minimisation of the potential function $V_{t}(x)$ for a fixed $t>0$. However, in order to solve the initial problem \eqref{eq:Opt}, one must trace the sequence of approximate solutions $\{z^{\ast}(t)\}_{t\geq 0}$ as $t\uparrow \infty$. The construction of such an increasing sequence of homotopy parameters is the purpose of this section. 
 
\begin{algorithm}[t]
 \caption{Method $\texttt{Homotopy}(x^{0},\eps,f)$ }
 \label{alg:Homotopy}
 \begin{algorithmic}
\State {\bfseries Input: } $x^{0}\in\setC$, $f\in\scrB_{\nu}(\setK)$, $t_{0},\eta_{0}>0$.
\State {\bfseries Parameters: } $\eps,\sigma\in(0,1)$.
\State {\bfseries Initialize: } $x^{0}_{t_{0}}=x^{0}$, $I=\lceil\frac{\log(2\eta_{0}/\eps)}{\log(1/\sigma)}\rceil$.
\For{$i=0,1,\ldots,I-1$}
\State Set $\hat{x}_{i}$ 
 the output of $\CG(x^{0}_{t_{i}},\eta_{i},t_{i})$. 
\State Update $t_{i+1}=t_{i}/\sigma,\eta_{i+1}=\sigma\eta_{i}$, $x^{0}_{t_{i+1}}=\hat{x}_{i}$.
\EndFor
\State Set $\hat{z}=\hat{x}_{I}$ the output of $\CG(x^{0}_{t_{I}},\eta_{I},t_{I})$.
\end{algorithmic}
\end{algorithm}
 
Let $\{\eta_{i}\}_{i\geq 0},\{t_{i}\}_{i\geq 0}$ be a sequence of approximation errors and homotopy parameters. For each run $i$, we activate procedure $\CG(x^{0}_{t_{i}},\eta_{i},t_{i})$ with the given configuration $(\eta_{i},t_{i})$. For $i=0$, we assume to have an admissible initial point $x^{0}_{t_{0}}=x^{0}$ available. 
For $i\geq 1$, we restart Algorithm \ref{alg:CndG} using a kind of warm-start strategy by choosing $x^{0}_{t_{i}}=x^{R_{i-1}}_{t_{i-1}}$, where \ms{$R_{i-1}\leq R(x^{0}_{t_{i-1}},\eta_{i-1},t_{i-1})$ denotes the first iterate $k$ of Algorithm $\CG(x^{0}_{t_{i-1}},\eta_{i-1},t_{i-1})$ satisfying $\gap_{t_{i-1}}(x^{k}_{t_{i-1}})\leq\eta_{i-1}$. Note that $R_{i-1}$ is upper bounded by the constant $R(x^{0}_{t_{i-1}},\eta_{i-1},t_{i-1})$ defined in Proposition \ref{prop:gap}.}
After the $i$-th restart, let us call the obtained iterate $\hat{x}_{i}:= x^{R_{i}}_{t_{i}}$. In this way we obtain a sequence $\{\hat{x}_{i}\}_{i\geq 0}$, consisting of candidates for approximately following the central path, as they are $\eta_{i}$-close in terms of the gap $\gap_{t_{i}}$ (and, hence, in terms of the function gap) to the stage $t_{i}$'s optimal point $z^{\ast}(t_i)$. We update the parameters $(\eta_{i},t_{i})$ as follows:
\begin{itemize}
\item The sequence of homotopy parameters is determined as $t_{i}=t_{0}\sigma^{-i}$ for $\sigma\in(0,1)$ until the last round of $\texttt{Homotopy}(x^{0},\eps,f)$ is reached.
\item The sequence of accuracies requested in the algorithm $\CG(x^{0}_{t_{i}},\eta_{i},t_{i})$ is updated by $\eta_{i}=\eta_{0}\sigma^{i}$. 
\item The Algorithm stops after $I\equiv I(\sigma,\eta_{0},\eps)=\lceil\frac{\log(2\eta_{0}/\eps)}{\log(1/\sigma)}\rceil$ updates of the accuracy and homotopy parameters and yields an $\eps$-approximate solution of problem \eqref{eq:Opt}.
\end{itemize}

Our updating strategy for the parameters ensures ''detailed balance'' $t_{i}\eta_{i}=t_{0}\eta_{0}$. This equilibrating choice between the increasing homotopy parameter and the decreasing accuracy parameter is sensible, because the iteration complexity of the \ac{CndG} solver is inversely proportional to $t_{i}\eta_{i}$ (cf. Proposition \ref{prop:gap}). Making the judicious choice $\eta_{0}t_{0}=2\nu$, yields a compact assessment of the total complexity of method $\texttt{Homotopy}(x^{0},\eps,f)$. 
\begin{theorem}\label{th:Complexity}
Choose $t_{0}=\frac{\nu}{\Omega_{g}}$ and $\eta_{0}=2\Omega_{g}$. The total iteration complexity of method $\texttt{Homotopy}(x^{0},\eps,f)$ to find an $\eps$-solution of \eqref{eq:Opt} is 
\begin{equation}
\texttt{Compl}(x^{0},\eps,f)=\tilde{\bigoh}\left(\frac{384\Omega^{2}_{g}\nu}{\eps^{2}(1-\sigma^{2})}+\Omega_{g}\log(21.2\nu(1+(2/\eps)\Omega_{g}))\frac{21.2\nu}{\eps}\frac{2-\sigma}{1-\sigma}\right),
\end{equation}
where $\tilde{\bigoh}$ hides polylogarithmic factors in $\eps^{-1}$. 
\end{theorem}
\begin{remark}
The theorem leaves open the choice of the parameter $\sigma$. We treat $\sigma$ as a hyperparameter that should be calibrated by the user. The proof of Theorem \ref{th:Complexity} can be found in Section \ref{sec:complexityLMO}.
\end{remark}

\section{Modifications and Extensions}
\label{sec:extensions}
%

\subsection{Line Search}
The step size policy employed in Algorithm \ref{alg:CndG} is inversely proportional to the local norm $\ce_{t}(x)$. In particular, its derivation is based on the a-priori restriction that we force our iterates to stay within a trust-region defined by the Dikin ellipsoid. This restriction may force the method to take very small steps, and as a result display bad performance in practice. A simple remedy would be to employ a line search. Given $(x,t)\in\dom(F)\times(0,+\infty)$, let 
\begin{equation}\label{eq:ExactLS}
\gamma_{t}(x)=\argmin_{\gamma\in[0,1]\text{ s.t. } x+\gamma(s_{t}(x)-x) \in \dom(F)}V_{t}(x+\gamma(s_{t}(x)-x)).
\end{equation}
\begin{algorithm}[t]
 \caption{\ac{CndG} with line search: $\LCG(x^{0},\eps,t)$}
 \label{alg:LineSearchCndG}
 \begin{algorithmic}
\State {\bfseries Input: } $(x^{0},t)\in\setC\times(0,\infty)$ initial state; $\eps>0$ accuracy level
\For{$k=0,1,\ldots$}
 \If{$\gap_{t}(x^{k})>\eps$}
\State Same as Algorithm \ref{alg:CndG}, but with step size strategy \eqref{eq:ExactLS}.
\Else
  \State \ms{STOP. Return iteration index $k$ and last iterate $x^{k}$.}
\EndIf
\EndFor
\end{algorithmic}
\end{algorithm}
Thanks to the barrier structure of the potential function $V_{t}$, some useful consequences can be drawn from the definition of $\gamma_{t}(x)$. First, $\gamma_{t}(x)\in\{\gamma\geq 0\vert x+\gamma(s_{t}(x)-x)\in\dom(F)\}$. This implies $x+\gamma_{t}(x)(s_{t}(x)-x)\in\setX\cap\dom(F)\cap\dom(g)$. Second, since $\alpha_{t}(x)$ is also contained in the latter set, we have 
\[
V_{t}(x+\gamma_{t}(x)(s_{t}(x)-x))\leq V_{t}(x+\alpha_{t}(x)(s_{t}(x)-x))\qquad\forall (x,t)\in\dom(F)\times(0,+\infty).
\]
Via a comparison principle, this allows us to deduce the analysis of the sequence produced by $\LCG(x^{0},\eps,t)$ from the analysis of the sequence induced by $\CG(x^{0},\eps,t)$.
Indeed, if $\{\xi^{k}_{t}\}_{k \geq 0}$ is the sequence constructed by the line search procedure $\LCG(x^{0},\eta,t)$, then we have 
$V_{t}(\xi^{k+1}_{t})\leq V_{t}(\xi_{t}^{k}+\alpha_{t}(\xi^{k}_{t})(s_{t}(\xi^{k}_{t})-x^{k}_{t}))$ for all $k$. 
Hence, we can perform the complexity analysis on the majorising function as in the analysis of procedure $\CG(x^{0},\eps,t)$. 
Consequently, all the complexity-related estimates for the method $\CG(x^{0},\eta,t)$ apply verbatim to the sequence $\{\xi^{k}_{t}\}_{k\geq 0}$ induced by the method $\LCG(x^{0},\eta,t)$. 


\subsection{Algorithm with inexact \ac{LMO}}
A key assumption in our approach so far is that we can obtain an exact answer from the \ac{LMO} \eqref{eq:s}. In this section we consider the case when our search direction subproblem is solved only with some pre-defined accuracy, a setting that is of particular interest in optimization problems defined over matrix variables. As an illustration, let us consider instances of \eqref{eq:Opt} where $\setX$ is the spectrahedron of symmetric matrices, namely $\setX=\{\bX\in\R^{n\times n}\vert\bX\succeq 0,\tr(\bX)=1\}$. For these instances solving the subproblem \eqref{eq:LMO} with 
\sh{$g=\tr(\bC\bX)$} corresponds to computing the leading eigenvector of a symmetric matrix, which when $n$ is very large is typically computed inexactly using iterative methods. We therefore relax our definition of the \ac{LMO} by allowing for an \ac{ILMO}, featuring additive and multiplicative errors. To define our \ac{ILMO}, let 
$$
\Gamma_{t}(x,s):=\frac{1}{t}F'(x)[x-s]+g(x)-g(s)\quad t>0,(x,s)\in\setC\times\setC.
$$
By definition of the point $s_{t}(x)$ in \eqref{eq:s} and of the gap function $\gap_{t}(x)$, we have $\Gamma_{t}(x,s)\leq \gap_{t}(x)$ for all $s\in\setC$. 
\begin{definition}\label{def:ILMO}
Given $t>0,\delta\in(0,1]$, and $\theta>0$, we call \sh{an oracle producing} a point $\tilde{s}_{t}(x)\in\setC$ a $(\delta,\theta)$-\ac{ILMO}, 
\sh{such that}
\begin{equation}\label{eq:InexactLMO}
\widetilde{\gap}_{t}(x)\geq\delta(\gap_{t}(x)-\theta), \text{where }\widetilde{\gap}_{t}(x)\equiv \Gamma_{t}(x,\tilde{s}_{t}(x)).
\end{equation}
We write $\tilde{s}_{t}(x)\in\scrL^{(\delta,\theta)}_{g}(t^{-1}\nabla F(x))$ for any direction $\tilde{s}_{t}(x)$ satisfying \eqref{eq:InexactLMO}.
\end{definition}
We note that a $(0,0)$-\ac{ILMO} returns a search point satisfying inclusion \eqref{eq:s}. We therefore refer to the case where we can design a $(0,0)$-\ac{ILMO} as the \emph{exact oracle case.} Moreover, a $(1,\theta)$-\ac{ILMO} is an oracle that only features additive errors. This relaxation for inexact computations has been considered in previous work, such as \cite{Dunn:1978aa,Jag13,Freund:2016aa}. The combination of multiplicative error $\delta$ and additive error $\theta$ displayed in \eqref{eq:InexactLMO} is similar to the oracle constructed in \cite{pmlr-v28-lacoste-julien13}. The above definition can also be interpreted as follows. The answer $s_t(x)$ of the exact \ac{LMO} solves the maximization problem $\max_{s\in \setX} \{\Gamma_{t}(x,s)=t^{-1}F'(x)[x-s]+g(x)-g(s)\}$ and the optimal value in this problem is nonnegative. Thus, \eqref{eq:InexactLMO} can be interpreted as the requirement that the point $\tilde{s}_{t}(x)$ solves the same problem up to the additive error $\theta$ and the multiplicative error $\delta$ (i.e. because it is itself the output of a convex optimisation routine). We emphasize that our analysis does not depend on the specific choice of the point $\tilde{s}_{t}(x)$.\\

Equipped with these concepts, we can derive an analytic step-size policy as in the exact regime. Let $x\in\setC$ be given, $\tilde{s}_{t}(x)\in\scrL^{(\delta,\theta)}_{g}(t^{-1}\nabla F(x))$, and $\alpha\in(0,1)$ small enough so that $\alpha\norm{\tilde{s}_{t}(x)-x}_{x}\equiv\alpha\tilde{\ce}_{t}(x)\in(0,1)$. Just like in \eqref{eq:upperV}, we can then obtain the upper bound 
$$
V_{t}(x+\alpha(\tilde{s}_{t}(x)-x))\leq V_{t}(x)-\alpha \widetilde{\gap}_{t}(x)+\frac{1}{t}\omega_{\ast}(\alpha\tilde{\ce}_{t}(x)).
$$
Optimizing the upper bound with respect to $\alpha$ gives us the analytic step-size criterion 
\begin{equation}\label{eq:inexactstep}
\tilde{\alpha}_{t}(x):=\min\left\{1,\frac{t\widetilde{\gap}_{t}(x)}{\tilde{\ce}_{t}(x)(\tilde{\ce}_{t}(x)+t\widetilde{\gap}_{t}(x))}\right\}.
\end{equation}
The $(\delta,\theta)$-\ac{ILMO}, and the step size policy \eqref{eq:inexactstep} are enough to define a \ac{CndG}-update with inexact computations that are executed by a function $P(x,t,\delta,\theta)$, defined in Algorithm \ref{alg:functionP}.
\begin{algorithm}[h!]
 \begin{algorithmic}
 \caption{Function $P(x,t,\delta,\theta)$}
 \label{alg:functionP}
 \State Obtain $\tilde{s}_{t}(x)\in\scrL^{(\delta,\theta)}_{g}(t^{-1}\nabla F(x))$;
 \State Obtain $\tilde{\alpha}_{t}(x)$ defined in \eqref{eq:inexactstep};
 \State Update $x\leftarrow x+\tilde{\alpha}_{t}(x)(\tilde{s}_{t}(x)-x)$.
\end{algorithmic}
\end{algorithm}

%
 To use this method within a bona fide algorithm, we would need to introduce some stopping criterion. Ideally, such a stopping criterion should aim at making the merit functions $\gap_{t}$ or $\Delta_{t}$ small. However, our inexact oracle does not allow us to make direct inferences on the values of these merit functions. Hence, we need to come up with a stopping rule that uses only quantities that are observable and that are coupled with these merit functions. The natural optimality measure for the inexact oracle case is thus $\widetilde{\gap}_{t}(x)$. Indeed, by the very definition of our inexact oracle, if, for accuracy $\eta>0$, we have arrived at an iterate $x$ for which $\widetilde{\gap}_{t}(x)\leq\delta\eta$, we know from \eqref{eq:InexactLMO} that $\gap_{t}(x)\leq\eta+\theta$. Therefore, by making $\theta$ a function of $\eta$, we would attain essentially the same accuracy as the exact implementation requires when running Algorithm \ref{alg:CndG}. This is the main idea leading to the development of Algorithm \ref{alg:ICG}.

\begin{algorithm}[t]
 \caption{$\texttt{ICG}(x,t,\eta,\delta,\theta)$}
 \label{alg:ICG}
 \begin{algorithmic}
\State {\bfseries Input: } $x\in\setC$, $\delta\in(0,1],\theta>0,t>0,\eta>0$.
\sh{\For{$k=0,1,\ldots$}
 \If{$\widetilde{\gap}_{t}(x^{k})>\eta\delta$}
\State  $x^{k+1}=P(x^{k},t,\delta,\theta)$
\Else
\State   STOP. Return iteration index $k$ and last iterate $x^k$.
\EndIf
\EndFor
}
\end{algorithmic}
\end{algorithm}
There are two types of iteration complexity guarantees that we can derive for $\texttt{ICG}(x,t,\eta,\delta,\theta)$. One is concerned with upper bounding the number of iterations after which our stopping criterion will be satisfied. The other is an upper bound on the number of iterations needed to make the potential function gap $\Delta_{t}(x^{k})$ smaller than $\eta+\theta$. The derivation of these estimates is rather technical and follows closely the complexity analysis of the exact regime. We therefore only present the statements here and invite the reader to consult the Appendix \ref{app:Inexact} for a detailed proof.
\begin{proposition}\label{prop:GapInexact}
 Assume that $\eta,\theta, t>0$, $\delta\in(0,1]$, and $x^{0}\in\setC$. Algorithm $\ICG(x^{0},t,\eta,\delta,\theta)$ requires at most
\begin{align}
\tilde{R}_{t}(x^{0},\eta,\delta,\theta)&:= \left\lceil \frac{5.3(t\delta(\Delta_{t}(x^{0} )-\theta)+\nu+t\Omega_{g})}{\delta}\log \left(\min\left\{\frac{10.6 t \Delta_{t}(x^{0} )}{(1-10.6 t \theta)_+}, \frac{\Delta_{t}(x^{0})}{\eta} \right\}\right)  \right\rceil \notag \\
&+ \left\lceil \frac{12(\nu+t\Omega_{g})^{2}}{t\delta^{2}\eta} \left(2+\frac{\theta}{\eta}\right)\right\rceil+1
\end{align}
iterations in order to reach a point $x^{k}$ satisfying $\widetilde{\gap}_{t}(x^{k})\leq\eta\delta$ for the first time.
\end{proposition}
\begin{proposition}\label{prop:DeltaInexact}
Assume that $\eta,\theta, t>0$, $\delta\in(0,1]$, and $x^{0}\in\setC$. Algorithm $\ICG(x^{0},t,\eta,\delta,\theta)$ requires at most
\begin{align}
\tilde{N}(x^{0},\eta,\delta,\theta)&:=\left\lceil \frac{5.3(t\delta(\Delta_{t}(x^{0} )-\theta)+\nu+t\Omega_{g})}{\delta}\log\left( \min\left\{\frac{10.6 t \Delta_{t}(x^{0})}{(1-10.6 t \theta)_+}, \frac{\Delta_{t}(x^{0} )}{\eta} \right\}\right)  \right\rceil \notag\\
& \hspace{2em}+\left\lceil \frac{12(\nu+t\Omega_{g})^{2}}{t\delta^{2}}\left(\frac{1}{\eta}-\frac{1}{\Delta_{t}(x^{0})-\theta}\right)_{+}\right\rceil +1 \label{eq:prop:DeltaInexact}
\end{align}
iterations, in order to reach a point $x^{k}$ satisfying $\Delta_{t}(x^{k})\leq\eta+\theta$ for the first time.
\end{proposition}

\paragraph{Updating the Homotopy Parameter in the Inexact Case}
\label{sec:inexact-path-following}
We are now in a position to describe the outer loop of our algorithm with \ac{ILMO}. The construction is very similar to Algorithm \ref{alg:Homotopy}.
\begin{algorithm}[t]
 \caption{$\texttt{Inexact Homotopy}(x^{0},\eps,f,\delta)$ }
 \label{alg:I-Homotopy}
 \begin{algorithmic}
\State {\bfseries Input: } $x^{0}\in\setC$, $f\in\scrB_{\nu}(\setK)$, $\delta \in(0,1]$.
\State {\bfseries Parameters: } $\eps,t_{0},\theta_0=\eta_{0}>0$, $\sigma\in(0,1)$.
\State {\bfseries Initialisation: } $x^{0}_{t_{0}}=x^{0}$, $I=\lceil\frac{\log(3\eta_{0}/\eps)}{\log(1/\sigma)}\rceil$.
\For{$i=0,1,\ldots,I-1$}
\State Set $\hat{x}_{i}$ the output of $\ICG(x^{0}_{t_{i}},t_{i},\eta_{i},\delta,\theta_i)$. 
\State Update $t_{i+1}=t_{i}/\sigma,\eta_{i+1}=\sigma\eta_{i},\theta_{i+1}=\sigma\theta_{i+1}$, $x^{0}_{t_{i+1}}=\hat{x}_{i}$.
\EndFor
\State $\hat{z}=\hat{x}_{I}$ the output of $\ICG(x^{0}_{t_{I}},t_{I},\eta_{I},\delta,\theta_{I})$. 
\end{algorithmic}
\end{algorithm}
As before, our aim is to reach an $\eps$-solution (Definition \ref{def:eps}). 
Let $\{t_{i}\}_{i\geq 0},\{\eta_{i}\}_{i\geq 0},\{\theta_{i}\}_{i\geq 0}$ be sequences of homotopy parameters, accuracies in the inner loop, and \ac{ILMO} inexactnesses. For each run $i$, we activate procedure $\ICG(x^{0}_{t_{i}},t_{i},\eta_{i},\delta,\theta_i)$ with the given configuration $(t_{i},\eta_{i},\theta_i)$. For $i=0$, we assume to have an admissible initial point $x^{0}_{t_{0}}=x^{0}$ available. 
For $i\geq 1$, we restart Algorithm \ref{alg:ICG} using a kind of warm-start strategy by choosing $x^{0}_{t_{i}}=x^{R_{i-1}}_{t_{i-1}}$, where $\tilde{R}_{i-1} \leq \tilde{R}_{t_{i-1}}(x^{0}_{t_{i-1}},\eta_{i-1},\delta,\theta_{i-1})$ is the first iterate $k$ of Algorithm $\ICG(x^{0}_{t_{i-1}},t_{i-1},\eta_{i-1},\delta,\theta_{i-1})$ satisfying $\widetilde{\gap}_{t_{i-1}}(x^{k}_{t_{i-1}})\leq\eta_{i-1}\delta$. Note $\tilde{R}_{i-1}$ is upper bounded by the mentioned number, by means of Proposition \ref{prop:GapInexact}. After the $i$-th restart, let us call the obtained iterate $\hat{x}_{i}:= x^{\tilde{R}_{i}}_{t_{i}}$. In this way we generate a sequence $\{\hat{x}_{i}\}_{i\geq 0}$, consisting of candidates for approximately following the central path, as they are $(\eta_{i}+\theta_i)$-close in terms of $\gap_{t_{i}}$ (and, hence, in terms of the potential function gap $\Delta_{t_{i}}$). We update the parameters $(t_{i},\eta_{i},\theta_{i})$ as follows:
\begin{itemize}
\item The sequence of homotopy parameters is determined as $t_{i}=t_{0}\sigma^{-i}$ for $\sigma\in(0,1)$ until the last round of $\texttt{Inexact Homotopy}(x^{0},\eps,f,\delta)$ is reached.
\item The sequences of accuracies in the inner loop and \ac{ILMO} inexactnesses requested in procedure $\ICG(x^{0}_{t_{i}},t_{i},\eta_{i},\delta,\theta_i)$ is updated by $\eta_{i}=\eta_{0}\sigma^{i}$, $\theta_{i}=\theta_{0}\sigma^{i}$. 
\item The Algorithm stops after $I\equiv I(\sigma,\eta_{0},\eps)=\lceil\frac{\log(3\eta_{0}/\eps)}{\log(1/\sigma)}\rceil$ updates of the accuracy and homotopy parameters, and yields an $\eps$-approximate solution of problem \eqref{eq:Opt}.
\end{itemize}
We underline that it is not necessary to stop the algorithm after a prescribed number of iterations and it is essentially an any-time convergent algorithm.

Our updating strategy for the parameters ensures that $t_{i}\eta_{i}=t_{i}\theta_{i}=t_{0}\eta_{0}=t_{0}\theta_{0}$.
We will make the judicious choice $\eta_{0}t_{0}=2\nu$  to obtain an overall assessment of the iteration complexity of method $\texttt{Inexact Homotopy}(x^{0},\eps,f,\delta)$. The proof is given in Appendix \ref{app:Inexact}.
\begin{theorem}\label{th:Complexity_inexact}
Choose $t_{0}=\frac{\nu}{\Omega_{g}}$ and $\theta_0=\eta_{0}=2\Omega_{g}$. The total iteration complexity, i.e. the number of \ac{CndG} steps, of method $\texttt{Inexact Homotopy}(x^{0},\eps,f,\delta)$ to find an $\eps$-solution of \eqref{eq:Opt} is 
$
\texttt{Compl}(x^{0},\eps,f,\delta)= O\left( \frac{\nu \Omega_{g}^2}{\delta^2 \eps^2 (1-\sigma^2)}\right).
$
\end{theorem}

\section{Complexity Analysis}
\label{sec:complexityLMO}
%
In this section, we give a detailed presentation of the analysis of the iteration complexity of Algorithm \ref{alg:Homotopy}.
\subsection{Analysis of procedure $\CG(x^{0},\eps,t)$}
\label{sec:inner}
Recall $\Omega_{g}:=\max_{x,y\in\dom(g)\cap\setX}\abs{g(x)-g(y)}$ and $\setC=\dom(F)\cap\setX\cap\dom(g)$. The following estimate can be established as in \cite[Prop. 3]{ZhaFre20}. We therefore omit its simple proof. 
\begin{lemma}\label{lem:e-bound}
For all $(x,t)\in\setC\times(0,\infty)$, we have
\begin{equation}\label{eq:eOmega}
t^{-1}\ce_{t}(x)\leq \nu/t+\gap_{t}(x)+\Omega_{g}.
\end{equation}
\end{lemma}
Observe that
\begin{align}
\ce_{t}(x)&=\norm{s_{t}(x)-x}_{x}=\sqrt{F''(x)[s_{t}(x)-x,s_{t}(x)-x]}\nonumber\\
&\geq \nu^{-1/2}\abs{F'(x)[s_{t}(x)-x]}\nonumber\\
&=\frac{t}{\sqrt{\nu}}\abs{\gap_{t}(x)-g(x)+g(s_{t}(x))}\nonumber\\
&\geq \frac{t}{\sqrt{\nu}}(\gap_{t}(x)-\Omega_{g}). \label{eq:app_1}
\end{align}

\subsubsection{Proof of Proposition \ref{prop:Delta}}
\label{S:proof_exact_Delta}
The proof mainly follows the steps of the proof of Theorem 1 in \cite{ZhaFre20}. The main technical challenge in our analysis is taking care of the penalty parameter $t$. A main step in the complexity analysis is the identification of two convergence regimes, depending on the magnitude of the merit function $\gap_{t}$. 

\begin{lemma}\label{lem:phases}
Define $\ca_{t}:=\frac{\nu}{t}+\Omega_{g}$, and the partition 
\[
\K_{I}(t)=\{k\in\N\vert \gap_{t}(x_{t}^{k})>\ca_{t}\}, \quad \K_{II}(t)=\N\setminus \K_{I}(t). 
\]
If $k\in\K_{I}(t)$ we have
  \begin{equation}
\Delta_{t}(x^{k}_{t})-\Delta_{t}(x^{k+1}_{t})\geq\frac{1}{5.3}\frac{G^{k}_{t}}{\nu+tG^{k}_{t}+t\Omega_{g}} \geq \frac{\Delta_{t}(x^{k}_{t})}{5.3(\nu+t\Omega_{g}+t\Delta_{t}(x^{k}_{t}))},
\end{equation}
and if $k\in\K_{II}(t)$ we have  
\begin{equation}\label{eq:LBe}
\frac{1}{\Delta_{t}(x^{k+1}_{t})}\geq \frac{1}{\Delta_{t}(x^{k}_{t})}+\frac{t}{12(\nu+t\Omega_{g})^{2}}
\end{equation}
In particular, $\{\Delta_{t}(x^{k}_{t})\}_{k\in\N}$ is monotonically decreasing. 
\end{lemma}
Before proving this Lemma, let us explain how we use it to estimate the total number of iterations needed by method $\CG(x^{0},\eps,t)$ until the stopping criterion is satisfied. Motivated by the two different regimes identified in Lemma \ref{lem:phases}, we can bound the number of iterations produced by $\CG(x^{0}_{t},\eta,t)$ before the stopping criterion $\Delta_{t}(x^{k}_{t})\leq\eta$ applies in two steps: First, we bound the number of iterations which the algorithm spends in the set $\K_{I}(t)$; Second we bound the number of iterations the algorithm spends in $\K_{II}(t)$. Combining these two upper bounds, yields the total iteration bound reported in Proposition \ref{prop:Delta}. 

\begin{proof}[Proof of Lemma \ref{lem:phases}]
Assume that iteration $k$ is in phase $\K_{I}(t)$. To reduce notational clutter, we set $\ce^{k}_{t}\equiv\ce_{t}(x^{k}_{t})$ and $G^{k}_{t}\equiv\gap_{t}(x^{k}_{t})$. Combining the condition $G^{k}_{t} > \frac{\nu}{t}+\Omega_{g}$ of phase $\K_{I}(t)$ with \eqref{eq:app_1}, we deduce that $\ce^{k}_{t}\geq \sqrt{\nu}\geq 1$. This in turn implies $\frac{tG^{k}_{t}}{\ce^{k}_{t}(\ce^{k}_{t}+tG^{k}_{t})}<\frac{tG^{k}_{t}}{\ce^{k}_t+t G^{k}_{t}}<1$. Hence, at this iteration, algorithm $\CG(x^{0},\eps,t)$ chooses the step sizes $\alpha_{k}=\frac{tG^{k}_{t}}{\ce^{k}_{t}(\ce^{k}_{t}+tG^{k}_{t})}$. The per-iteration reduction of the potential function can be estimated as
\begin{align}
V_{t}(x^{k+1}_{t})&\leq V_{t}(x^{k}_{t})-\frac{G^{k}_{t}}{\ce^{k}_{t}}\frac{t G^{k}_{t}}{\ce_{t}^{k}+tG^{k}_{t}}+\frac{1}{t}\omega_{\ast}\left(\frac{t G^{k}_{t}}{\ce_{t}^{k}+tG^{k}_{t}}\right)\notag\\
&=V_{t}(x^{k}_{t})-\frac{G^{k}_{t}}{\ce^{k}_{t}}+\frac{1}{t}\log\left(1+\frac{tG_{t}^{k}}{\ce^{k}_{t}}\right) \notag\\
&=V_{t}(x^{k}_{t})-\frac{1}{t}\omega\left(\frac{tG^{k}_{t}}{\ce^{k}_{t}}\right) \label{eq:V_t_progress_1}
\end{align}
where $\omega(\tau)=\tau-\log(1+\tau)$ and $\omega_{*}(t)=-\tau-\log(1-\tau)$. This readily yields $\Delta_{t}(x^{k}_{t})-\Delta_{t}(x^{k+1}_{t})\geq \frac{1}{t}\omega\left(\frac{tG^{k}_{t}}{\ce^{k}_{t}}\right)$ and implies that $\{\Delta_{t}(x^{k}_{t})\}_{k\in\K_{I}(t)}$ is monotonically decreasing. For $k\in\K_{I}(t)$, we also see $\frac{tG^{k}_{t}}{\nu+tG^{k}_{t}+t\Omega_{g}}>\frac{1}{2}$. Together with \eqref{eq:eOmega}, this implies 
\begin{equation}\label{eq:ratio1/2}
\frac{tG_{t}^{k}}{\ce^{k}_{t}}\geq\frac{G^{k}_{t}}{\nu/t+G_{t}^{k}+\Omega_{g}}=\frac{tG^{k}_{t}}{\nu+tG_{t}^{k}+t\Omega_{g}}>\frac{1}{2}.
\end{equation}
Using \eqref{eq:relationGap}, the monotonicity of $\omega(\cdot)$ and that $\omega(\tau) \geq \frac{\tau}{5.3}$ for $\tau\geq 1/2$ (see \cite[Prop. 1]{ZhaFre20}), the fact that the function $x\mapsto \frac{x}{a+tx}$ is strictly increasing for $t,a>0$, we arrive at 
\begin{equation}\label{eq:step_1_decrease_intial}
\Delta_{t}(x^{k}_{t})-\Delta_{t}(x^{k+1}_{t})\geq\frac{1}{5.3}\frac{G^{k}_{t}}{\nu+tG^{k}_{t}+t\Omega_{g}} \geq \frac{\Delta_{t}(x^{k}_{t})}{5.3(\nu+t\Omega_{g}+t\Delta_{t}(x^{k}_{t}))}.
\end{equation}
Next, consider an integer $k\in\K_{II}(t)$. We have to distinguish the iterates using large steps $(\alpha_{k}=1$), from those using short steps ($\alpha_{k}<1$). Let us first consider the case where $\alpha_{k}=1$. Then, $tG^{k}_{t}\geq \ce^{k}_{t}(tG^{k}_{t}+\ce^{k}_{t})$, which implies $\ce^{k}_{t}\in(0,1)$, and $G^{k}_{t}\geq\frac{(\ce^{k}_{t})^{2}}{t(1-\ce^{k}_{t})}$. Moreover, using \eqref{eq:boundomega}, we readily obtain
\begin{align*}
V_{t}(x^{k+1}_{t})&\leq V_{t}(x^{k}_{t})-G^{k}_{t}+\frac{1}{t}\omega_{\ast}(\ce^{k}_{t})\\
&\leq V_{t}(x^{k}_{t})-G^{k}_{t}+\frac{1}{t}\frac{(\ce^{k}_{t})^{2}}{2(1-\ce^{k}_{t})}\\
&\leq V_{t}(x^{k}_{t})-\frac{G_{t}^{k}}{2}.
\end{align*}
Therefore, 
$$
\Delta_{t}(x^{k}_{t})-\Delta_{t}(x^{k+1}_{t})\geq \frac{G_{t}^{k}}{2}\geq \frac{1}{2}\frac{(G_{t}^{k})^{2}}{\nu/t+\Omega_{g}}=\frac{1}{2}\frac{t(G^{k}_{t})^{2}}{\nu+t\Omega_{g}}\geq \frac{1}{12}\frac{t(G^{k}_{t})^{2}}{(\nu+t\Omega_{g})^{2}}.
$$
In the regime where $\alpha_{k}<1$, we obtain the estimate
$\Delta_{t}(x^{k}_{t})-\Delta_{t}(x^{k+1}_{t})\geq \frac{1}{12}\frac{t(G^{k}_{t})^{2}}{(\nu+t\Omega_{g})^{2}}$, by following the same reasoning as in \cite{ZhaFre20}. Using the relation $G^{k}_{t}\geq\Delta_{t}(x^{k}_{t})\geq\Delta_{t}(x^{k+1}_{t})$, we conclude 
$$
\frac{1}{\Delta_{t}(x^{k+1}_{t})}-\frac{1}{\Delta_{t}(x^{k}_{t})}\geq\frac{t}{12(\nu+t\Omega_{g})^{2}}.
$$
\end{proof}

We now exploit this monotonicity of the potential function gap in order to upper bound the number of iterations which Algorithm $\CG(x^{0}_{t},\eta,t)$ spends in the disjoint phases $\K_{I}(t)$ and $\K_{II}(t)$, respectively.

\begin{proof}[Proof of Proposition \ref{prop:Delta}]
Given the pair $(t,\eta)\in(0,\infty)^{2}$, denote by $N\equiv N(x^{0}_{t},\eta,t)$ the upper bound defined in \eqref{eq:N}. We separate $N$ into two parts, corresponding with stages $\K_{I}(t)$ and $\K_{II}(t)$, as they are defined in the proof of Lemma~\ref{lem:phases}.
We will show that $N=K_t(x^0_t)+M$ where $K_t(x^0_t)$ is a bound on the number of iterates within $\K_I(t)$ and $M$ is bound on the number of iterates within $\K_{II}(t)$ until $\Delta_t(x^k_t)\leq\eta$ is reached.

Note that Lemma~\ref{lem:phases} ensures that for all $k\in \N=\K_{I}(t)\cup \K_{II}(t)$, it holds that $\Delta_t(x^{k+1}_t)\leq \Delta_t(x^{k}_t)$, i.e. the sequence $\Delta_t(x^{k}_t)$ is monotonically decreasing.

We begin by deriving an expression for $K_{t}(x^{0}_{t})$, the number of iterations which $\CG(x^{0}_{t},\eta,t)$ spends in $\K_{I}(t)$. To that end, let $\{k_{i}\}_{i\in\N}$ denote the increasing set of indices enumerating the elements of the set $\K_{I}(t)$. From \eqref{eq:step_1_decrease_intial} and  $\Delta_{t}(x^{k_{i+1}}_{t})\leq \Delta_{t}(x^{k_i+1}_{t})\leq \Delta_{t}(x^{k_i}_{t})\leq\Delta_{t}(x^{0}_{t}$), we conclude 
\begin{equation}\label{eq:step_1_decrease}
\Delta_{t}(x^{k_i}_{t})\left[1-\frac{1}{5.3(\nu+t\Omega_{g}+t\Delta_{t}(x^{0}_{t}))}\right]\geq\Delta_{t}(x^{k_i+1}_{t})\geq \Delta_{t}(x^{k_{i+1}}_{t}).
\end{equation}
Setting $q_{t}:=\frac{1}{5.3(\nu+t\Omega_{g}+t\Delta_{t}(x^{0}_{t}))}\in(0,1)$ (since $\nu\geq 1$), we arrive at the recursion 
$$
\Delta_{t}(x^{k_{i}}_{t})\leq(1-q_{t})^{i}\Delta_{t}(x^{0}_{t})\leq\exp(-q_{t}i)\Delta_{t}(x^{0}_{t}).
$$
Furthermore, by definition, all $k_i\in \K_I(t)$ satisfy $G_t^{k_i}\geq \ca_{t}= \frac{\nu}{t}+\Omega_g$, so that \eqref{eq:ratio1/2} yields
\begin{align*}
\Delta_{t}(x^{k_i}_{t})&\geq \Delta_{t}(x^{k_i}_{t})-\Delta_{t}(x^{k_i+1}_{t}) \geq\frac{G_{t}^{k_i}}{5.3(\nu+tG^{k_i}_{t}+t\Omega_{g})}\geq \frac{1}{10.6 t}.
\end{align*}
Hence, $\frac{1}{10.6 t\Delta_{t}(x^{0}_{t})}\leq \exp(-q_{t}\abs{\K_{I}(t)})$, and  
\begin{equation}\label{eq:ub_first_stage}
K_t(x^0_t):= \lceil 5.3(\nu+t\Delta_{t}(x^{0}_{t})+t\Omega_{g})\log(10.6t\Delta_{t}(x^{0}_{t}))\rceil
\end{equation}
yields the upper bound. 

We proceed analogously with the analysis on the set $\K_{II}(t)$. Let $\{\ell_{j}\}_{j\in\N_{0}}$ the increasing set of indices enumerating the elements of the set $\K_{II}(t)$, and $M=N-K_{t}(x^{0}_{t})$ the upper bound on iterations in this set. Using inequality \eqref{eq:LBe} for $k:=\ell_j$ together with the monotonicity of $\Delta_t(x^k_t)$ results in
\begin{equation}\label{eq:inverse_new}
\frac{1}{\Delta_{t}(x^{\ell_{j+1}}_{t})}\geq \frac{1}{\Delta_{t}(x^{\ell_j+1}_{t})}\geq\frac{1}{\Delta_{t}(x^{\ell_j}_{t})}+\frac{t}{12(\nu+t\Omega_{g})^{2}}.
\end{equation}
 Telescoping the inequality \eqref{eq:inverse_new} together with 
 $\Delta_{t}(x^{\ell_1}_{t})\leq \Delta_{t}(x^{0}_{t})$ we get 
$$
\frac{1}{\Delta_{t}(x^{\ell_{M}}_{t})}\geq \frac{1}{\Delta_{t}(x^{\ell_0}_{t})}+\frac{Mt}{12(\nu+t\Omega_{g})^{2}}\geq\frac{1}{\Delta_{t}(x^{0}_{t})}+\frac{Mt}{12(\nu+t\Omega_{g})^{2}}.
$$
Thus, $\Delta_{t}(x^{\ell_{M}}_{t})\leq\eta$, is satisfied for $M=\lceil 12(\nu+t\Omega_{g})^{2}(\frac{1}{t\eta}-\frac{1}{t\Delta_{t}(x^{0}_{t})})_{+}\rceil$. This demonstrates that at most $N=K_{t}(x^{0}_{t})+M$ iterations of $\CG(x^{0}_{t},\eta,t)$ are sufficient for achieving $\Delta_{t}(x_{t}^{k})\leq\eta.$
\end{proof}

\subsubsection{Proof of Proposition \ref{prop:gap}}
\sh{Using the bound $K_t(x^0_t)$ on the number of iterations in $\K_I(t)$ from the proof of Proposition~\ref{prop:Delta}, we are left to bound the number of iteration in $\K_{II}(t)$, until the condition $G_t^k\leq \eta$ is satisfied.}

Let $\{k_{j}(t)\}_{j}$ denote the increasing set of indices enumerating the elements of the set $\K_{II}(t)$. From \eqref{eq:LBe} we know that 
\[
\Delta_{t}(x_{t}^{k_{j+1}(t)})-\Delta_{t}(x^{k_{j}(t)}_{t})\leq-\frac{1}{12}\frac{t(G_{t}^{k_{j}(t)})^{2}}{(\nu+t\Omega_{g})^{2}}, \text{ and }G_{t}^{k_{j}(t)}\geq\Delta_{t}(x^{k_{j}(t)}_{t})\geq \Delta_{t}(x^{k_{j}(t)}_{t}).
\]
Set $d_{j}\equiv \Delta_{t}(x^{k_{j}(t)}_{t})$ and $\frac{1}{M}\equiv \frac{t}{12(\nu+t\Omega_{g})^{2}},\Gamma_{j}\equiv G_{t}^{k_{j}(t)}$. We thus obtain the recursion 
\[
d_{j+1}\leq d_{j}-\frac{\Gamma_{j}^2}{M},\quad\Gamma_{j}\geq d_{j},
\]
and we can apply \cite[Prop. 2.4]{ZhaFre20} directly to the above recursion to obtain the estimates 
\[
d_{j}\leq \frac{M}{j},\quad \text{ and } \min\{\Gamma_{0},\ldots,\Gamma_{j}\}\leq \frac{2M}{j}.
\]
Since the algorithm terminates when we obtain an iterate with $\Gamma_{j} = \gap_{t}(x^{k_{j}(t)}_{t})\leq\eta$, it is sufficient to stop at the label $j$ is reached satisfying $\frac{2M}{j}\leq\eta$. Solving this for $j$ yields 
\begin{equation}\label{eq:ub_second_stage_iterations}j\leq \left\lceil \frac{24(\nu+t\Omega_{g})^{2}}{t\eta}\right\rceil.\end{equation}
Combining the bound
\eqref{eq:ub_second_stage_iterations} with $K_{t}(x^{0}_{t})$, 
we obtain the total complexity estimate postulated in Proposition 
\ref{prop:gap}.

\subsection{Analysis of the outer loop and proof of Theorem \ref{th:Complexity}}
\label{sec:outer}
Let $I\equiv I(\eta_{0},\sigma,\eps)=\lceil\frac{\log(2\eta_{0}/\eps)}{\log(1/\sigma)}\rceil$ denote the a-priori fixed number of updates of the accuracy and homotopy parameter. We set $\hat{x}_{i}\equiv x^{R_{i}}_{t_{i}}$, the last iterate of procedure $\CG(x^{0}_{t_{i}},\eta_{i},t_{i})$ satisfying $\Delta_{t_{i}}(\hat{x}_{i})\leq\gap_{t_{i}}(\hat{x}_{i})\leq\eta_{i}$. From this, using the definition of the gap function, we deduce that 
\begin{equation}\label{eq:g2}
g(\hat{x}_{i})-g(z^{\ast}(t_{i})) \leq \gap_{t_{i}}(\hat{x}_{i})+\frac{1}{t_{i}}F'(\hat{x}_{i})[z^{\ast}(t_{i})-\hat{x}_{i}]\leq \eta_{i}+\frac{\nu}{t_{i}}.
\end{equation}
Hence, using Lemma \ref{lem:pathfollowing}, we observe
\begin{equation}
\label{eq:final_error}
g(\hat{x}_{i})-\Opt=g(\hat{x}_{i})-g(z^{\ast}(t_{i}))+g(z^{\ast}(t_{i}))-\Opt\leq \eta_{i}+\frac{2\nu}{t_{i}}.
\end{equation}
Since $t_{i}=\frac{2\nu}{\eta_{i}}$, we obtain $g(\hat{x}_{i})-\Opt\leq  2\eta_{i}$. We next estimate the initial gap $\Delta_{t_{i+1}}(x^{0}_{t_{i+1}})=\Delta_{t_{i+1}}(\hat{x}_{i})$ incurred by our warm-start strategy. Observe that 
\begin{align*}
	V_{t_{i+1}}(x^{0}_{t_{i+1}})-V_{t_{i+1}}(z^{\ast}(t_{i+1}))&=V_{t_{i}}(x^{0}_{t_{i+1}})-V_{t_{i}}(z^{\ast}(t_{i+1}))\\
	&+\left(\frac{1}{t_{i+1}}-\frac{1}{t_{i}}\right)(F(x^{0}_{t_{i+1}})-F(z^{\ast}(t_{i+1})))\\
	&\leq V_{t_{i}}(x^{R_{i}}_{t_{i}})-V_{t_{i}}(z^{\ast}(t_{i}))\\
	&+\left(1-\frac{t_{i+1}}{t_{i}}\right)\left(V_{t_{i+1}}(x^{0}_{t_{i+1}})-V_{t_{i+1}}(z^{\ast}(t_{i+1}))\right)\\
	&+(1-\frac{t_{i+1}}{t_{i}})(g(z^{\ast}(t_{i+1}))-g(x^{0}_{t_{i+1}})),
\end{align*}
\sh{where the last transitions follows from $V_{t_{i}}(z^{\ast}(t_{i+1}))\geq V_{t_{i}}(z^{\ast}(t_{i}))$.}
Moreover,
$$
g(x^0_{t_{i+1}})-g(z^{\ast}(t_{i+1}))\leq g(x^0_{t_{i+1}})-\Opt\leq \eta_i+\frac{2\nu}{t_i}=\frac{4\nu}{t_{i}}.
$$
Since $t_{i+1}>t_{i}$, the above two inequalities imply
\begin{align*}
\frac{t_{i+1}}{t_{i}}[V_{t_{i+1}}(x^{0}_{t_{i+1}})-V_{t_{i+1}}(z^{\ast}(t_{i+1}))]\leq V_{t_{i}}(x^{R_{i}}_{t_{i}})-V_{t_{i}}(z^{\ast}(t_{i}))+\frac{t_{i+1}-{t_i}}{t_{i}^2}4\nu.
\end{align*}
Whence, 
\begin{equation}
	\label{eq:tDeltat_estimate_new}
	t_{i+1}\Delta_{t_{i+1}}(x^{0}_{t_{i+1}})\leq t_{i}\Delta_{t_{i}}(x^{R_{i}}_{t_{i}})+\frac{t_{i+1}-{t_i}}{t_{i}}4\nu\leq t_{i}\eta_{i}+\frac{t_{i+1}-{t_i}}{t_{i}}4\nu=\left(2\frac{t_{i+1}}{t_i}-1\right)2\nu.
\end{equation}

We are now in the position to estimate the total iteration complexity $\texttt{Compl}(x^{0},\eps,f):=\sum_{i=0}^{I}R_i \leq \sum_{i=0}^{I}R(x^{0}_{t_{i}},\eta_{i},t_{i})$ by using Proposition \ref{prop:gap}, and thereby prove Theorem \ref{th:Complexity}. We do so by using the updating regime of the sequences $\{t_{i}\}$ and $\{\eta_{i}\}$ explained in Section \ref{sec:algo}. These updating mechanisms imply 
\begin{align*}
	\sum_{i=1}^{I}R_{i}&\leq \sum_{i=1}^{I}5.3\left(\nu+2\nu\left(\frac{2t_i}{t_{i-1}}-1\right)+t_{i}\Omega_{g}\right)\log\left(10.6\left(2\nu\left(\frac{2t_i}{t_{i-1}}-1\right)\right)\right)\\
	&+\sum_{i=1}^{I}\frac{12}{\nu}(\nu+t_{i}\Omega_{g})^{2}\\
	&\leq \log(21.2(\nu(2/\sigma-1)))\sum_{i=1}^{I}5.3(\nu(4/\sigma-1)+t_{i}\Omega_{g}))+\sum_{i=1}^{I}\frac{24}{\nu}(\nu^{2}+t^{2}_{i}\Omega^{2}_{g})\\
	&\leq {24}{\nu}\times I+21.2\nu/\sigma\log(42.4\nu/\sigma)\textcolor{black}{\times I}+\textcolor{black}{5.3}\log(42.4\nu/\sigma)\Omega_{g}\sum_{i=1}^I t_{i}+\frac{24\Omega_{g}^{2}}{\nu}\sum_{i=1}^{I}t_{i}^{2}\\
	&\leq \textcolor{black}{\frac{24\nu\log(2\eta_{0}/\eps)}{\log(1/\sigma)}}+\log\left(\frac{42.4\nu}{\sigma}\right)\frac{21.2\nu}{\sigma}\textcolor{black}{\frac{\log(2\eta_{0}/\eps)}{\log(1/\sigma)}}\\
	&+\log\left(\frac{42.4\nu}{\sigma}\right)\frac{\textcolor{black}{21.2}\nu\Omega_{g}}{\eps(1-\sigma)}+\frac{384\Omega_{g}^{2}\nu}{\eps^{2}(1-\sigma^{2})}.
\end{align*}

In deriving these bounds, we have used the \sh{definition of 
 $I$ and the } following estimates: 
\begin{align*}
\sum_{i=1}^{I}t_{i}=t_{0}\sigma^{-1}\sum_{i=0}^{I-1}(1/\sigma)^{i}\leq t_{0}\frac{\sigma^{-I}}{1-\sigma}\leq \frac{4\nu}{\eps(1-\sigma)},
\end{align*}
and
\begin{align*}
\sum_{i=1}^{I}t_{i}^{2}=t^{2}_{0}\sum_{i=1}^{I}(1/\sigma)^{2i}=\frac{t^{2}_{0}}{\sigma^{2}}\sum_{i=0}^{I-1}(1/\sigma^{2})^{i}\leq \frac{16\nu^{2}}{\eps^{2}(1-\sigma^{2})}.
\end{align*}
It remains to bound the complexity at $i=0$. We have 
\begin{align*}
R_{0}&\leq 5.3(\nu+t_{0}\Delta_{t_{0}}(x^{0}_{t_{0}})+t_{0}\Omega_{g})\log(10.6t_{0}\Delta_{t_{0}}(x^{0}))+\frac{12}{\nu}(\nu+t_{0}\Omega_{g})^{2}\\
&\leq 5.3(\nu+F(x^{0})-F(z^{\ast}(t_{0}))+2t_{0}\Omega_{g})\log\left(10.6(F(x^{0})-F(z^{\ast}(t_{0}))+t_{0}\Omega_{g})\right)\\
&+\frac{24}{\nu}(\nu^{2}+t^{2}_{0}\Omega^{2}_{g}).
\end{align*}
Since $t_{0}=\frac{\nu}{\Omega_{g}}$ and $\eta_{0}=2\Omega_{g}$, we have 
$$
R_{0}\leq 5.3(3\nu+F(x^{0})-F(z^{\ast}(t_{0})))\log\left(10.6(\nu+F(x^{0})-F(z^{\ast}(t_{0})))\right)+48\nu.
$$
Adding the two gives the total complexity bound
\begin{align*}
\texttt{Compl}(x^{0},\eps,f)&\leq 5.3(3\nu+F(x^{0})-F(z^{\ast}(t_{0})))\log\left(10.6(\nu+F(x^{0})-F(z^{\ast}(t_{0})))\right)+48\nu\\
&+\frac{\nu\log(\sh{4\Omega_g}/\eps)}{\log(1/\sigma)}\left[24+\log\left(\frac{42.4\nu}{\sigma}\right)\frac{\sh{21.2}}
{\sigma}\right]\\
&+\log\left(\frac{42.4\nu}{\sigma}\right)\frac{21.2\nu\Omega_{g}}{\eps(1-\sigma)}+\frac{384\Omega_{g}^{2}\nu}{\eps^{2}(1-\sigma^{2})}.\\
&=\sh{\tilde O}\left(\frac{\Omega_{g}^{2}\nu}{\eps^{2}(1-\sigma^{2})}\right).
\end{align*}

\section{Experimental results}
\label{sec:numerics}
%

In this section, we give \sh{two} examples covered by the problem template \eqref{eq:Opt}, and report the results of numerical experiments that illustrate the practical performance and scalability of our method.
We implement Algorithm \ref{alg:Homotopy} using procedure $\CG$ and $\LCG$ as solvers for the inner loops. We compare the results of our algorithm to two benchmarks: (1) an interior point solution of the problem by SDPT3 \cite{Toh:1999aa}, and (2) the solution obtained by CGAL \cite{YurTroFerUdeCev21}. \sh{Our performance measure for these experiments is the \emph{relative optimality gap}, which given the value of the objective of the solution computed for each method $M$ and iteration $i$, $g^{M}_i$, is computed by ${|g^M_i-g^{\SDPT3}|}/{|g^{\SDPT3}|}$, where $g^{\SDPT3}$ is the interior point solution.}
The full description of the setup for these experiments is available in Appendix \ref{sec:exp}. 
\sh{A key takeaway from these numerical experiments is that $\CG$ and $\LCG$ are not influenced by the problem's scaling and are therefore well-suited to solving ill-posed problems with differently scaled constraints.}
\subsection{Estimating the mixing time of a Markov chain}
Consider a continuous-time Markov chain on a weighted graph, with edge weights representing the transition rates of the process. Let $\bP(t)=e^{-Lt}$ be the stochastic semi-group of the process, where $L$
	is the graph Laplacian. From $L\1=0$ it follows that $\frac{1}{n}\1$ is a stationary distribution of the process. From \cite{DiaStr91} we have the following bound on the total variation distance of the time $t$ distribution of the Markov process and the stationary distribution
	$
	\sup_{\pi}\norm{\pi \bP(t)-\frac{1}{n}\1}_{\text{TV}}\leq\frac{1}{2}\sqrt{n}e^{-\lambda t}.
	$
	The quantity $T_{\text{mix}}=\lambda$ is called the \emph{mixing time}. The constant $\lambda$ is the second-largest eigenvalue of the graph Laplacian $L$. To bound this eigenvalue, we follow the SDP-strategy laid out in \cite{SunBoyDiacXia06}, and in some detail described in Appendix \ref{sec:exp}.\\
	In the experiments, we generated random connected undirected graphs with $100-800$ nodes and $10^3-10^4$ edges, and for each edge $\{i,j\}$ in the graph, we generated a random weight uniformly in $[0,1]$. 
	Figure \ref{fig:Mixing_GapVsIter_Selected} illustrates the relative optimality gap for the generated results by each method (and the dependence of $\CG$ and $\LCG$ on the starting accuracy $\eta_0$), $\CGAL$ only appears in the right most graphs\footnote{More detailed results are available in Appendix \ref{sec:exp}.}.
    \sh{We conclude that, despite $\LCG$ being generally costly per iteration,} it exhibits the best computational performance both in the number of iterations and in time. \sh{Moreover, for all of the datasets, the primal-dual method $\CGAL$ generated solutions that were far from being feasible, and was not able to reduce the feasibility gap even after 1000 seconds or 10000 iterations, resulting in very poor relative optimality gap.
    Our conjecture is that the slow convergence rate of $\CGAL$ with respect to feasibility is due to the different scaling for the right-hand-side of the constraints, which results in a relative feasibility gap which is very large for constraints with small right-hand-side. To test this conjecture , in the next section we present further experiments testing the performance of the methods as a function of the right-hand-side scaling.\footnote{We thank an anonymous referee for suggesting this experiment.}} 
\begin{figure}[h]%
		\centering
		\vspace{2em}
		\includegraphics[width=0.95\linewidth]{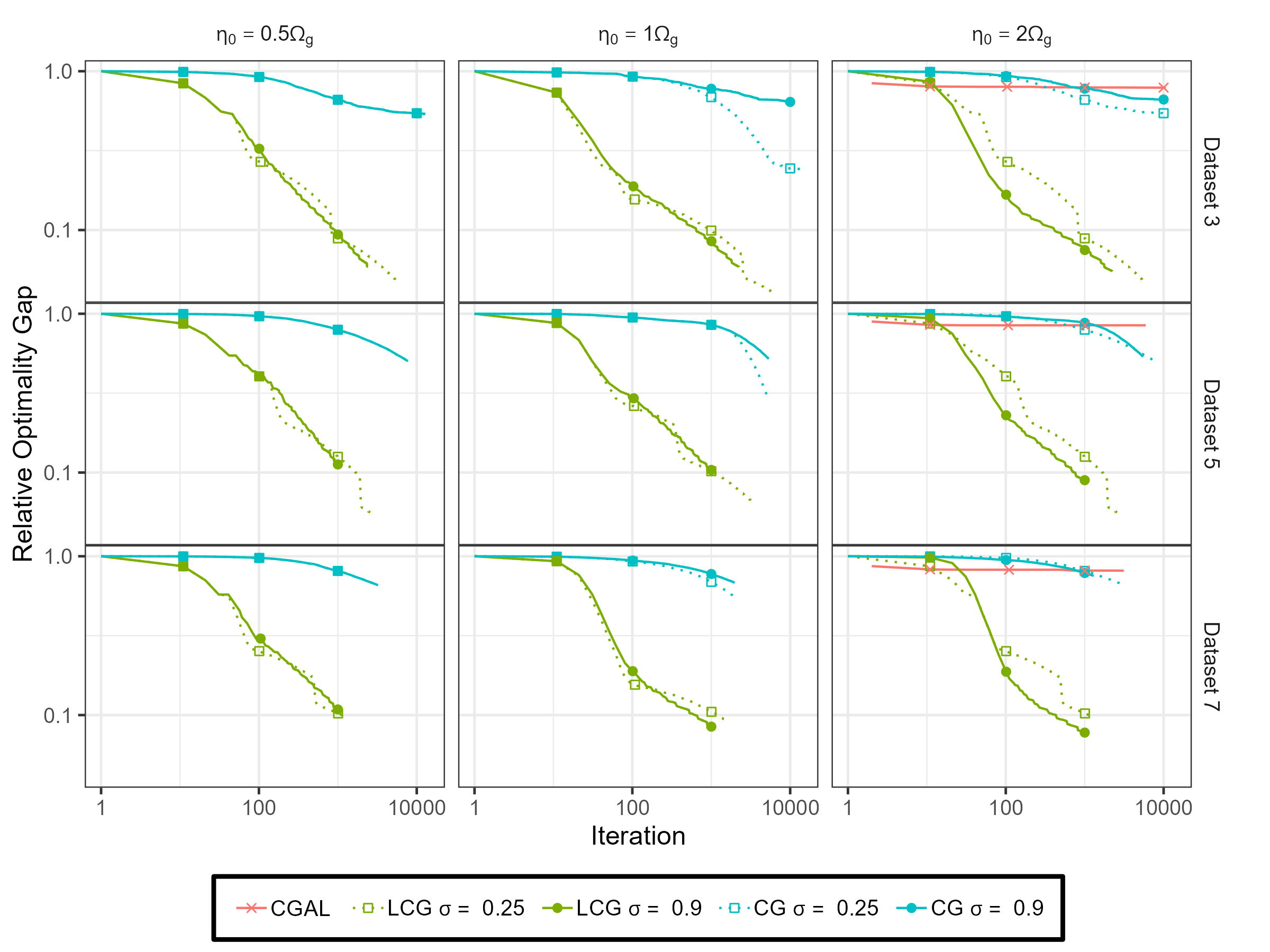}
		\caption{Mixing time problem: relative gap from SDPT3 solution vs. iteration.}%
		\label{fig:Mixing_GapVsIter_Selected}%
\end{figure}

\subsection{Randomly Scaled SDP}\label{sec:SRS_main}
In order to further investigate the robustness of $\CG$, $\LCG$ and $\CGAL$ to constraint scaling, we've created a synthetic data set for a general SDP with inequality constraints.
The Synthetic Random Scaling (SRS) problem is formulated as follows:
\begin{equation}\label{eq:srs}\tag{SRS}
	\max \inner{\bC,\bX}\text{ s.t.: }\inner{\bA_i,\bX}\leq \gi{b_i},i\in\{1,2,\ldots,m\},\bX\in\symm^{n}_{+},\tr(\bX)\leq 1,
\end{equation}
where both the objective function coefficient matrix $\gi{\bC}$ and, the constraint coefficient matrices $\gi{\bA_i}$, for $i = 1, \dots, m$, are randomly generated PSD matrices with unit Frobenius norm.
The coefficients $b_i$ are defined as $b_i = \frac{u_i}{i^{p/2}} \inner{\gi{\bA_i}, \gi{\bX_0}}$, where $u_i \in [\frac{1}{2}, 1]$ is a uniformly distributed random number, $p \in \{0, 1, 2\}$, and $\gi{\bX_0}$ is a randomly generated PSD matrix with unit trace.
Thus, parameter $p$ controls the scaling of the right-hand-side of the constraints. Specifically, when $p=0$ we expect all $b_i$ to have the same order of magnitude, while for $p=2$ constraint $b_m=O(1/m)b_1$.
As we observed for the Mixing problem in Section~\ref{sec:Mixing}, our expectation is for $\CG$ and $\LCG$ to be impervious to this right-hand-side scaling, and for the performance of $\CGAL$ to deteriorate as $p$ increases.
To compare our method with CGAL, the CGAL output was adjusted to obtain a feasible solution.
The details of this procedure are provided in Appendix \ref{sec:SRS}.

\begin{figure}[h!]
\centering
\includegraphics[scale=0.37]{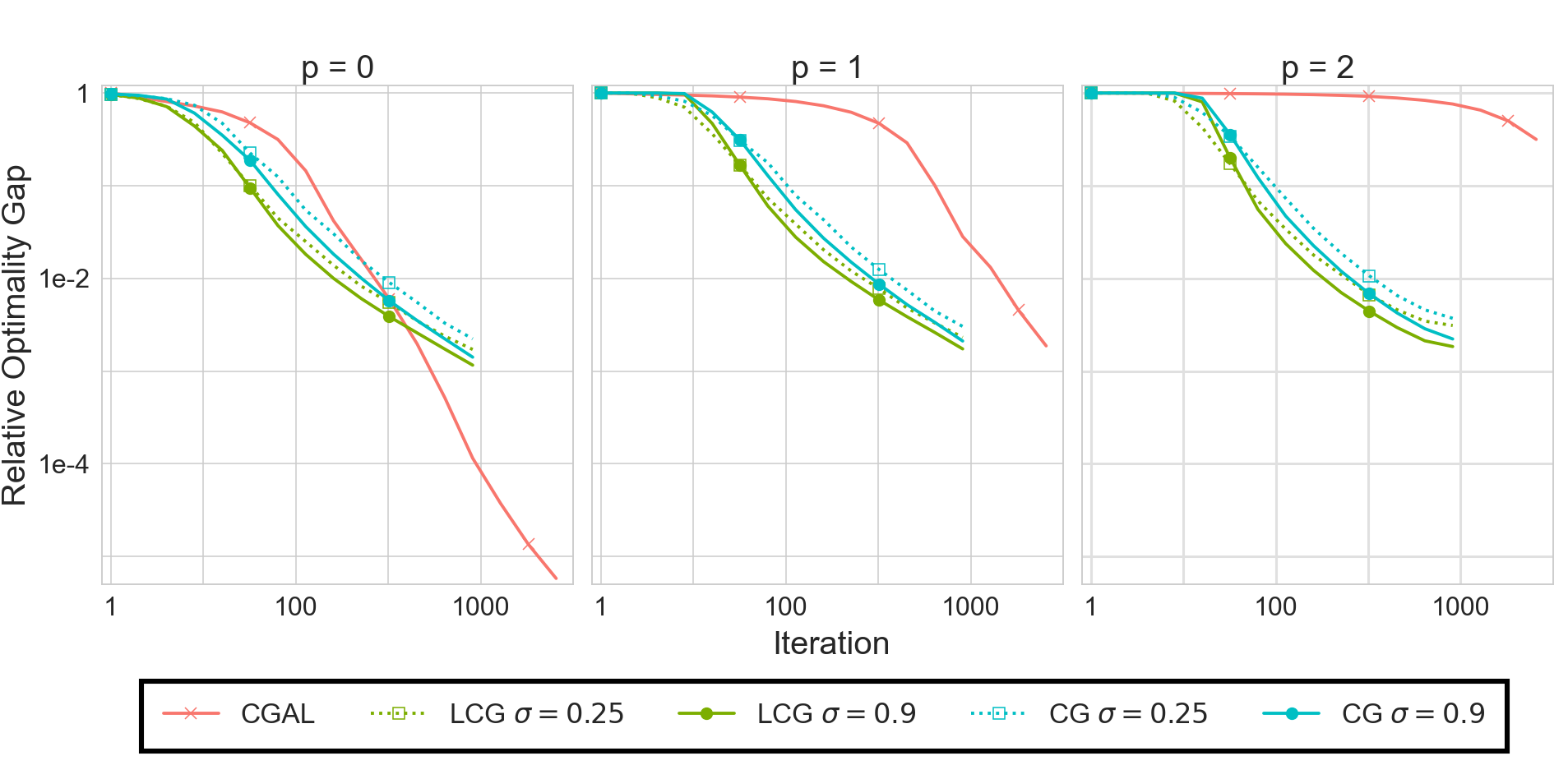}
\caption{Randomly scaled SDP problem: relative optimality gap vs. iterations for several parameter choices (averaged over 30 instances for each $p$)}
\label{fig:SRS_gap_iter}
\end{figure}

\begin{figure}[h!]
\centering
\includegraphics[scale=0.37]{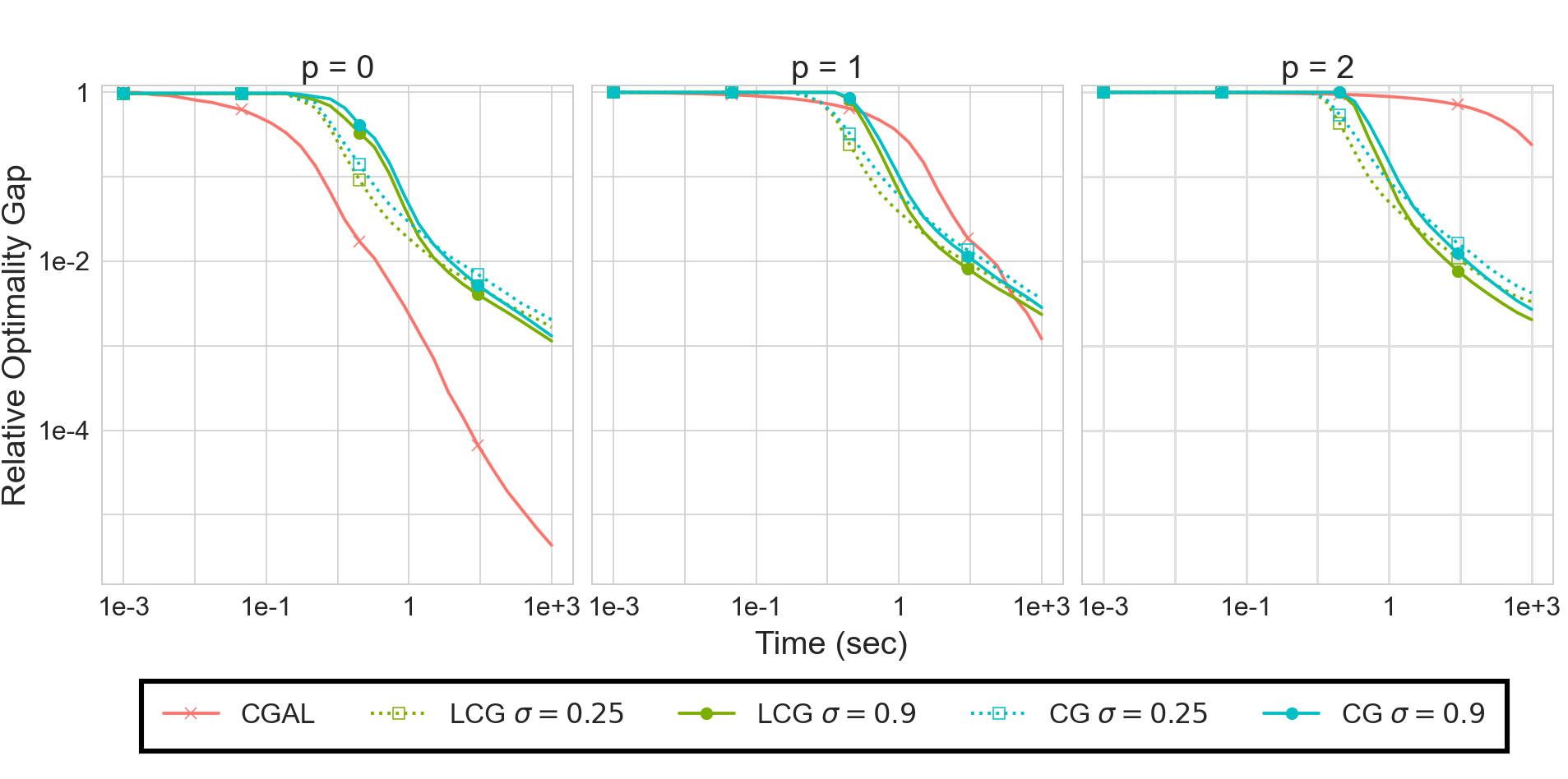}
\caption{Randomly scaled SDP problem: relative optimality gap vs. time for several parameter choices (averaged over 30 instances for each $p$)}
\label{fig:SRS_gap_time}
\end{figure}

For the numerical experiments, we have chosen parameters $n=100$ and $m=100$. Figures \ref{fig:SRS_gap_iter} and \ref{fig:SRS_gap_time} show the performance of our methods in terms of the relative optimality gap for problem \eqref{eq:srs}, where the ``optimal'' solution is computed using the SDPT3 solver, as described in Appendix \ref{sec:SRS}.
Since the dataset consists of homogeneous data points, the plots present aggregated results over 30 instances for each value of $p$.
Figure \ref{fig:SRS_gap_iter} illustrates the performance of each algorithm in terms of \sh{relative optimality gap} 
over iteration count, while Figure \ref{fig:SRS_gap_time} presents the performance in terms of runtime, within a time limit of 1000 seconds.
The plots indicate that, while in the initial stages of running $\CG$ and $\LCG$ with a smaller value of $\sigma$ ($\sigma = 0.25$) are superior, larger values ($\sigma = 0.9$) achieve slightly better overall results. Furthermore, the line search-based method $\LCG$ consistently achieves better iteration and time complexity performances than $\CG$.

Importantly, the plots confirm the conjecture that the performance of $\CGAL$ deteriorates as $p$ increases, our methods show consistent
performance across different $p$ values. In fact, $\CGAL$ has inferior performance for ill-scaled problems ($p=2$) compared to $\CG$ and $\LCG$, leading to an average optimality gap of up to two orders of magnitude larger.
Although Figure~\ref{fig:SRS_gap_iter} shows that $\CGAL$ has iteration complexity comparable to or worse than $\CG$ and $\LCG$ across all values of $p$, its low-cost iterations can make it competitive in terms of time complexity, as shown in Figure~\ref{fig:SRS_gap_time}.
Specifically, $\CGAL$ is able to perform more iterations leading to superior time complexity when $p=0$.
However, when $p=1$ and especially $p=2$, $\CGAL$'s iterations, while still inexpensive, fail to decrease the solution infeasibility, leading to limited improvement.
A detailed analysis of the numerical experiments is provided in Appendix \ref{sec:SRS}.

\section{Conclusions}
\label{sec:conclusions}
%

Solving large scale conic-constrained convex programming problems is still a challenge in mathematical optimization and machine learning. In particular, SDPs with a massive amount of linear constraints appear naturally as convex relaxations of combinatorial optimization problems. For such problems, the development of scalable algorithms is a very active field of research \cite{YurTroFerUdeCev21,ElbMakNaj22}. In this paper, we introduce a new path-following/homotopy strategy to solve such massive conic-constrained problems, building on recent advances in projection-free methods for self-concordant minimization \cite{ZhaFre20,FWICML20,DvuShtSta22}. Our theoretical analysis and numerical experiments scheme show that the method is competitive with state-of-the-art solvers in terms of its iteration complexity and gives promising results in practice. Future work will focus on improvements of our deployed \ac{CndG} solver to accelerate the subroutines. 
\bibliographystyle{plainnat}
\bibliography{mybib}

\appendix
\section{Modifications due to inexact oracles}
\label{app:Inexact}
%
In this appendix we give the missing proofs of the theoretical statements in Sections \ref{sec:extensions} and \ref{sec:inexact-path-following}.  
Consider the function $\tilde{\ce}_{t}(x):=\norm{\tilde{s}_{t}(x)-x}_{x}$, where $\tilde{s}_{t}(x)$ is any member of the set $\scrL^{(\delta,\theta)}_{g}(\frac{1}{t}\nabla F(x))$ (cf. Definition \ref{def:ILMO}). It is easy to see that \ref{lem:e-bound} translates to the case with inexact oracle feedback as
\begin{equation}\label{eq:tildee-bound}
t^{-1}\tilde{\ce}_{t}(x)\leq \nu/t+\widetilde{\gap}_{t}(x)+\Omega_{g}.
\end{equation}
Moreover, one can show that 
\begin{equation}\label{eq:e-LB}
\tilde{\ce}_{t}(x)\geq \frac{t}{\sqrt{\nu}}\left(\widetilde{\gap}_{t}(x)-\Omega_{g}\right).
\end{equation}

\subsection{Proof of Proposition \ref{prop:DeltaInexact}}
The proof of this proposition mimics the proof of the exact regime (cf. Proposition \ref{prop:Delta}), with some important twists. We split the analysis into two phases, bound the maximum number of iterations the algorithms can spend on these phases, and add these numbers up to obtain the worst-case upper bound.

First, for a round $k$ of $\texttt{ICG}(x,t,\eta,\delta,\theta)$, we use the abbreviations  
$$
\tilde{s}_{t}^{k}=\tilde{s}_{t}(x^{k}_{t}),\; \tilde{\ce}_{t}^{k}\equiv\norm{\tilde{s}_{t}^{k}-x^{k}_{t}}_{x^{k}_{t}},\; \tilde{G}^{k}_{t}\equiv \widetilde{\gap}_{t}(x^{k}_{t}),\;\tilde{\alpha}_{t}^{k}\equiv \tilde{\alpha}_{t}(x^{k}_{t}).
$$
We further redefine the phases as $\K_{I}(t):=\{k\in\N\vert \widetilde{\gap}_{t}(x^{k}_{t})>\ca_{t}\}$ (phase I) and $\K_{II}(t):=\{k\in\N\vert\widetilde{\gap}_{t}(x^{k}_{t})\leq \ca_{t}\}$ (phase II), where $\ca_{t}:=\frac{\nu}{t}+\Omega_{g}$.

Note that in the analysis we use the inequality $\Delta_{t}(x^{k}_{t}) >\eta + \theta$ since our goal is to estimate the number of iterations to achieve $\Delta_{t}(x^{k}_{t}) \leq \eta + \theta$. Moreover, if the stopping condition $\widetilde{\gap}_{t}(x^{k}_{t}) \equiv \tilde{G}^{k}_{t} \leq \eta \delta $ holds, we obtain by the definition of the \ac{ILMO} and inequality \eqref{eq:relationGap} that 
\begin{equation}
\label{eq:ILMO_prop_Delta_proof_0}
\eta \delta \geq \widetilde{\gap}_{t}(x^{k}_{t})\geq\delta(\gap_{t}(x^{k}_{t})-\theta)\geq \delta(\Delta_{t}(x^{k}_{t})-\theta),
\end{equation}
which implies that $\Delta_{t}(x^{k}_{t}) \leq \eta + \theta$. Thus, in the analysis we also use the inequality $\tilde{G}^{k}_{t} > \eta \delta >0$.
\sh{We start by showing that the sequence $\{\Delta_t(x_t^k)\}_{k\in \N}$ is monotonically decreasing.}

\sh{\paragraph{Monotonicity of $\Delta_t(x^t_k)$}  Consider $k\in\K_{I}(t)$. 
Combining the condition $\tilde{G}^{k}_{t} > \frac{\nu}{t}+\Omega_{g}$ of this phase with \eqref{eq:e-LB}, we deduce  that $\tilde{\ce}^{k}_{t}\geq \sqrt{\nu}\geq 1$. This in turn implies $\frac{t\tilde{G}^{k}_{t}}{\tilde{\ce}^{k}_{t}(\tilde{\ce}^{k}_{t}+t\tilde{G}^{k}_{t})}<\frac{t\tilde{G}^{k}_{t}}{\tilde{\ce}^{k}+t \tilde{G}^{k}_{t}}<1$. Hence, at this iteration algorithm $\texttt{ICG}(x,t,\eta,\delta,\theta)$ chooses the step size $\tilde{\alpha}_{t}^{k}=\frac{t\tilde{G}^{k}_{t}}{\tilde{\ce}^{k}_{t}(\tilde{\tilde{\ce}}^{k}_{t}+t\tilde{G}^{k}_{t})}$. The per-iteration reduction of the potential function can then be estimated just as in the derivation of \eqref{eq:V_t_progress_1}:
\begin{equation}
\label{eq:ILMO_prop_Delta_proof_1}
V_{t}(x^{k+1}_{t})\leq V_{t}(x^{k}_{t})-\frac{1}{t}\omega\left(\frac{t\tilde{G}^{k}_{t}}{\tilde{\ce}_{t}^{k}}\right).
\end{equation}
Note that the r.h.s. is well-defined since $\tilde{G}^{k}_{t} > \eta \delta >0$ as the stopping condition is not yet achieved. The above inequality, using the definition of the potential function gap as $\Delta_{t}(x^{k}_{t})=V_{t}(x^{k}_{t})-V_{t}(z^{\ast}(t))$, readily yields \begin{equation}\label{eq:inexact_stage_I_monotone}\Delta_{t}(x^{k}_{t})-\Delta_{t}(x^{k+1}_{t})\geq \frac{1}{t}\omega\left(\frac{t\tilde{G}^{k}_{t}}{\tilde{\ce}^{k}_{t}}\right)\geq 0.
\end{equation}

Now, consider $k\in\K_{II}(t)$, which means that $\tilde{G}_{t}^{k}\leq \frac{\nu}{t}+\Omega_{g}$. There are two cases we need to consider, depending on the step size $\tilde{\alpha}_{t}^{k}$. We first consider the case where $\tilde{\alpha}_{t}^{k}=1$. Then $t\tilde{G}_{t}^{k}\geq\tilde{\ce}^{k}_{t}(\tilde{\ce}_{t}^{k}+t\tilde{G}_{t}^{k})$, which implies $\tilde{\ce}_{t}^{k}\in(0,1)$. From there we conclude that $\tilde{G}_{t}^{k}\geq \frac{(\tilde{\ce}_{t}^{k})^{2}}{t(1-\tilde{\ce}_{t}^{k})}$, and, by \eqref{eq:upperV} and \eqref{eq:boundomega}, that
\begin{align*}
V_{t}(x^{k+1}_{t})&\leq V_{t}(x^{k}_{t})-\tilde{G}_{t}^{k}+\frac{1}{t}\omega_{\ast}(\tilde{\ce}_{t}^{k})\\
&\leq V_{t}(x^{k}_{t})-\tilde{G}_{t}^{k}+\frac{1}{t}\frac{(\tilde{\ce}^{k}_{t})^{2}}{2(1-\tilde{\ce}_{t}^{k})}\\
&\leq V_{t}(x^{k}_{t})-\frac{\tilde{G}^{k}_{t}}{2}.
\end{align*}
Therefore, using again that $\tilde{G}_{t}^{k}\leq \frac{\nu}{t}+\Omega_{g}$, we have
\begin{equation}
\label{eq:ILMO_prop_Delta_proof_7}
\Delta_{t}(x^{k}_{t})-\Delta_{t}(x^{k+1}_{t})\geq \frac{\tilde{G}^{k}_{t}}{2} \geq\frac{(\tilde{G}^{k}_{t})^{2}}{2(\nu/t+\Omega_{g})}\geq \frac{t(\tilde{G}^{k}_{t})^{2}}{12(\nu+t\Omega_{g})^{2}}\geq 0,
\end{equation}
where the last inequality follows $\tilde{G}^{k}_{t} > \eta \delta >0$. 

Consider now the case where $\tilde{\alpha}_{t}^{k}<1$, i.e.,  $\tilde{\alpha}_{k}=\frac{t\tilde{G}^{k}_{t}}{\tilde{\ce}^{k}_{t}(\tilde{\ce}^{k}_{t}+t\tilde{G}^{k}_{t})}$.
Then, using this stepsize, just as in the derivation of \eqref{eq:V_t_progress_1} and \eqref{eq:ILMO_prop_Delta_proof_1}, we get 
\begin{equation}
\label{eq:ILMO_prop_Delta_proof_8}
V_{t}(x^{k+1}_{t})\leq V_{t}(x^{k}_{t})-\frac{1}{t}\omega\left(\frac{t\tilde{G}^{k}_{t}}{\tilde{\ce}_{t}^{k}}\right).
\end{equation}
Again, since $\tilde{G}_{t}^{k}\leq \frac{\nu}{t}+\Omega_{g}$ for $k\in\K_{II}(t)$, we have $\frac{\tilde{G}^{k}_{t}}{\tilde{G}^{k}_{t}+\nu/t+\Omega_{g}} \leq \frac{1}{2}$. 
Using the monotonicity of $\omega(\tau)$, inequality \eqref{eq:tildee-bound}, and that $\omega(\tau) \geq \frac{\tau^2}{3}$ for $\tau\leq 1/2$ \cite[Prop. 1]{ZhaFre20},  we can develop the bound
$$
\frac{1}{t}\omega\left(\frac{t\tilde{G}^{k}_{t}}{\tilde{\ce}_{t}^{k}}\right)=\frac{1}{t}\omega\left(\frac{\tilde{G}^{k}_{t}}{\tilde{\ce}_{t}^{k}/t}\right)\geq \frac{1}{t}\omega\left(\frac{\tilde{G}^{k}_{t}}{\tilde{G}^{k}_{t}+\nu/t+\Omega_{g}}\right)\geq\frac{1}{3t}\left(\frac{\tilde{G}^{k}_{t}}{\tilde{G}^{k}_{t}+\nu/t+\Omega_{g}}\right)^{2}.
$$
This and \eqref{eq:ILMO_prop_Delta_proof_8}, again by $\tilde{G}_{t}^{k}\leq \frac{\nu}{t}+\Omega_{g}$, implies that also in this regime we obtain the estimate 
\begin{equation}
\label{eq:ILMO_prop_Delta_proof_9}
\Delta_{t}(x^{k}_{t})-\Delta_{t}(x^{k+1}_{t})\geq \frac{t(\tilde{G}^{k}_{t})^{2}}{12(\nu+t\Omega_{g})^{2}}.
\end{equation}
As we see, the above inequality holds on the whole phase II, both when $\tilde{\alpha}_{t}^{k}<1$ and when $\tilde{\alpha}_{t}^{k}=1$. In particular, since $\tilde{G}^{k}_{t} > \eta \delta >0$, together with \eqref{eq:ILMO_prop_Delta_proof_7} imply that  $\{\Delta_{t}(x^{k}_{t})\}_{k\in\N}$ is monotonically decreasing.}

\sh{We now move to the analysis of the iteration complexity.}
\paragraph{Analysis for Phase I} \sh{Let let $\{k_{i}\}_{i\in\N_{0}}$ denote the increasing set of indices enumerating the elements of the set $\K_{I}(t)$. Since $k_i\in \K_{I}(t)$, it holds that $\tilde{G}^{k_i}_{t}>\frac{\nu}{t}+\Omega_g$ and so $\frac{t\tilde{G}^{k}_{t}}{\nu+t\tilde{G}^{k}_{t}+t\Omega_{g}}>\frac{1}{2}$. Together with \eqref{eq:tildee-bound}, this implies 
\[
\frac{t\tilde{G}_{t}^{k}}{\tilde{\ce}^{k}_{t}}\geq\frac{\tilde{G}^{k}_{t}}{\nu/t+\tilde{G}_{t}^{k}+\Omega_{g}}=\frac{t\tilde{G}^{k}_{t}}{\nu+t\tilde{G}_{t}^{k}+t\Omega_{g}}>\frac{1}{2}.
\]
Using the monotonicity of $\omega(\cdot)$ and that $\omega(\tau) \geq \frac{\tau}{5.3}$ for $\tau\geq 1/2$ \cite[Prop. 1]{ZhaFre20}, we further obtain
\begin{equation}
\label{eq:ILMO_prop_Delta_proof_2}
\Delta_{t}(x^{k_i}_{t})-\Delta_{t}(x^{k_i+1}_{t})\geq \frac{1}{t}\omega\left(\frac{t\tilde{G}^{k_i}_{t}}{\tilde{\ce}^{k_i}_{t}}\right) \geq \frac{\tilde{G}^{k_i}_{t}}{5.3(\nu+t\Omega_{g}+t\tilde{G}_{t}^{k_i})}.
\end{equation}}
By our definition of the inexact oracle, inequalities \eqref{eq:relationGap} and \sh{$\Delta_{t}(x^{k_i}_{t}) > \eta + \theta$, we see that
\[
\tilde{G}^{k_i}_{t} \equiv\widetilde{\gap}_{t}(x^{k_i}_{t})\geq\delta(\gap_{t}(x^{k_i}_{t})-\theta) \geq \delta(\Delta_{t}(x^{k_i}_{t})-\theta) > \eta >0. 
\]}
Combining this with \eqref{eq:ILMO_prop_Delta_proof_2} and the fact that the function $x\mapsto \frac{x}{a+tx}$ is strictly increasing for $a,t >0$ and $x>0$, we obtain
\sh{\begin{equation}
\label{eq:ILMO_prop_Delta_proof_3}
\Delta_{t}(x^{k_i}_{t})-\Delta_{t}(x^{k_{i+1}}_{t})
\geq \frac{\delta(\Delta_{t}(x^{k_i}_{t})-\theta)}{5.3(t\delta(\Delta_{t}(x^{k_i}_{t})-\theta)+\nu+t\Omega_{g})}  \geq \frac{\delta(\Delta_{t}(x^{k_i}_{t})-\theta)}{5.3(t\delta(\Delta_{t}(x^{0}_{t})-\theta)+\nu+t\Omega_{g})},
\end{equation}
where in the first and last inequalities follow from $\{\Delta_{t}(x^{k}_{t})\}_{k\in\N_0}$ being monotonically decreasing and $\Delta_{t}(x^{k_i}_{t})> \eta + \theta \geq \theta$.} Rearranging the terms and defining $q:=\frac{\delta}{5.3(t\delta(\Delta_{t}(x^{0}_{t})-\theta)+\nu+t\Omega_{g})}$, which we know to be in $(0,1)$ since $\nu\geq 1$ and $\Delta_{t}(x^{0}_{t}) > \eta + \theta$, it follows from the previous display that 
\begin{equation}
\label{eq:ILMO_prop_Delta_proof_4}
\sh{\Delta_{t}(x^{k_{i+1}}_{t})\leq  \Delta_{t}(x^{k_i}_{t}) (1-q) + q \theta \leq  \Delta_{t}(x^{0}_{t})(1-q)^{i}+\theta \leq \exp(-iq)\Delta_{t}(x^{0}_{t})+\theta.}
\end{equation}

On the other hand, we have
\sh{$$
\Delta_{t}(x^{k_i}_{t}) \geq \Delta_{t}(x^{k_i}_{t}) - \Delta_{t}(x^{k_i+1}_{t}) \geq \frac{\tilde{G}^{k_i}_{t}}{5.3(\nu+t\Omega_{g}+t\tilde{G}_{t}^{k_i})}>\frac{1}{10.6 t }
$$
since $\frac{t\tilde{G}^{k_i}_{t}}{\nu+t\tilde{G}^{k_i}_{t}+t\Omega_{g}}>\frac{1}{2}$ for $k_i\in\K_{I}(t)$.}
Moreover, we know that \sh{$\Delta_{t}(x^{k_i}_{t}) > \eta + \theta$}, which gives us in combination with \eqref{eq:ILMO_prop_Delta_proof_4} 
\begin{equation}
\label{eq:ILMO_prop_Delta_proof_5}
\max\left\{\frac{1}{10.6 t }, \eta + \theta \right\} \leq \sh{\Delta_{t}(x^{k_i}_{t}) \leq \exp(-iq)\Delta_{t}(x^{0}_{t})}+\theta.
\end{equation}
Solving w.r.t. \sh{$i$}, we see that the number of iterations in $\K_{I}(t)$ is bounded as 
\begin{equation}
\label{eq:ILMO_prop_Delta_proof_6}
K_{t}(x_{t}^{0}): = \left\lceil \frac{5.3(t\delta(\Delta_{t}(x^{0}_{t})-\theta)+\nu+t\Omega_{g})}{\delta}\log\left( \min\left\{\frac{10.6 t \Delta_{t}(x^{0}_{t})}{(1-10.6 t \theta)_+}, \frac{\Delta_{t}(x^{0}_{t})}{\eta} \right\}\right)  \right\rceil.
\end{equation}

\paragraph{Analysis for Phase II} \sh{Let let $\{\ell_{j}\}_{i\in\N_{0}}$ denote the increasing set of indices enumerating the elements of the set $\K_{II}(t)$.
Let $N_{t}(x_{t}^{0})$ be a bound on the number of iterations until $\Delta_{t}(x^{k}_{t})\leq\eta+\theta$ belonging to phase II. We will show that setting $N_t(x_t^0)=\left\lceil \frac{12(\nu+t\Omega_{g})^{2}}{t\delta^{2}}\left(\frac{1}{\eta}-\frac{1}{\Delta_{t}(x^{0})-\theta}\right)_{+}\right\rceil$  is such a bound. Indeed, if $\Delta_{t}(x^{N_t(x^0_t)+1}_{t})\leq \theta$ then the proof is done, otherwise $\Delta_t(x_t^{\ell_j})\geq \theta$ for all $j\leq N_t(x_t^0)+1$ from monotonicity of the sequence \sh{$\{\Delta_{t}(x^{\ell_j}_{t})\}_{\ell_j\in\K_{II}(t)}$}.} 
Calling $\Delta_{t}^{j}\equiv\sh{\Delta(x^{\ell_{j}(t)}_{t})}$ and using the monotonicity of this  sequence, we see from \eqref{eq:ILMO_prop_Delta_proof_9} and \eqref{eq:ILMO_prop_Delta_proof_0} that, for $j=1,\ldots,N_{t}-1(x_{t}^{0})$,
$$
\Delta^{j}_{t}-\Delta^{j+1}_{t}\geq \Delta^{j}_{t}-\Delta_{t}(x^{\sh{\ell_{j}}(t)+1}_{t})\geq \frac{t(\tilde{G}^{\sh{\ell_{j}(t)}}_{t})^{2}}{12(\nu+t\Omega_{g})^{2}} \geq \frac{t\delta^{2}(\Delta^{j}_{t}-\theta)(\Delta^{j+1}_{t}-\theta)}{12(\nu+t\Omega_{g})^{2}}.
$$
This implies that %
for $j=1,\ldots,\sh{N_{t}(x_{t}^{0})}$ that
$$
\frac{1}{\Delta^{j+1}_{t}-\theta}-\frac{1}{\Delta^{j}_{t}-\theta}\geq\frac{t\delta^{2}}{12(\nu+t\Omega_{g})^{2}}, \quad j=1,\ldots,N_{t}(x_{t}^{0}).
$$
Telescoping this expression \sh{together with $\Delta_t^1\leq \Delta_t(x^0_t)$, which stems from the entire sequence $\{\Delta_t(x^k_t)\}_{k\in\N_0}$ being monotonically decreasing},  we see that 
$$
 \frac{1}{\Delta^{N_{t}(x_{t}^{0})+1}_{t}-\theta} 
\geq \sh{\frac{(N_{t}(x_{t}^{0})-1) t\delta^{2}}{12(\nu+t\Omega_{g})^{2}} + \frac{1}{\Delta_{t}^1-\theta}}
\geq \frac{(N_{t}(x_{t}^{0})-1) t\delta^{2}}{12(\nu+t\Omega_{g})^{2}} + \frac{1}{\Delta_{t}(x_{t}^{0})-\theta}\geq \frac{1}{\eta},
$$ is satisfied for $N_t(x^0_t)$.

\sh{We can therefore conclude that the number of iterations required to obtain $\Delta_t(x_t^k)\leq \eta+\theta$ does not exceed $N_t(x^0_t)+K_t(x^0_t)+1$ which is exactly the value of $\tilde{N}(x^{0},\eta,\delta,\theta)$, thus concluding the proof.}

\subsection{Proof of Proposition \ref{prop:GapInexact}}

Our analysis concentrates on the behaviour of the algorithm in Phase II since the estimate for number of iterations in Phase I still holds. Here we split the iterates of the set $\K_{II}(t)=\{k_{1}(t),\ldots,k_{p}(t)\}\equiv \{k_{1},\ldots,k_{p}\}$ into two regimes so that $\K_{II}(t)=H_{t}\cup L_{t}$ with $L_{t}:=\{k\in\K_{II}(t)\vert \Delta_{t}(x^{k}_{t})<\theta\}$ and $H_{t}:=\{k\in\K_{II}(t)\vert \Delta_{t}(x^{k})\geq \theta\}$. By monotonicity of the potential function gap, we know that the indices of of $H_{t}$ ('high states') precede those of $L_{t}$ ('low states'). We bound the size of each of these subsets in order to make $\widetilde{\gap}_{t}$ smaller than $\eta\delta$. Since the potential function gap is monotonically decreasing, we can organise the iterates on phase II as $H_{t}=\{k_{1},\ldots,k_{q}\}$ and accordingly, $L_{t}=\{k_{q+1},\ldots,k_{p}\}$ for some $q\in\{1,\ldots,p\}$. 

\paragraph{Bound on $H_{t}$:}  
For $k_{j}(t)\in H_{t}$ we know that 
\begin{equation}\label{eq:DeltaBound}
\Delta_{t}(x^{k_{j}(t)}_{t})-\Delta_{t}(x^{k_{j+1}(t)}_{t})\geq \Delta_{t}(x^{k_{j}(t)}_{t})-\Delta_{t}(x^{k_{j}(t)+1}_{t})\geq \frac{t(\tilde{G}^{k_{j}(t)}_{t})^{2}}{12(\nu+t\Omega_{g})^{2}}.
\end{equation}
Define $d_{j}:=\delta(\Delta_{t}(x^{k_{j}(t)}_{t})-\theta)$ for $j=1,\ldots,q$, so that $d_{j}\geq 0$. Then, we obtain the recursion 
$$
d_{j+1}-d_{j}\leq-\frac{\Gamma^{2}_{j}}{M}
$$
for $j\in\{1,\ldots,q-1\}$ and $\Gamma_{j}\equiv\tilde{G}^{k_{j}(t)}_{t}\geq d_{j}$ and $\frac{1}{M}\equiv \frac{t\delta}{12(\nu+t\Omega_{g})^{2}}$. By \cite[Proposition 4]{ZhaFre20}, we can conclude that 
$$
d_{q}<\frac{M}{q-1}\text{ and }\min\{\Gamma_{1},\ldots,\Gamma_{q}\}<\frac{2M}{q-1}.
$$
Since we know that $\min\{\Gamma_{1},\ldots,\Gamma_{q}\}\geq\delta\eta$, we obtain the bound 
\begin{equation}\label{eq:q}
q\leq \frac{2M}{\delta\eta}+1=\frac{24(\nu+t\Omega_{g})^{2}}{t\delta^{2}\eta}+1.
\end{equation}

\paragraph{Bound on $L_{t}$:} 
On this subset we cannot use the same argument as for the iterates in $H_{t}$, since we cannot guarantee that the sequence $d_{j}$ involved in the previous part of the proof is positive. However, we still know that \eqref{eq:DeltaBound} applies for the iterates in $L_{t}$. If we telescope this expression over the indices $q+1,\ldots,p$, and using the fact that $\Delta_{t}(x^{k_{j}}(t))\in[0,\theta]$, we see that 
\begin{align*}
\theta\geq \Delta_{t}(x_{t}^{k_{q+1}(t)})-\Delta_{t}(x_{t}^{k_{p+1}(t)})&\geq \frac{t(p-q)}{12(\nu+t\Omega_{g})^{2}}\min_{q+1\leq j\leq p}\left(\tilde{G}_{t}^{k_{j}(t)}\right)^{2}\\
&\geq  \frac{t(p-q)\eta^{2}\delta^{2}}{12(\nu+t\Omega_{g})^{2}}
\end{align*}
It follows
\begin{equation}\label{eq:pq}
(p-q)\leq \frac{12\theta(\nu+t\Omega_{g})^{2}}{t\eta^{2}\delta^{2}}.
\end{equation}

\paragraph{Combining the bounds:}
Combining the expressions \eqref{eq:q} and \eqref{eq:pq} together with the bound \eqref{eq:ILMO_prop_Delta_proof_6}, we see that 
\begin{align*}
\tilde{R}_{t}(x^{0}_{t},\eta,\delta,\theta)&= \left\lceil \frac{5.3(t\delta(\Delta_{t}(x^{0}_{t})-\theta)+\nu+t\Omega_{g})}{\delta}\log\left( \min\left\{\frac{10.6 t \Delta_{t}(x^{0}_{t})}{(1-10.6 t \theta)_+}, \frac{\Delta_{t}(x^{0}_{t})}{\eta} \right\}  \right)\right\rceil\\
&+ \left\lceil \frac{12(\nu+t\Omega_{g})^{2}}{t\delta^{2}\eta} \left(2+\frac{\theta}{\eta}\right)\right\rceil+1
\end{align*}
defines an upper bound for the length of phase I. 

\subsection{Analysis of the outer loop and proof of Theorem \ref{th:Complexity_inexact}}
We set $\hat{x}_{i}\equiv x^{\tilde{R}_{i}}_{t_{i}}$ the last iterate of procedure $\ICG(x^{0}_{t_{i}},t_{i},\eta_{i},\delta,\theta_{i})$, where $\tilde{R}_{i}$ is at most $R_{\tilde{G},t_{i}}(x^{0}_{t_{i}},\eta_{i},\delta,\theta_{i})$ by \ref{prop:GapInexact}. Therefore, we know that $\widetilde{\gap}_{t_{i}}(\hat{x}_{i})\leq\eta_{i}\delta$, and a-fortiori 
$$
\Delta_{t_{i}}(\hat{x}_{i})\leq \gap_{t_{i}}(\hat{x}_{i})\leq \eta_{i}+\theta_{i}.  
$$
By eq. \ref{eq:g2} and the definition of the $(\delta,\theta_{i})$-\ac{ILMO}, we obtain 
\begin{align*}
g(\hat{x}_{i})-g(z^{\ast}(t_{i}))&\leq\gap_{t_{i}}(\hat{x}_{i})+\frac{1}{t}F'(\hat{x}_{i})[z^{\ast}(t_{i})-\hat{x}_{i}]\\
&\leq \frac{\widetilde{\gap}_{t_{i}}(\hat{x}_{i})}{\delta}+\theta_{i}+\frac{\nu}{t_{i}}\\
&\leq \eta_{i}+\theta_{i}+\frac{\nu}{t_{i}}.
\end{align*}
Hence, using \ref{lem:pathfollowing}, we observe
$$
g(\hat{x}_{i})-\Opt=g(\hat{x}_{i})-g(z^{\ast}(t_{i}))+g(z^{\ast}(t_{i}))-\Opt\leq \eta_{i}+\theta_{i}+\frac{2\nu}{t_{i}}.
$$
Since $t_{i}=\frac{2\nu}{\eta_{i}}$ and $\theta_i=\eta_i$, we obtain $g(\hat{x}_{i})-\Opt\leq  3\eta_{i}$. The latter, by the choice of the number of restarts $I$ implies that 
$$
g(\hat{z})-\Opt = g(\hat{x}_{I})-\Opt \leq \eps.
$$

Our next goal is to estimate the total number of inner iterations based on the estimation for the first epoch $i=0$ and summation over epochs $i=1,...,I$. For each epoch we use the estimate from \ref{prop:GapInexact}, which we repeat here for convenience
\begin{align}
\label{eq:IO_outer_proof_2}
\tilde{R}_{t_i}(x^{0}_{t_{i}},\eta_{i},\delta,\theta_{i})&= \left\lceil \frac{5.3(t_{i}\delta(\Delta_{t_{i}}(x^{0}_{i})-\theta_{i})+\nu+t_{i}\Omega_{g})}{\delta}\log\left( \min\left\{\frac{10.6 t_{i} \Delta_{t_{i}}(x^{0}_{t_{i}} )}{(1-10.6 t_{i} \theta_{i})_+}, \frac{\Delta_{t_{i}}(x^{0}_{t_{i}})}{\eta_{i}} \right\} \right) \right\rceil \notag \\
& \hspace{2em} + \left\lceil \frac{12(\nu+t_{i}\Omega_{g})^{2}}{t_{i}\delta^{2}\eta_{i}} \left(2+\frac{\theta_{i}}{\eta_{i}}\right)\right\rceil+1.
\end{align}
We see that we need to estimate the quantities $t_{i}(\Delta_{t_{i}}(x^{0}_{i})-\theta_{i})$ and $\frac{\Delta_{t_{i}}(x^{0}_{t_{i}})}{\eta_{i}}$, which is our next step.
Consider $i\geq 1$ and observe that
\begin{align*}
	V_{t_{i+1}}(x^{0}_{t_{i+1}})-V_{t_{i+1}}(z^{\ast}(t_{i+1}))&=V_{t_{i}}(x^{0}_{t_{i+1}})-V_{t_{i}}(z^{\ast}(t_{i+1}))\\
	&+\left(\frac{1}{t_{i+1}}-\frac{1}{t_{i}}\right)(F(x^{0}_{t_{i+1}})-F(z^{\ast}(t_{i+1})))\\
	&\leq V_{t_{i}}(x^{R_{i}}_{t_{i}})-V_{t_{i}}(z^{\ast}(t_{i}))\\
	&+\left(1-\frac{t_{i+1}}{t_{i}}\right)\left(V_{t_{i+1}}(x^{0}_{t_{i+1}})-V_{t_{i+1}}(z^{\ast}(t_{i+1}))\right)\\
	&+(1-\frac{t_{i+1}}{t_{i}})(g(z^{\ast}(t_{i+1}))-g(x^{0}_{t_{i+1}})).
\end{align*}
Moreover,
$$
g(x^0_{t_{i+1}})-g(z^{\ast}(t_{i+1}))\leq g(x^0_{t_{i+1}})-\Opt\leq \eta_i+\theta_{i}+\frac{2\nu}{t_i}.
$$
Since $t_{i+1}>t_{i}$, the above two inequalities imply
\begin{align*}
\frac{t_{i+1}}{t_{i}}[V_{t_{i+1}}(x^{0}_{t_{i+1}})-V_{t_{i+1}}(z^{\ast}(t_{i+1}))]\leq V_{t_{i}}(x^{R_{i}}_{t_{i}})-V_{t_{i}}(z^{\ast}(t_{i}))+\frac{t_{i+1}-{t_i}}{t_{i}}(\eta_{i}+\theta_{i}+\frac{2\nu}{t_{i}}).
\end{align*}
Whence, 
\begin{align*}
	t_{i+1}(\Delta_{t_{i+1}}(x^{0}_{t_{i+1}}) - \theta_{i+1}) & \leq t_{i}\Delta_{t_{i}}(x^{R_{i}}_{t_{i}})+(t_{i+1}-t_{i})(\eta_{i}+\theta_{i}+\frac{2\nu}{t_{i}}) - t_{i+1} \theta_{i+1} \\
	& \leq t_{i}(\eta_{i}+\theta_{i}) +(t_{i+1}-t_{i})(\eta_{i}+\theta_{i}+\frac{2\nu}{t_{i}}) - t_{i+1} \theta_{i+1} \\
	& = 4 \nu + (t_{i+1}-t_{i})\cdot \frac{6\nu}{t_{i}} - 2 \nu = 2 \nu +  (1/\sigma-1)6\nu = 2\nu (3/\sigma - 2),  
\end{align*}
where we used our assignments $t_{i+1}=t_i/\sigma, \eta_it_i=\theta_i t_i = 2 \nu$.
In the same way, we obtain
\begin{align*}
	\frac{\Delta_{t_{i+1}}(x^{0}_{t_{i+1}})}{\eta_{i+1}} &\leq \frac{t_{i}}{t_{i+1}\eta_{i+1}}\Delta_{t_{i}}(x^{R_{i}}_{t_{i}})+\frac{t_{i+1}-t_{i}}{t_{i+1}\eta_{i+1}}(\eta_{i}+\theta_{i}+\frac{2\nu}{t_{i}}) \\
	&\leq \frac{t_{i}}{t_{i+1}\eta_{i+1}} (\eta_{i}+\theta_{i}) + \frac{t_{i+1}-t_{i}}{t_{i+1}\eta_{i+1}} (\eta_{i}+\theta_{i}+\frac{2\nu}{t_{i}}) \\
	& = \frac{4\nu }{t_{i+1}\eta_{i+1}} + \frac{t_{i+1}-t_{i}}{t_{i+1}\eta_{i+1}} \frac{6\nu}{t_{i}} = 2 + (1/\sigma-1) 3 =   (3/\sigma - 1),
\end{align*}
where we again used our assignments $t_{i+1}=t_i/\sigma, \eta_it_i=\theta_i t_i = 2 \nu$.

Using these bounds, our assignments $\eta_i=\theta_i= 2 \nu/t_i$, we see that $1-10.6 t_{i} \theta_{i} <0$ and the bound \eqref{eq:IO_outer_proof_2} simplifies to
\begin{align}
\label{eq:IO_outer_proof_3}
\tilde{R}_{t_i}(x^{0}_{t_{i}},\eta_{i},\delta,\theta_{i})&\leq \left\lceil \frac{5.3(2\delta\nu (3/\sigma - 2)+\nu+t_{i}\Omega_{g})}{\delta}\log (3/\sigma - 1) \right\rceil + \left\lceil \frac{18(\nu+t_{i}\Omega_{g})^{2}}{\nu\delta^{2}}\right\rceil+1 \\
& \leq 3 + \frac{5.3(2\delta\nu (3/\sigma - 2)+\nu)}{\delta}\log (3/\sigma - 1) + \frac{18\nu}{\delta^2}\\
& \hspace{1em} + t_i \cdot \left(\frac{5.3  \Omega_{g}}{\delta}\log (3/\sigma - 1) + \frac{18 \Omega_{g}}{\delta^2} \right) + t_i^2 \cdot \frac{18 \Omega_{g}^2}{\nu\delta^2}.
\end{align}
Using the same bounds for $\sum_i t_i$, $\sum_i t_i^2$ as in the proof of \ref{th:Complexity} (see Section \ref{sec:outer}), we obtain that
\begin{align}
	\sum_{i=1}^{I}\tilde{R}_{i}&\leq \left(3 + \frac{5.3(2\delta\nu (3/\sigma - 2)+\nu)}{\delta}\log (3/\sigma - 1) + \frac{18\nu}{\delta^2} \right) \frac{\log(2\eta_{0}/\eps)}{\log(1/\sigma)} \notag  \\
	& \hspace{1em}+ \left(\frac{5.3  \Omega_{g}}{\delta}\log (3/\sigma - 1) + \frac{18 \Omega_{g}}{\delta^2} \right) \frac{4\nu}{\eps(1-\sigma)} \notag \\
	& \hspace{1em} + \frac{288 \nu \Omega_{g}^2}{\delta^2 \eps^2 (1-\sigma^2)}. \label{eq:IO_outer_proof_4}
\end{align}
Finally, we estimate $\tilde{R}_0$ as follows using \eqref{eq:IO_outer_proof_2} and that $t_0\eta_0=2\nu$, $t_0=\nu/\Omega_g$, $\theta_0=\eta_0$, $\delta <1$:
\begin{align*}
\tilde{R}_0 &\leq \left\lceil \frac{5.3(t_{0}\delta\Delta_{t_{0}}(x^{0})+\nu+t_{0}\Omega_{g})}{\delta}\log\left( \frac{\Delta_{t_{0}}(x^{0})}{\eta_{0}}\right)    \right\rceil + \left\lceil \frac{18(\nu+t_{0}\Omega_{g})^{2}}{\nu\delta^{2}} \right\rceil+1\\
& \leq \frac{5.3(F(x^{0})-F(z^*(t_0))+3\nu)}{\delta}\log\left( \frac{  F(x^{0})-F(z^*(t_0))  + \nu}{2 \nu} \right)   +   \frac{72\nu}{\delta^{2}} +3.
\end{align*}
Combining the latter with \eqref{eq:IO_outer_proof_4}, we obtain that the leading complexity term is $O\left( \frac{\nu \Omega_{g}^2}{\delta^2 \eps^2 (1-\sigma^2)}\right).$

We remark that assuming that the oracle accuracies $\theta_i$ have geometric decay is crucial in order to obtain the same dependence on the accuracy in the complexity as in the exact case. Indeed, if the oracle accuracy is fixed to be $\theta_i\equiv\theta$, then the last term in the bound \eqref{eq:IO_outer_proof_2} contains a term $\frac{12t_i^2\Omega_g^2\theta}{2\nu\delta^2\eta_i} = \frac{12t_i^3\Omega_g^2\theta}{4\nu^2\delta^2}$, which after summation will imply that the leading term in the complexity would be $O(1/\eps^4)$.

\section{Added Details to the Numerical Experiments}
\label{sec:exp}
%

In this section we report outcomes obtained by extensive numerical experiments illustrating the practical performance of our method.
We implement Algorithm \ref{alg:Homotopy} using procedure $\CG$ and $\LCG$ as solvers for the inner loops. We compare the results of our algorithm to two benchmarks: (1) an interior point solution of the problem, and (2) $\CGAL$, implemented following the suggestions in \cite{YurTroFerUdeCev21}. \sh{In order to enable a fair comparison with $\CGAL$, both $\CG$ and $\LCG$ also used the LMO implementation of the random Lanczos algorithm provided as part of the $\CGAL$ code.} The interior point solution was obtained via \gi{the Matlab-based modelling} \sh{inteface} CVX \cite{cvx} (version 2.2) with the SDPT3 solver (version 4.0) \cite{Toh:1999aa}.  The implementation of CGAL was taken from \cite{SKetchyCGAL}. All experiments have been conducted on an Intel(R) Xeon(R) Gold 6254 CPU @3.10GHz server limited to 4 threads per run and 512G total RAM using Matlab R2019b.

\subsection{MaxCut Problem}\label{sec:max_cut}

The following class of SDPs is studied in \cite{AroHazKal05}. This SDP arises in many algorithms such as approximating MaxCut, approximating the CUTNORM of a matrix, and approximating solutions to the little Grothendieck problem \cite{CharWir04,BanNouVor16}. The optimization problem is defined as
\begin{equation}\label{eq:MaxQP}\tag{MAXQP}
 \begin{aligned}
	&\max&&\inner{\bC,\bX}\\
	&\text{s.t.}&& X_{ii}\leq 1\quad i=1,\ldots,n\\
	&&& \gi{\bX}\succeq 0
 \end{aligned}
\end{equation}
Let $\setQ:=\{y\in\Rn\vert y_{i}\leq 1,\; i=1,\ldots,n\}$, and consider the conic hull of $\setQ$, defined as
$\setK:=\{(y,t)\in\Rn\times\R\vert \frac{1}{t}y\in\setQ,t>0\}\subset\R^{n+1}.$ This set admits the $n$-logarithmically homogenous barrier
$$
f(y,t)=-\sum_{i=1}^{n}\log(t-y_{i})=-\sum_{i=1}^{n}\log(1-\frac{1}{t}y_{i})-n\log(t).
$$
We can then reformulate \eqref{eq:MaxQP} as
\begin{equation}
  \begin{aligned}
	  & \sh{\min_{\bX,t}}&&\{\sh{g(\bX):=-\inner{\bC,\bX}}\}\\
	  &\text{s.t. }& &(\bX,t)\in\sh{\setX:=\{(\bX',t')\in\scrS^{n}\times\R\vert \bX'\succeq 0,\;\tr(\bX')\leq n,t'=1\}},\\
	  &&&\scrP(\bX,t)\in\setK
  \end{aligned}
\end{equation}
where $\scrP(\bX,t):=[X_{11};\ldots;X_{nn};t]^{\top}$ is a linear homogenous mapping from $\symm^{n}\times\R$ \sh{to} $\R^{n+1}$. Set $F(\bX,t)=f(\scrP(\bX,t))$ for $(\bX,t)\in\symm^{n}\times(0,\infty)$.

\sh{We apply this formulation to the classical MaxCut problem: given an undirected graph $\scrG=(\scrV,\scrE)$ with vertex set $\scrV=\{1,\ldots,n\}$ and edge set $\scrE$, we seek to find 
the maximum-weight cut of the graph. The MaxCut problem is formulated as $\max_{\bx\in\{\pm 1\}^{n}}\bx^{\top}\bL\bx$,
where $\bL$ is the combinatorial Laplace matrix of 
graph $\scrG$. A convex relaxation of this NP-hard combinatorial optimization problem can be given by the SDP \cite{Goemans:1995aa}
\begin{align*}
  \max \frac{1}{4}\tr(\bL\bX)\text{ s.t.: }X_{ii}\leq 1,i\in\{1,2,\ldots,n\},\bX\in\symm^{n}_{+},
\end{align*}
which is the form of \eqref{eq:MaxQP} with $\bC=\frac{1}{4}\bL$.}

To evaluate the performance \gi{of the tested algorithms on this problem}, we consider the random graphs G1-G30, published online in \cite{Gset}. \sh{The $\CG$ and $\LCG$ methods were applied to all datasets with parameters $\eta_0=2\Omega_g$ and $\sigma\in\{0.25,0.5,0.9\}$. The solutions were compared to the solution obtained from running $\CGAL$, using the original $\CGAL$ code applied to the scaled problem and assuming that all constraint are satisfied with equality. The  computations were stopped after either 
$10,000$ seconds of running time.}
The benchmark solutions for the interior point method for each data set are displayed in Table \ref{tbl:MaxCutDataSetsInfo}. Table~\ref{tbl:MaxCutDataSetsInfo} also displays the size of each dataset, the value obtained by solving the MaxCut SDP relaxation using CVX with the SDPT3 solver, and $\lambda_{\text{min}}(\bX^{\SDPT3})$, that is the minimum eigenvalue of the solution. We observe that for larger graphs the interior-point solver SDPT3 returns infeasible solutions, displaying negative eigenvalues. If this occurs, the value obtained by a corrected solution is used as a reference value instead. This corrected solution $\tilde{\bX}^{\SDPT3}$ is constructed as $\tilde{\bX}^{\SDPT3}=(\bX^{\SDPT3}-\lambda_{\text{min}}(\bX^{\SDPT3})\bI)/\alpha$, where $\bI$ is the identity matrix, and $\alpha=\max_{i\in\{1,\ldots,n\}} \max\{\tilde{X}^{\SDPT3}_{ii},1\}$ is the minimal number for which $\tilde{X}^{\SDPT3}_{ii}\leq 1$ for all $i=1,\ldots,n$. \sh{The objective value of this corrected solution, $g(\tilde{X}^{\SDPT3})$ is also recorded in Table~\ref{tbl:MaxCutDataSetsInfo} as the feasible value (Feas. value). We therefore label the optimal solution $g^{*}$ as the minimum of $g(\tilde{X}^{\SDPT3})$ and the minimal value obtained by any of the methods.}
The performance measure we employ to benchmark our experimental results is the  \sh{\emph{relative optimality gap}, defined as $\frac{\abs{g(\bX)-g^*}}{\abs{g^*}}$.}
Moreover, since CGAL does not necessarily generate feasible solutions, its reported solution $\bX^{\CGAL}$ is corrected to the feasible one \sh{ $\tilde{\bX}^{\CGAL}=\bX^{\CGAL}/(\gamma+1)$, where $\gamma=\max_{i\in \{1,\ldots,n\}} \max\{X^{\CGAL}_{ii}-1,0\}$ is the \emph{relative feasibility gap} of $\CGAL$. In particular, for every dataset $D$, method $M$, and parameter choice $P$, after each iteration $i$ we record the objective value $g^{M,P,D}_i$ of the solution after this iteration (or, in the case of $\CGAL$, the corrected solution), the corresponding relative optimality gap 
\gi{$r^{M,P,D}_i={|g^{M,P,D}_{i}-g^*|}/{|g^*|}$,}
the time from the start of the run $t^{M,P,D}_i$ and, in the case of $\CGAL$, the feasibility gap $\gamma^{\CGAL,P,D}_i$ of the current solution.} 

Figures \ref{fig:MaxCut_Small_fVsIter}-\ref{fig:MaxCut_Large_fVsTime} illustrate the performance of our methods for various parameter values and datasets G1-G30 vs. both iteration and time.
\sh{More specifically,
Figures~\ref{fig:MaxCut_Small_fVsIter} and \ref{fig:MaxCut_Large_fVsIter} display, for each iteration $i$, the best relative optimality gap $\min_{j\leq i}\{r^{M,P,D}_j\}$ attained up to that iteration; Figures~\ref{fig:MaxCut_Small_fVsTime} and \ref{fig:MaxCut_Large_fVsTime} display, for each time point $t$, the best relative optimality gap $\min_{j}\{r^{M,P,D}_j \vert t^{M,P,D}_j\leq t\}$ obtained up to time $t$. Additionally, Tables \ref{tbl:max_cut1} and \ref{tbl:max_cut2} report the algorithmic performances of the tested methods for each data base and method, at specific time points $t\in\{100,1000,10000\}$ seconds. Specifically, for each of the time points we compute $j^*=\argmin_{j}\{g^{M,P,D}_j \vert t^{M,P,D}_j\leq t\}$, and then display the relative optimality gap (Opt. Gap) $r^{M,P,D}_{j^*}$, the relative feasibility gap $\gamma^{\CGAL,D}_{j^*}$ for $\CGAL$ (Feas. Gap), and the number of iterations $\max\{j\vert t^{M,P,D}_j\leq t\}$ performed until that time point (Iter). The smallest relative optimality gap for each data set and time is marked in bold.}

\sh{Figures \ref{fig:MaxCut_Small_fVsIter} and \ref{fig:MaxCut_Large_fVsIter}, show that the performances of $\CG$ and $\LCG$ are similar across iterations, with $\LCG$ having a slight advantage. Further, the performance of both methods remains nearly unchanged for different values of $\sigma$. Note that, for a low number of iterations (smaller than 1000), $\CG$ and $\LCG$ generally perform better than $\CGAL$. However, beyond that point, $\CGAL$ exhibits a sudden improvement in objective value, leading to better performance. When looking at the methods' performances over time in Figures \ref{fig:MaxCut_Small_fVsTime} and \ref{fig:MaxCut_Large_fVsTime}, the initial advantage of $\CG$ and $\LCG$ seems to disappear, due to the lower per-iteration cost of $\CGAL$. However, a closer look at the results in Tables~\ref{tbl:max_cut1} and \ref{tbl:max_cut2} reveals that $\CG$ achieves the lowest optimality gap in 7 out of 30 datasets after 1000 seconds, and in 13 out of the 30 datasets after 10000 seconds. 
The former instances can be identified by the fact that $\CGAL$ obtains an optimality gap higher than $0.5\%$ after 100 seconds, and has a very slow improvement rate after that point. In contrast, $\CG$ obtains a higher optimality gap of approximately $2\%$ at that time but manages to achieve better performance due to its superior optimality gap reduction rate.
Additionally, we observe that the costly iterations of $\LCG$ result in inferior performance over time \gi{compared} to $\CG$, as it \gi{can only complete} 
a tenth of the iterations \gi{within} the same time frame.
In summary, while $\CGAL$ manages to obtain a less than $1\%$ optimality gap quite quickly, in about a third of the instances it gets `stuck' and does not significantly improve the objective function, whereas $\LCG$ and $\CG$ are more consistent across the instances starting from a higher optimality gap, but making progress in a constant rate throughout their execution, with $\CG$ being more efficient in terms of time complexity.
}

\begin{table}[t]
   \centering
   \caption{MaxCut datasets information}
   \label{tbl:MaxCutDataSetsInfo}
   \begin{minipage}{0.48\textwidth}
   \resizebox{\linewidth}{!}{
   \begin{tabular}{p{1.5cm}p{1.5cm}p{1.5cm}|rrr}
	   \hline
	   \hline
	   &&& \multicolumn{3}{c}{$\SDPT3$}\\
	   Dataset\;& \#Nodes\; & \#Edges\; & \multicolumn{1}{>{\centering\arraybackslash}p{1cm}}{Obj. value} &  \multicolumn{1}{>{\centering\arraybackslash}p{0.8cm}}{$\lambda_{\text{min}}$} & \multicolumn{1}{>{\centering\arraybackslash}p{1cm}}{Feas. value} \\
	   \hline\hline
	   G1 	&	800	&	19176	&	12083.2	&	0	&	 -	\\
	   G2 	&	800	&	19176	&	12089.43	&	0	&	 -	\\
	   G3 	&	800	&	19176	&	12084.33	&	0	&	 -	\\
	   G4 	&	800	&	19176	&	12111.5	&	0	&	 -	\\
	   G5 	&	800	&	19176	&	12100	&	0	&	 -	\\
	   G6 	&	800	&	19176	&	2667	&	0	&	 -	\\
	   G7 	&	800	&	19176	&	2494.25	&	0	&	 -	\\
	   G8 	&	800	&	19176	&	2535.25	&	0	&	 -	\\
	   G9 	&	800	&	19176	&	12111.5	&	0	&	 -	\\
	   G10 	&	800	&	19176	&	12111.5	&	0	&	 -	\\
	   G11 	&	800	&	1600	&	634.83	&	0	&	 -	\\
	   G12 	&	800	&	1600	&	628.41	&	0	&	 -	\\
	   G13 	&	800	&	1600	&	651.54	&	0	&	 -	\\
	   G14 	&	800	&	1600	&	12111.5	&	0	&	 -	\\
	   G15 	&	800	&	4661	&	3171.56	&	0	&	 \phantom{14135.9}-	\\
	   \hline
	   \hline
   \end{tabular}}
   \end{minipage}
   \begin{minipage}{0.48\textwidth}
	   \resizebox{\linewidth}{!}{
	   \begin{tabular}{p{1.5cm}p{1.5cm}p{1.5cm}|rrr}
		   \hline
		   \hline
		   &&& \multicolumn{3}{c}{$\SDPT3$}\\
		   Dataset\;& \#Nodes\; & \#Edges\; & Value &  \multicolumn{1}{>{\centering\arraybackslash}p{0.8cm}}{$\lambda_{\text{min}}$} & \multicolumn{1}{>{\centering\arraybackslash}p{1cm}}{Feas. value} \\
		   \hline\hline
	   G16 	&	800	&	4672	&	3175.02	&	0	&	 -	\\
	   G17 	&	800	&	4667	&	3171.25	&	0	&	 -	\\
	   G18 	&	800	&	4694	&	1173.75	&	0	&	 -	\\
	   G19 	&	800	&	4661	&	1091	&	0	&	 -	\\
	   G20 	&	800	&	4672	&	1119	&	0	&	 -	\\
	   G21 	&	800	&	4667	&	1112.06	&	0	&	 -	\\
	   G22 	&	2000	&	19990	&	18276.89	&	-1	&	14135.95	\\
	   G23 	&	2000	&	19990	&	18289.22	&	-1	&	14142.11	\\
	   G24 	&	2000	&	19990	&	18286.75	&	-1	&	9997.75	\\
	   G25 	&	2000	&	19990	&	18293.5	&	-1	&	9996	\\
	   G26 	&	2000	&	19990	&	18270.24	&	-1	&	14132.87	\\
	   G27 	&	2000	&	19990	&	8304.5	&	-1	&	4141.75	\\
	   G28 	&	2000	&	19990	&	8253.5	&	-1	&	4100.75	\\
	   G29 	&	2000	&	19990	&	8377.75	&	-1	&	4209	\\
	   G30 	&	2000	&	19990	&	8381	&	-1	&	4215.5	\\
	   \hline
	   \hline
   \end{tabular}}
\end{minipage}
\end{table}

\begin{figure}[h!]
   \centering
   \includegraphics[width=\textwidth]{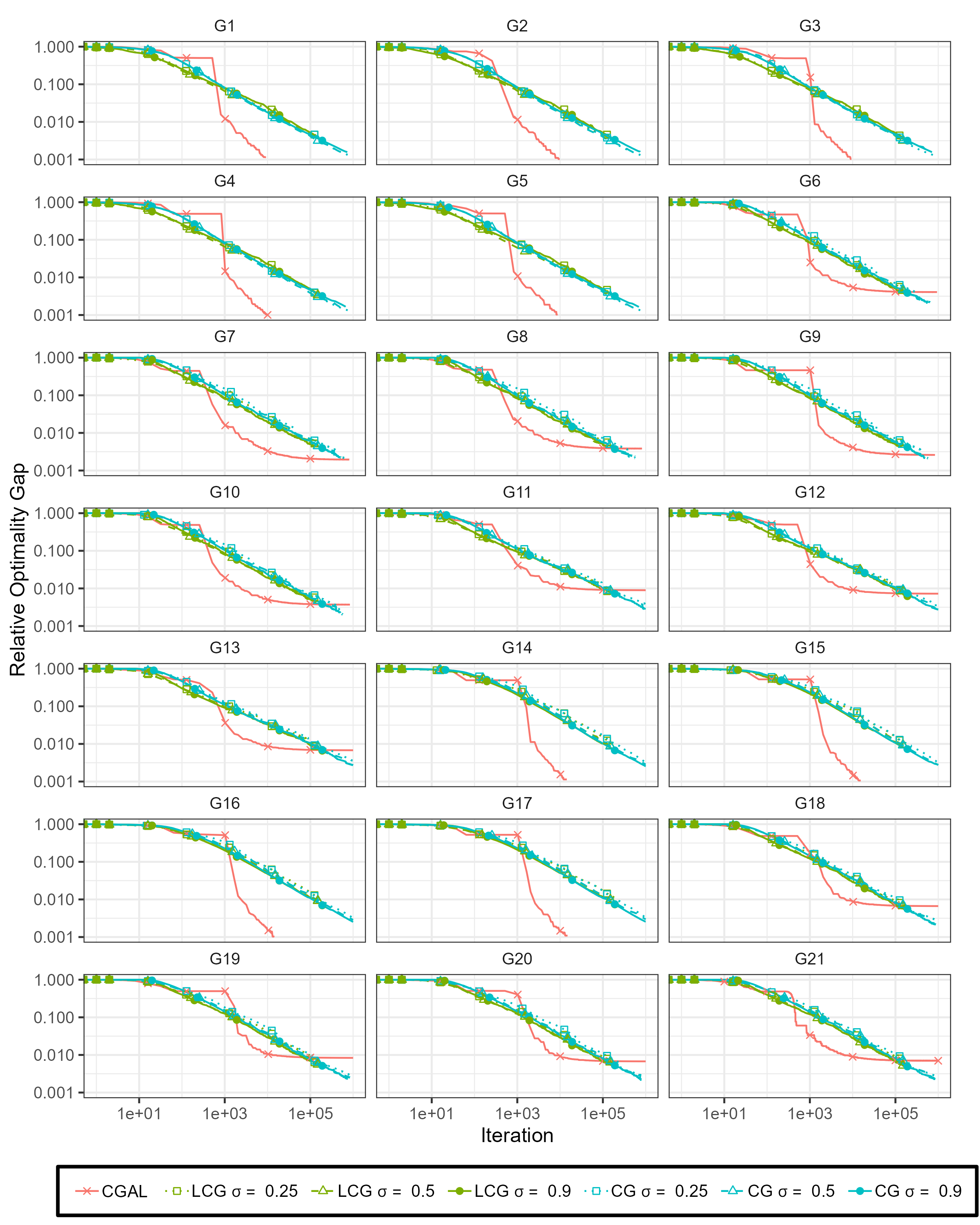}
   \caption{MaxCut datasets G1-G21 relative optimality gap vs. iteration.}
   \label{fig:MaxCut_Small_fVsIter}
\end{figure}
\begin{figure}[h!]
   \centering
   \includegraphics[width=\textwidth]{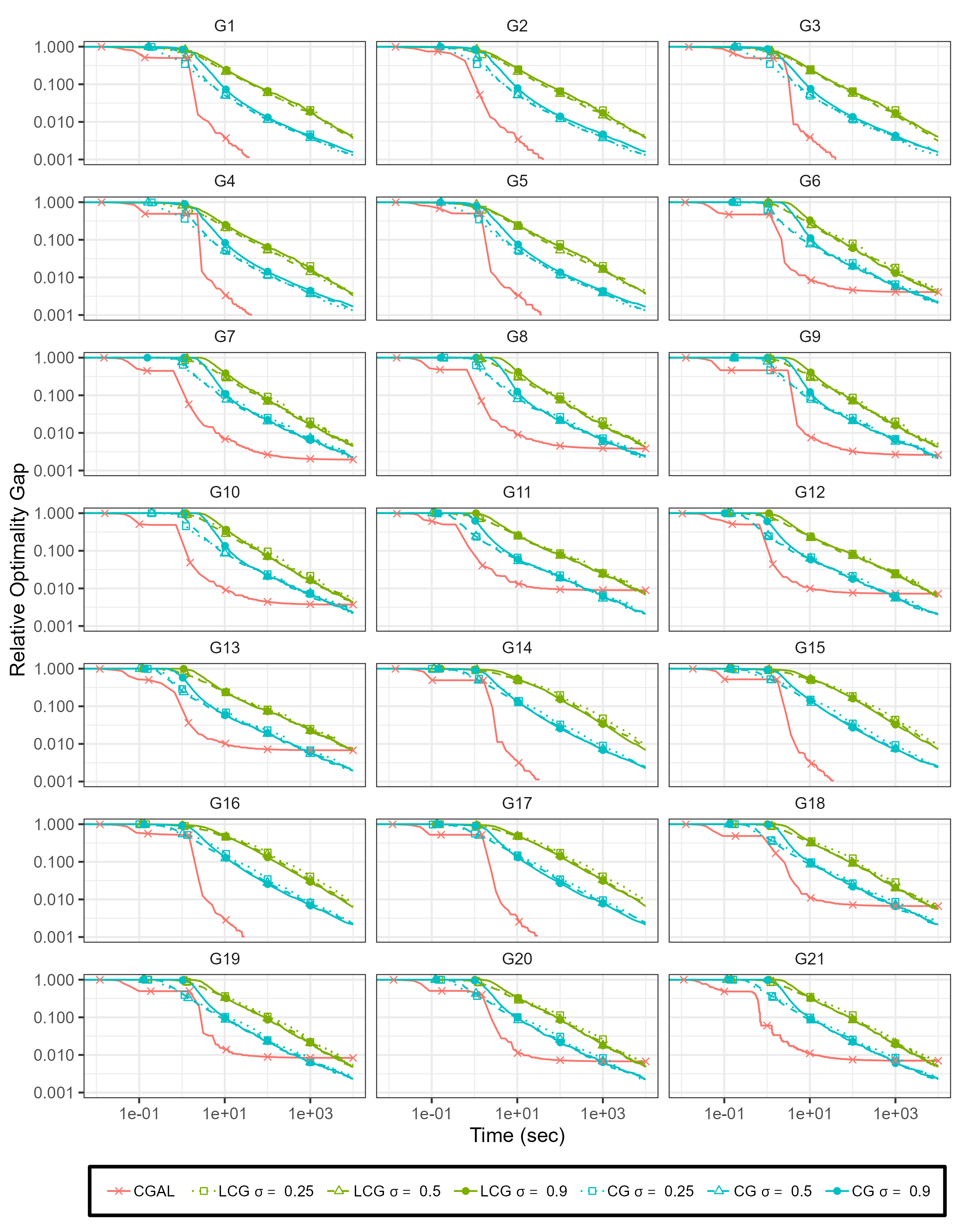}
   \caption{MaxCut datasets G1-G21 relative optimality gap vs. time.}
   \label{fig:MaxCut_Small_fVsTime}
\end{figure}
\begin{figure}[h!]
   \centering
   \includegraphics[scale=0.75]{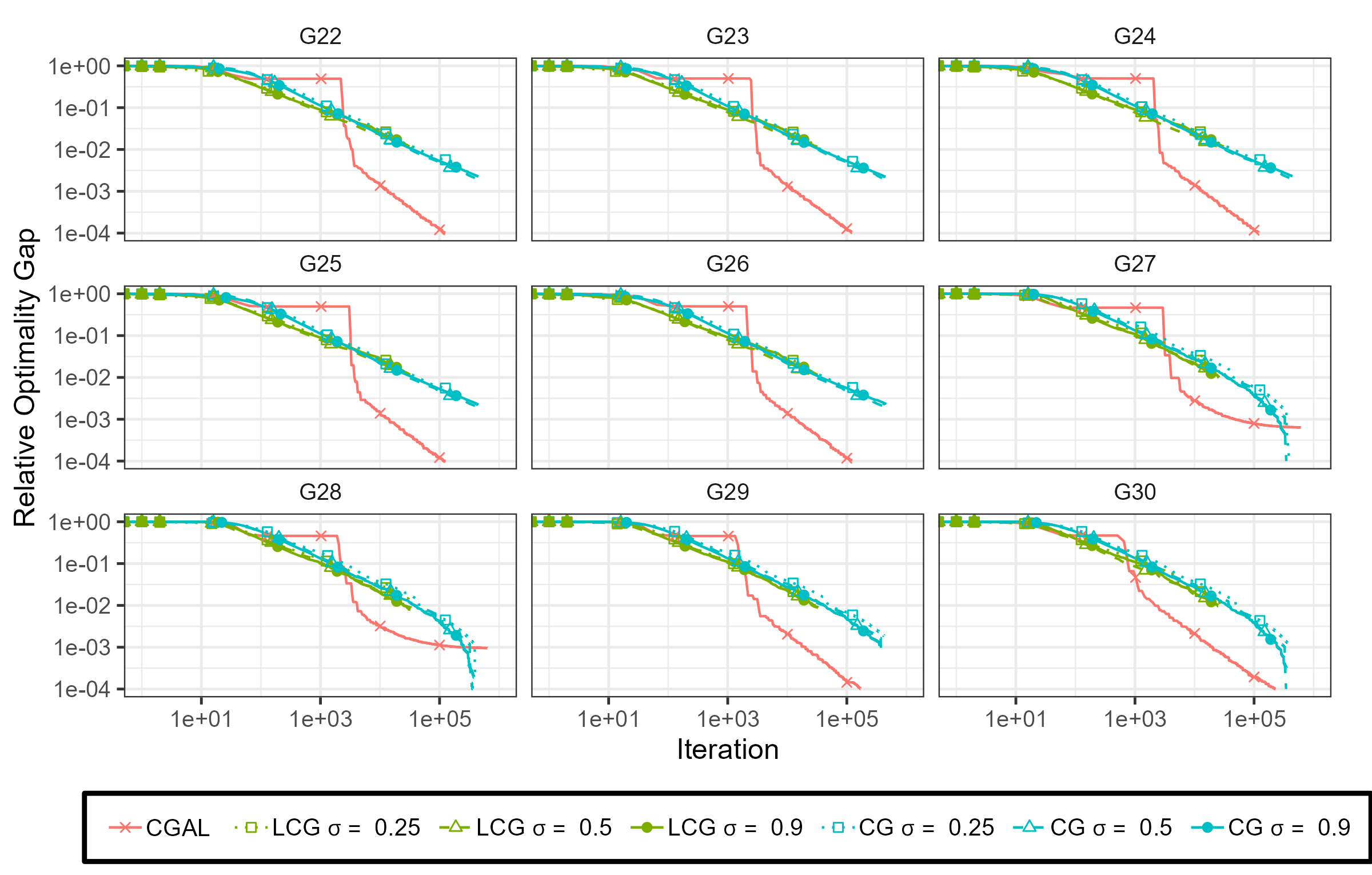}
   \caption{MaxCut datasets G22-G30 relative optimality gap vs. iteration.}\label{fig:MaxCut_Large_fVsIter}
\end{figure}
\begin{figure}[h!]
   \centering
   \includegraphics[scale=0.75]{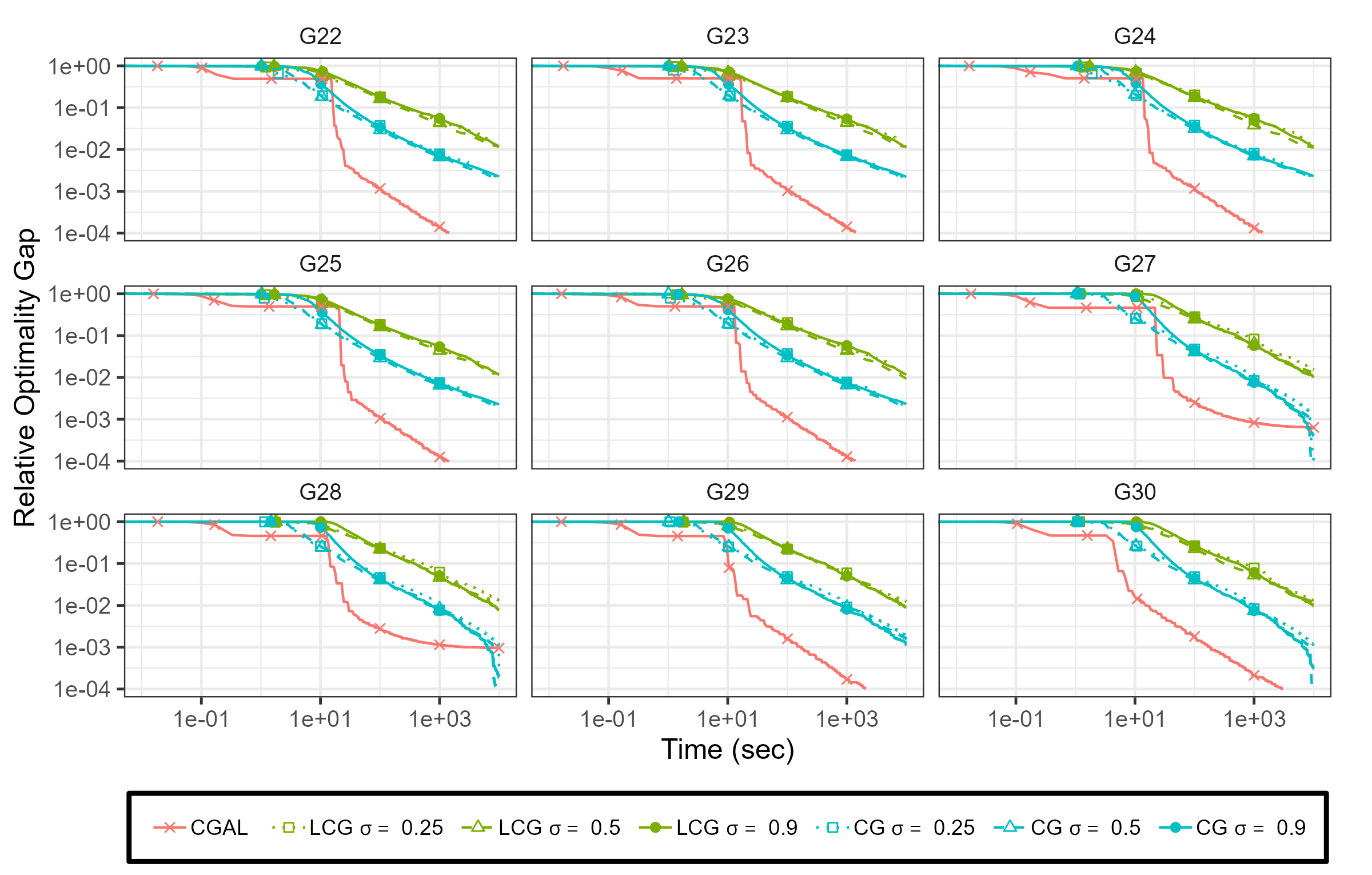}
   \caption{MaxCut datasets G22-G30 relative optimality gap vs. time.}\label{fig:MaxCut_Large_fVsTime}
   \vskip20pt
\end{figure}

\begin{table}[H]
   \caption{Performance for datasets \sh{G1-G18 at specific time points. For each data set, each time point, and each method the relative optimality gap (Opt. Gap), the iteration count (Iter) and for $\CGAL$ the relative feasibility gap (Feas. Gap) of the solution with the best objective value obtained until that time.  The best relative optimality gap for each row is marked in bold.}}
   \label{tbl:max_cut1}
   \centering
   \setlength{\tabcolsep}{1.5pt}
   \resizebox{\textwidth}{!}{
	   \begin{tabular}{lr|rrr|rr|rr|rr|rr|rr|rr}
		   \hline
		   Data & Time & \multicolumn{3}{|c|}{CGAL}  & \multicolumn{2}{|c|}{LCG $\sigma=$0.25} & \multicolumn{2}{|c|}{LCG $\sigma=$0.5} & \multicolumn{2}{|c|}{LCG $\sigma=$0.9} & \multicolumn{2}{|c|}{CG $\sigma=$0.25} & \multicolumn{2}{|c|}{CG $\sigma=$0.5} & \multicolumn{2}{|c}{CG $\sigma=$0.9}\\
		   Set & (sec)& \multicolumn{1}{|l}{Opt.}   & \multicolumn{1}{l}{Feas.}  & \multicolumn{1}{l|}{Iter}  & \multicolumn{1}{|l}{Opt.}  &  \multicolumn{1}{l|}{Iter} & \multicolumn{1}{|l}{Opt.}  & \multicolumn{1}{l|}{Iter} & \multicolumn{1}{|l}{Opt.}  & \multicolumn{1}{l|}{Iter}  & \multicolumn{1}{|l}{Opt.}   & \multicolumn{1}{l|}{Iter}  & \multicolumn{1}{|l}{Opt.}  & \multicolumn{1}{l|}{Iter}  & \multicolumn{1}{|l}{Opt.}   & \multicolumn{1}{l}{Iter}  \\
			& & \multicolumn{1}{|l}{Gap \%}& \multicolumn{1}{l}{Gap \%}&   & \multicolumn{1}{|l}{Gap \%}&  & \multicolumn{1}{|l}{Gap \%}&  & \multicolumn{1}{|l}{Gap \%}&  & \multicolumn{1}{|l}{Gap \%}&   & \multicolumn{1}{|l}{Gap \%}&   & \multicolumn{1}{|l}{Gap \%}&   \\
		   \hline
		   \hline
		   \multirow{3}{*}{G1} &  100 & {\bf 0.0452} & 0.0427 &  21685 & 5.3256 &  1650 & 5.1013 &  1592 & 6.3882 &  1494 & 1.1139 &  19046 & 1.1743 &  18468 & 1.2960 &  17742 \\
		   & 1000 & {\bf 0.0100} & 0.0092 & 144575 & 1.8670 & 17024 & 1.5269 & 16930 & 1.7090 & 16589 & 0.3782 & 126215 & 0.3774 & 124709 & 0.4113 & 127334 \\
		     & 10000 & \bf{0.0010} & 0.0010 &  888035 & 0.4146 & 134546 & 0.3628 & 141906 & 0.3890 & 154157 & 0.1312 &  786346 & 0.1322 &  728109 & 0.1565 &  730861 \\   \hline
		   \multirow{3}{*}{G2} &  100 & {\bf 0.0443} & 0.0416 &  20514 & 5.1990 &  1673 & 4.9872 &  1668 & 6.1180 &  1614 & 1.1186 &  20079 & 1.1583 &  17849 & 1.3328 &  17078 \\
		   & 1000 & {\bf 0.0076} & 0.0074 & 135864 & 1.8878 & 17695 & 1.5429 & 17554 & 1.5702 & 17592 & 0.3857 & 136275 & 0.3675 & 120983 & 0.4254 & 121586 \\
		 & 10000 & {\bf 0.0013} & 0.0012 &  816010 & 0.3419 & 150126 & 0.3133 & 153430 & 0.4046 & 148095 & 0.1317 &  763621 & 0.1463 &  687932 & 0.1571 &  715246\\
         \hline
		   \multirow{3}{*}{G3} &  100 & {\bf 0.0641} & 0.0622 &  21664 & 5.9692 &  1416 & 5.2227 &  1629 & 6.1860 &  1585 & 1.1050 &  20425 & 1.1555 &  17401 & 1.3028 &  17647 \\
		   & 1000 & {\bf 0.0074} & 0.0071 & 141485 & 1.8837 & 17015 & 1.3351 & 17410 & 1.5635 & 17266 & 0.4115 & 124686 & 0.3536 & 120982 & 0.4347 & 116904 \\
            & 10000 & {\bf 0.0008} & 0.0008 &  892189 & 0.3171 & 151626 & 0.3283 & 156878 & 0.3777 & 156813 & 0.1326 &  763597 & 0.1405 &  649211 & 0.1665 &  669949 \\ 
		\hline
		   \multirow{3}{*}{G4} &  100 & {\bf 0.0686} & 0.0666 &  19922 & 5.6793 &  1487 & 5.7132 &  1355 & 6.9583 &  1208 & 1.0994 &  19435 & 1.1341 &  18535 & 1.2752 &  17574 \\
		   & 1000 & {\bf 0.0077} & 0.0073 & 133750 & 2.1315 & 13938 & 1.4965 & 14001 & 1.7269 & 14378 & 0.3918 & 124072 & 0.3522 & 125967 & 0.4050 & 123681 \\
            & 10000 & {\bf 0.0008} & 0.0008 &  892189 & 0.3171 & 151626 & 0.3283 & 156878 & 0.3777 & 156813 & 0.1326 &  763597 & 0.1405 &  649211 & 0.1665 &  669949 \\ 
		   \hline
		   \multirow{3}{*}{G5} &  100 & {\bf 0.0668} & 0.0644 &  20334 & 5.8090 &  1509 & 5.7594 &  1379 & 6.8983 &  1338 & 1.1164 &  20254 & 1.1668 &  16756 & 1.2883 &  17533 \\
		   & 1000 & {\bf 0.0125} & 0.0118 & 135291 & 2.0338 & 14165 & 1.3668 & 14650 & 1.7768 & 14696 & 0.4197 & 124738 & 0.3487 & 116404 & 0.4222 & 118817 \\
           & 10000 & {\bf 0.0008} & 0.0007 &  888095 & 0.3882 & 137871 & 0.3553 & 143148 & 0.3893 & 146200 & 0.1322 &  772666 & 0.1335 &  714684 & 0.1643 &  694540 \\
              \hline
		   \multirow{3}{*}{G6} &  100 & {\bf 0.4803} & 0.0682 &  19842 & 8.2189 &  1541 & 7.2658 &  1426 & 7.0960 &  1293 & 2.5433 &  16239 & 2.0011 &  14659 & 2.1106 &  12283 \\
		   & 1000 & {\bf 0.4179} & 0.0100 & 133015 & 1.8849 & 16247 & 1.7304 & 14630 & 1.5365 & 14492 & 0.6772 & 109973 & 0.5733 &  94957 & 0.6847 &  86027 \\
           & 10000 & 0.4078 & 0.0009 &  937170 & 0.4791 & 172357 & 0.4163 & 177925 & 0.3755 & 163790 & {\bf 0.2108} &  720098 & 0.2232 &  581993 & 0.2106 &  566692 \\ 
		   \hline
		   \multirow{3}{*}{G7} &  100 & {\bf 0.2743} & 0.0735 &  19832 & 9.9632 &  1282 & 7.2801 &  1206 & 8.0242 &  1129 & 2.2855 &  16331 & 1.9467 &  13740 & 2.1709 &  12006 \\
		   & 1000 & {\bf 0.2068} & 0.0110 & 133207 & 2.1909 & 13391 & 1.7102 & 12885 & 1.7527 & 12644 & 0.6705 & 113963 & 0.6946 &  91575 & 0.6064 &  89294 \\
           & 10000 & {\bf 0.1959} & 0.0009 &  813946 & 0.4959 & 142268 & 0.4302 & 129157 & 0.4549 & 124502 & 0.2127 &  657896 & 0.2206 &  525425 & 0.2223 &  510794 \\ 
   		   \hline
		   \multirow{3}{*}{G8} &  100 & {\bf 0.4518} & 0.0627 &  20101 & 8.5605 &  1370 & 7.5897 &  1277 & 7.9211 &  1179 & 2.3751 &  15179 & 2.1037 &  13710 & 2.1617 &  12686 \\
		   & 1000 & {\bf 0.3945} & 0.0102 & 134291 & 2.0238 & 13734 & 1.7576 & 13463 & 1.6682 & 13493 & 0.7010 & 104772 & 0.5873 &  92671 & 0.6569 &  84183 \\
           & 10000 & 0.3839 & 0.0009 &  820939 & 0.5342 & 137709 & 0.4611 & 133831 & 0.4212 & 129272 & {\bf 0.2134} &  566425 & 0.2330 &  566242 & 0.2457 &  480991 \\ 
		   \hline
		   \multirow{3}{*}{G9} &  100 &  {\bf 0.3402} & 0.0741 &  20486 & 8.6602 &  1454 & 6.8707 &  1402 & 6.7527 &  1364 & 2.5177 &  17074 & 2.0543 &  14883 & 2.1253 &  12280 \\
		   & 1000 & {\bf 0.2705} & 0.0105 & 137199 & 1.9809 & 15167 & 1.7249 & 14672 & 1.5419 & 14895 & 0.6765 & 113442 & 0.5708 & 101335 & 0.6687 &  78602 \\
            & 10000 & 0.2603 & 0.0011 &  839573 & 0.5128 & 148527 & 0.4597 & 141307 & 0.4304 & 137229 & {\bf 0.2126} &  578638 & 0.2375 &  561210 & 0.2158 &  492155 \\ 
		   \hline
		   \multirow{3}{*}{G10} &  100 & {\bf 0.4693} & 0.0931 &  20176 & 9.1327 &  1365 & 7.0820 &  1293 & 7.5663 &  1210 & 2.3524 &  16086 & 2.0193 &  13699 & 2.1163 &  12605 \\
		   & 1000 & {\bf 0.3829} & 0.0118 & 134447 & 2.0525 & 13854 & 1.7098 & 13728 & 1.5771 & 13428 & 0.6489 &  95647 & 0.5940 &  95183 & 0.5856 &  89508 \\
           & 10000 & 0.3711 & 0.0008 &  843504 & 0.5452 & 140556 & 0.4503 & 139367 & 0.4136 & 136075 & {\bf 0.2046} &  572185 & 0.2221 &  525312 & 0.2224 &  502739 \\ 
		   \hline
		   \multirow{3}{*}{G11} &  100 & {\bf 0.9391} & 0.0293 &  52396 & 7.6742 &  2344 & 6.8577 &  2332 & 7.4084 &  2073 & 2.1040 &  42363 & 1.8304 &  36848 & 1.7089 &  39655 \\
		   & 1000 & 0.9021 & 0.0045 & 343271 & 2.3392 & 25386 & 1.9903 & 24475 & 2.0117 & 24162 & 0.6135 & 361872 & 0.5343 & 341395 & {\bf 0.5732} & 309567 \\
           & 10000 & 0.8946 & 0.0007 & 1823694 & 0.7343 & 196738 & 0.7811 & 193144 & 0.6929 & 187647 & {\bf 0.2105} & 2176005 & 0.2228 & 1945371 & {\bf 0.2105} & 1813119 \\ 
		   \hline
		   \multirow{3}{*}{G12} &  100 & {\bf 0.7673} & 0.0295 &  52859 & 8.1471 &  2120 & 7.0435 &  2302 & 7.3282 &  2236 & 2.1726 &  38919 & 1.8728 &  37346 & 1.7495 &  37614 \\
		   & 1000 & 0.7294 & 0.0033 & 342496 & 2.4091 & 24045 & 2.1357 & 24512 & 2.3042 & 23901 & 0.6337 & 325029 & {\bf 0.5363} & 314001 & 0.5735 & 276203 \\
           & 10000 & 0.7236 & 0.0007 & 1901444 & 0.7065 & 228102 & 0.5723 & 222533 & 0.6089 & 205372 & {\bf 0.2002} & 2311377 & 0.2094 & 2029389 & 0.2037 & 1692541 \\ 
		   \hline
		   \multirow{3}{*}{G13} &  100 & {\bf0.7167} & 0.0260 &  50953 & 7.5696 &  2414 & 7.2154 &  1830 & 6.6205 &  2339 & 2.1816 &  42733 & 1.9334 &  37760 & 1.8089 &  36572 \\
		   & 1000 & 0.6846 & 0.0050 & 332950 & 2.4416 & 24491 & 2.1377 & 23816 & 1.9294 & 25070 & 0.6357 & 343063 & {\bf 0.5274} & 312566 & 0.5570 & 279764 \\
           & 10000 & 0.6775 & 0.0005 & 1843586 & 0.7297 & 211955 & 0.7435 & 204609 & 0.6543 & 195695 & 0.1995 & 1887493 & 0.2062 & 1895654 & {\bf 0.1948} & 1547687 \\ 
		   \hline
		   \multirow{3}{*}{G14} &  100 & {\bf 0.0488} & 0.0401 &  36050 & 21.069
           &  1686 & 16.422
           &  1764 & 14.932
           &  1636 & 2.9552 &  32769 & 2.6901 &  30857 & 2.3423 &  28826 \\
		   & 1000 & {\bf 0.0062} & 0.0051 & 247958 & 4.4529 & 18672 & 3.5753 & 18026 & 3.3863 & 17441 & 0.8246 & 243773 & 0.7049 & 220892 & 0.6569 & 206482 \\
           & 10000 & {\bf 0.0010} & 0.0006 & 1653247 & 0.9661 & 175901 & 0.7821 & 173581 & 0.6847 & 186339 & 0.2270 & 1629850 & 0.2462 & 1237850 & 0.2253 & 1212944 \\ 
		   \hline
		   \multirow{3}{*}{G15} &  100 & {\bf 0.0495} & 0.0411 &  36866 & 18.219
           &  2085 & 15.498
           &  2061 & 14.534
           &  2039 & 3.2756 &  35930 & 2.7522 &  29819 & 2.5041 &  27217 \\
		   & 1000 & {\bf 0.0083} & 0.0064 & 251117 & 4.1913 & 22101 & 3.2621 & 21712 & 3.0489 & 20963 & 0.8459 & 255448 & 0.7324 & 225802 & 0.6674 & 204773 \\
            & 10000 & {\bf 0.0011} & 0.0007 & 1579597 & 0.9704 & 196785 & 0.7993 & 192710 & 0.7211 & 183253 & 0.2433 & 1550812 & 0.2374 & 1314984 & 0.2514 & 1151339 \\ 
		   \hline
		   \multirow{3}{*}{G16} &  100 & {\bf 0.0365} & 0.0287 &  41046 & 16.989
           &  2178 & 14.670
           &  2150 & 13.058
           &  2034 & 2.7754 &  33679 & 2.5751 &  31471 & 2.3225 &  30479 \\
		   & 1000 &  {\bf 0.0061} & 0.0049 & 271597 & 3.9114 & 22674 & 3.1580 & 22386 & 2.8169 & 21605 & 0.7870 & 260408 & 0.7366 & 228063 & 0.7303 & 222851 \\
            & 10000 & {\bf 0.0010} & 0.0008 & 1679075 & 0.8573 & 224493 & 0.6819 & 227033 & 0.6209 & 224707 & 0.2386 & 1499401 & 0.2269 & 1423609 & 0.2134 & 1430100 \\ 
		   \hline
		   \multirow{3}{*}{G17} &  100 & {\bf 0.0518} & 0.0434 &  37335 & 20.096
           &  1707 & 17.255
           &  1757 & 14.974
           &  1703 & 3.2235 &  33611 & 2.7927 &  29911 & 2.6083 &  26684 \\
		   & 1000 & {\bf 0.0090} & 0.0075 & 253376 & 4.5791 & 18601 & 3.6161 & 18267 & 3.4578 & 17899 & 0.8749 & 243142 & 0.7551 & 216924 & 0.6898 & 205937 \\
            & 10000 & {\bf 0.0009} & 0.0006 & 1650749 & 0.9701 & 232319 & 0.7884 & 226754 & 0.6712 & 220683 & 0.2417 & 1537197 & 0.2402 & 1352397 & 0.2139 & 1371570 \\ 
		   \hline
		   \multirow{3}{*}{G18} &  100 & {\bf 0.7302} & 0.0560 &  38999 & 11.727 
           &  1795 & 10.971 
           &  1780 & 9.5330 &  1586 & 2.4057 &  24255 & 2.0556 &  22388 & 2.0462 &  20555 \\
		   & 1000 & 0.6753 & 0.0084 & 260520 & 2.7593 & 18331 & 2.1365 & 18265 & 2.1493 & 17040 & 0.7666 & 185915 & 0.6925 & 159283 & {\bf 0.6533} & 141416 \\
            & 10000 & 0.6645 & 0.0006 & 1599182 & 0.6728 & 182254 & 0.4799 & 183154 & 0.5566 & 166432 & 0.2586 & 1061813 & 0.2166 &  962496 & {\bf 0.2154} &  879504 \\ 
		   \hline
   \end{tabular}}
\end{table}

\begin{table}[h!]
   \caption{Performance for datasets \sh{G19-G30 at specific time points. For each data set, each time point, and each method we record the best relative optimality gap (Opt. Gap), the iteration count (Iter) and for $\CGAL$ the relative feasibility gap (Feas. Gap) of the solution with the best objective value obtained until that time. The best relative optimality gap for each row is marked in bold.}}
   \label{tbl:max_cut2}
   \centering
   \setlength{\tabcolsep}{1.5pt}
   \resizebox{1\textwidth}{!}{
	   \begin{tabular}{lr|rrr|rr|rr|rr|rr|rr|rr}
		   \hline
		   Data & Time & \multicolumn{3}{|c|}{CGAL}  & \multicolumn{2}{|c|}{LCG $\sigma=$0.25} & \multicolumn{2}{|c|}{LCG $\sigma=$0.5} & \multicolumn{2}{|c|}{LCG $\sigma=$0.9} & \multicolumn{2}{|c|}{CG $\sigma=$0.25} & \multicolumn{2}{|c|}{CG $\sigma=$0.5} & \multicolumn{2}{|c}{CG $\sigma=$0.9}\\
		   Set & (sec)& \multicolumn{1}{|l}{Opt.}   & \multicolumn{1}{l}{Feas.}  & \multicolumn{1}{l|}{Iter}  & \multicolumn{1}{|l}{Opt.}  &  \multicolumn{1}{l|}{Iter} & \multicolumn{1}{|l}{Opt.}  & \multicolumn{1}{l|}{Iter} & \multicolumn{1}{|l}{Opt.}  & \multicolumn{1}{l|}{Iter}  & \multicolumn{1}{|l}{Opt.}   & \multicolumn{1}{l|}{Iter}  & \multicolumn{1}{|l}{Opt.}  & \multicolumn{1}{l|}{Iter}  & \multicolumn{1}{|l}{Opt.}   & \multicolumn{1}{l}{Iter}  \\
		   & & \multicolumn{1}{|l}{Gap \%}&  \multicolumn{1}{l}{Gap \%}&   & \multicolumn{1}{|l}{Gap \%}&  &  \multicolumn{1}{|l}{Gap \%}&  &  Gap \%&  & \multicolumn{1}{|l}{Gap \%}&   & \multicolumn{1}{|l}{Gap \%}&   & \multicolumn{1}{|l}{Gap \%}&   \\
		   \hline
			   \hline
			   \multirow{3}{*}{G19} &  100 & {\bf 0.8967} & 0.0492 &  37083 & 12.059
               &  1630 & 10.580
               &  1541 & 10.506
               &  1527 & 2.3693 &  25356 & 2.2978 &  22770 & 2.0119 &  20054 \\
			   & 1000 & 0.8444 & 0.0066 & 254368 & 2.6314 & 16530 & 2.1796 & 16473 & 2.1604 & 15966 & 0.6124 & 162830 & 0.6152 & 137255 & {\bf 0.6637} & 129524 \\
               & 10000 & 0.8351 & 0.0006 & 1723323 & 0.5384 & 205382 & 0.5075 & 194687 & 0.4851 & 186757 & 0.2624 & 1026619 & 0.2506 &  842195 & {\bf 0.2320} &  756654 \\ 
			   \hline
			   \multirow{3}{*}{G20} &  100 & {\bf0.7363} & 0.0538 &  36281 & 12.956
               &  1486 & 10.386
               &  1593 & 10.453
               &  1462 & 2.8064 &  23999 & 2.4927 &  21325 & 2.2067 &  19130 \\
			   & 1000 & 0.6819 & 0.0062 & 251142 & 2.9914 & 15873 & 2.4176 & 15989 & 2.2335 & 16143 & 0.7705 & 162281 & 0.7034 & 144692 & {\bf 0.6373} & 132063 \\
			    & 10000 & 0.6742 & 0.0006 & 1716273 & 0.6061 & 208638 & 0.5587 & 206001 & 0.4746 & 187323 & 0.2475 & 1068889 & {\bf 0.2064} &  924839 & 0.2231 &  790836 \\ 
                \hline
   \multirow{3}{*}{G21} &  100 & {\bf 0.7840} & 0.0726 &  41830 & 10.288
   &  1935 & 7.7700 &  1998 & 8.3024 &  1955 & 2.4996 &  24785 & 2.0761 &  20738 & 2.0545 &  20794 \\
			   & 1000 & 0.7121 & 0.0102 & 275681 & 2.1759 & 20341 & 1.8593 & 21211 & 1.8629 & 20506 & 0.7668 & 160845 & {\bf 0.5995} & 139988 & 0.6157 & 136838 \\
                & 10000 & 0.7005 & 0.0008 & 1610239 & 0.4848 & 188165 & 0.4806 & 187510 & 0.4938 & 178700 & 0.2522 & 1016576 & {\bf 0.2235} &  914508 & 0.2284 &  824945 \\ 
			   \hline
			   \multirow{3}{*}{G22} &  100 & {\bf 0.1174} & 0.1139 &  12416 & 18.1438 &   300 & 16.5330 &   288 & 18.2411 &   248 & 3.7563 &   7042 & 2.9849 &   6751 & 3.4068 &   5745 \\ 
    & 1000 & {\bf 0.0193} & 0.0187 &  90981 & 4.7658 &  3319 & 4.3706 &  3215 & 5.5525 &  2941 & 0.7828 &  64426 & 0.6636 &  59780 & 0.7353 &  58167 \\ 
    & 10000 & {\bf 0.0022} & 0.0021 & 632496 & 1.0548 & 34970 & 1.1614 & 31913 & 1.1801 & 31293 & 0.2095 & 497773 & 0.1931 & 447623 & 0.2247 & 449915 \\  
		   \hline
		   \multirow{3}{*}{G23} &  100 & {\bf 0.1023} & 0.0980 &  12410 & 18.5561 &   292 & 17.3589 &   250 & 17.8255 &   256 & 3.6153 &   6927 & 2.9642 &   6491 & 3.3613 &   5583 \\ 
    & 1000 & {\bf 0.0149} & 0.0143 &  90915 & 4.7672 &  3275 & 4.3952 &  2940 & 5.3747 &  3069 & 0.7403 &  60120 & 0.6684 &  57688 & 0.7333 &  57294 \\ 
    & 10000 & {\bf 0.0021} & 0.0020 & 631480 & 1.0797 & 33094 & 1.1388 & 31910 & 1.1669 & 32278 & 0.2062 & 478154 & 0.1923 & 440393 & 0.2229 & 449699 \\
		   \hline
		   \multirow{3}{*}{G24} &  100 & {\bf 0.1257} & 0.1222 &  12342 & 20.1997 &   251 & 17.6862 &   259 & 18.7144 &   242 & 3.8042 &   6856 & 3.1557 &   6188 & 3.4891 &   5526 \\ 
    & 1000 & {\bf 0.0164} & 0.0157 &  90383 & 5.1886 &  2825 & 3.8498 &  2957 & 5.5847 &  2858 & 0.8075 &  61881 & 0.6991 &  55330 & 0.7356 &  57577 \\ 
    & 10000 & {\bf 0.0020} & 0.0018 & 624799 & 1.2467 & 30859 & 1.1084 & 31608 & 1.1785 & 30999 & 0.2195 & 475491 & 0.2035 & 412481 & 0.2309 & 438368 \\ 
		   \hline
		   \multirow{3}{*}{G25} &  100 & {\bf 0.1057} & 0.1012 &  12533 & 18.5114 &   300 & 16.3371 &   284 & 17.9667 &   256 & 3.5888 &   6971 & 2.9573 &   6687 & 3.4011 &   5701 \\ 
    & 1000 & {\bf 0.0135} & 0.0128 &  91948 & 4.7165 &  3368 & 4.4134 &  3115 & 5.4227 &  3022 & 0.7561 &  63518 & 0.6510 &  60912 & 0.7309 &  58963 \\ 
    & 10000 &  {\bf 0.0018} & 0.0017 & 637281 & 1.0292 & 34064 & 1.1919 & 32793 & 1.1593 & 32361 & 0.2028 & 491183 & 0.1930 & 448653 & 0.2257 & 452544 \\ 
		   \hline
		   \multirow{3}{*}{G26} &  100 & {\bf 0.1375} & 0.1337 &  12736 & 20.6034 &   248 & 17.1955 &   269 & 18.7351 &   245 & 3.7270 &   6860 & 2.9696 &   6664 & 3.4397 &   5703 \\ 
    & 1000 & {\bf 0.0151} & 0.0145 &  92871 & 4.9788 &  3025 & 4.4023 &  3025 & 5.7863 &  2932 & 0.7733 &  62802 & 0.6581 &  60555 & 0.7417 &  59907 \\ 
           & 10000 & {\bf 0.0022} & 0.0021 & 635639 & 1.0644 & 32298 & 0.9407 & 32064 & 1.1557 & 32364 & 0.2074 & 490930 & 0.1957 & 442197 & 0.2317 & 461381 \\ 
		   \hline
		   \multirow{3}{*}{G27} &  100 & {\bf 0.3024} & 0.2283 &  11530 & 28.4263 &   195 & 25.0943 &   214 & 27.2575 &   174 & 4.7957 &   6361 & 4.0809 &   6127 & 4.4816 &   4712 \\ 
    & 1000 & {\bf 0.0877} & 0.0246 &  82946 & 8.2273 &  2240 & 6.4036 &  2407 & 5.7783 &  2351 & 0.8400 &  55681 & 0.8083 &  52249 & 0.7602 &  47704 \\ 
    & 10000 & 0.0646 & 0.0031 & 609686 & 1.5543 & 25336 & 0.9938 & 25525 & 0.9926 & 25790 & 0.0102 & 390126 & {\bf 0.0000} & 378324 & 0.0409 & 344793 \\ 
		   \hline
		   \multirow{3}{*}{G28} &   100 & {\bf 0.3037} & 0.1984 &  12418 & 23.2191 &   281 & 22.2085 &   280 & 23.0712 &   226 & 4.6768 &   6415 & 4.1060 &   6054 & 4.3763 &   4992 \\ 
    & 1000 & {\bf 0.1194} & 0.0240 &  91115 & 6.2536 &  3044 & 4.6599 &  3113 & 5.0701 &  3019 & 0.7727 &  54325 & 0.8275 &  54898 & 0.7498 &  50886 \\ 
    & 10000 & 0.0964 & 0.0026 & 634335 & 1.3236 & 33006 & 0.8704 & 33326 & 0.7890 & 32323 & 0.0216 & 393334 & {\bf 0.0000} & 397716 & 0.0190 & 377339 \\ 
		   \hline
		   \multirow{3}{*}{G29} &  100 & {\bf 0.1737} & 0.1651 &  12428 & 22.335 &   291 & 22.374 &   276 & 22.594 &   244 & 4.8811 &   6431 & 4.1196 &   6132 & 4.4864 &   4976 \\ 
    & 1000 & {\bf 0.0238} & 0.0216 &  91157 & 5.9446 &  3169 & 5.2525 &  3145 & 5.0132 &  3068 & 0.9131 &  57929 & 0.9055 &  50524 & 0.8096 &  51421 \\ 
            & 10000 & {\bf 0.0025} & 0.0022 & 637658 & 1.2430 & 33012 & 0.8851 & 33265 & 0.9036 & 32289 & 0.1865\% & 421074 & 0.0965 & 375627 & 0.1135 & 356157 \\ 
		   \hline
		   \multirow{3}{*}{G30} &  100 & {\bf 0.2294} & 0.2176 &  11582 & 26.6363 &   217 & 22.9201 &   208 & 26.8685 &   190 & 4.8414 &   6122 & 4.0918 &   6027 & 4.4921 &   4685 \\ 
    & 1000 & {\bf 0.0300} & 0.0254 &  86515 & 7.8141 &  2482 & 5.3362 &  2403 & 6.1525 &  2295 & 0.8334 &  53756 & 0.7499 &  51787 & 0.7489 &  47761 \\ 
        & 10000 & 0.0048 & 0.0017 & 616813 & 1.0134 & 27305 & 1.1942 & 24995 & 0.9684 & 24960 & 0.1127 & 388638 & {\bf 0.0027} & 364280 & 0.0314 & 353046 \\ 
	\hline
   \end{tabular}}
\end{table}


\subsection{Markov chain mixing rate}\label{sec:Mixing}
We next consider the problem of finding the fastest mixing rate of a Markov chain on a graph \cite{SunBoyDiacXia06}. In this problem, a symmetric Markov chain is defined on an undirected graph $G=(\{1,\ldots,n\},\scrE)$. Given $G$ and weights $d_{ij}$ for $\{i,j\}\in\scrE$, we are tasked with finding the transition rates $w_{ij}\geq 0$ for each $\{i,j\}\in\scrE$ with weighted sum smaller than 1, that result in the fastest mixing rate. We assume that $\sum_{ij}d^{2}_{ij}=n^{2}$. The mixing rate is given by the second smallest eigenvalue of the graph's Laplacian matrix, described as
$$L(\bw)_{ij}=\begin{cases}
 \sum_{j:\{i,j\}\in\scrE} w_{ij} & j=i\\
 -w_{ij} & \{i,j\}\in\scrE\\
 0 &\text{otherwise}.
\end{cases}.$$
From $L\1=0$ it follows that $\frac{1}{n}\1$ is a stationary distribution of the process. From \cite{DiaStr91}, we have the following bound on the total variation distance of the time $t$ distribution of the Markov process and the stationary distribution
 $
 \sup_{\pi}\norm{\pi \bP(t)-\frac{1}{n}\1}_{\text{TV}}\leq\frac{1}{2}\sqrt{n}e^{-\lambda t}.
 $
 The quantity $T_{\text{mix}}=\lambda$ is called the \emph{mixing time}, and $\lambda\equiv\lambda_{2}(L(\bw))$ is the second-largest eigenvalue of the graph Laplacian $L(\bw)$. To bound this eigenvalue, we follow the strategy laid out in \cite{SunBoyDiacXia06}. Thus, the problem can be written as
\begin{align*}
 \max_{\bw}\quad &\lambda_2(L(\bw))\\
 \text{s.t.}\quad &\sum_{\{i,j\}\in\scrE} d_{ij}^2w_{ij}\leq 1\\
 \qquad  & \bw\geq 0.
\end{align*}
This problem can also alternatively be formulated as
\begin{align*}
 \min_{\bw}\quad &\sum_{\{i,j\}\in\scrE} d_{ij}^2w_{ij}\\
 \text{s.t.}\quad &\lambda_2(L(\bw)) \geq 1\\
 \qquad  & \bw\geq 0.
\end{align*}
Due to the properties of the Laplacian, the first constraint can be reformulated as
\begin{align}
 \min_{\bw}\quad &\sum_{\{i,j\}\in\scrE} d_{ij}^2w_{ij}\\
 \text{s.t.}\quad &L(\bw)\succeq \bI_{n\times n}-\frac{1}{n} \1 \1^\top\\
 \qquad  & \bw\geq 0.
\end{align}
The dual problem is then given by
\begin{equation}\label{eq:SDPMixing}
 \begin{split}
	\max_{\bX\in \symm^n_{+}}\quad &\inner{\bI_{n\times n}-\frac{1}{n} \1 \1^\top ,\bX}\\
	\text{s.t.}\quad &X_{ii}+X_{jj}-2X_{ij}\leq d_{ij}^2\qquad \{i,j\}\in\scrE
 \end{split}
\end{equation}
To obtain an SDP within our convex programming model, we combine arguments from \cite{SunBoyDiacXia06} and \cite{AroHazKal05}. Let $\bX$ be a feasible point for \eqref{eq:SDPMixing}. Then, there exists a $n\times n$ matrix $\bV$ such that $\bX=\bV\bV^{\top}$. It is easy to see that multiplying each row of the matrix $\bV^{\top}=[\bv_{1},\ldots,\bv_{n}]$ with an orthonormal matrix does not change the feasibility of the candidate solution. Moreover $X_{ij}=\bv_{i}^{\top}\bv_{j}$ for all $1\leq i,j\leq n$, so that $X_{ii}+X_{jj}-2X_{ij}=\norm{\bv_{i}-\bv_{j}}^{2}_{2}$. Since shifting each vector $\bv_i$ by a constant vector $\bv$ would not affect the objective or constraints, we can, without loss of generality, set $\bv_1=\0$.
This gives the equivalent optimization problem
\begin{equation}\label{eq:L2Embedding}
 \begin{split}
	\max_{\bv_{1},\ldots,\bv_{n}}\quad & \sum_{i=1}^{n}\norm{\bv_{i}}^{2}\\
	\text{s.t.}\quad & \norm{\bv_{i}-\bv_{j}}^{2}\leq d^{2}_{ij}\qquad\{i,j\}\in\scrE,\\
	&\bv_{i}\in\Rn, \bv_1=\0.
 \end{split}
\end{equation}
This is the geometric dual derived in \cite{SunBoyDiacXia06}, which is strongly connected to the geometric embedding of a graph in the plane \cite{GorHelWap08}.
Thus, setting again $\bX=\bV\bV^\top$, we therefore obtain that $X_{1j}=X_{j1}=0$ for all $i=1,\ldots,n$ thus reducing the dimension of the problem, and \eqref{eq:SDPMixing1} can be reformulated as follows:
\begin{equation}\label{eq:SDPMixing1}
 \begin{split}
	\max_{\bX\in \symm^{n-1}_{+}}\quad &\inner{\bI_{{n-1}\times {n-1}}-\frac{1}{n} \1 \1^\top ,\bX}\\
	\text{s.t.}\quad &X_{(i-1)(i-1)}+X_{(j-1)(j-1)}-2X_{(i-1)(j-1)}\leq d_{ij}^2\qquad \{i,j\}\in\scrE, i,j>1\\
	& X_{(j-1)(j-1)}\leq d_{1j}^2\qquad\qquad\qquad\qquad\qquad\qquad \{1,j\}\in \scrE
 \end{split}
\end{equation}
Moreover, since $\bX\in\symm^{n-1}_{+}$, it follows that $2|X_{ij}|\geq -\sqrt{X_{ii}X_{jj}}$ and so for any $ \{i,j\}\in\scrE$ such that $i,j>1$,$\sqrt{X_{(j-1)(j-1)}}\leq \sqrt{X_{(i-1)(i-1)}}+\sqrt{d_{ij}}$. Therefore, since the graph is connected, we can recursively bound each diagonal element of $\bX$ and therefore the trace by a constant $\alpha$, and add the trace constraint $\setX=\{\bX\in\symm^{n}\vert \bX\succeq 0,\tr(\bX)\leq \alpha\}$ to the problem formulation without affecting the optimal values.
Furthermore, defining $\scrD:\symm^{n-1}\to\R^{\abs{\scrE}}$ by $$\scrD(\bX)_{\{i,j\}}=\begin{cases}X_{(i-1)(i-1)}+X_{(j-1)(j-1)}-2X_{(i-1)(j-1)}, & i,j>0,\\
  X_{(j-1)(j-1)}, & i=1.\end{cases}$$ Let $\setQ=\{\by\in\R^{\abs{\scrE}}\vert y_{i,j}\leq d^{2}_{ij},\{i,j\}\in\scrE\}$ and $\setK=\{(y,t)\in\R^{\abs{\scrE}}\times\R\vert\frac{1}{t}y\in\setQ,t>0\}$. This is a closed convex cone with logarithmically homogeneous barrier
\[
f(y,t)=-\sum_{\{i,j\}\in\scrE}\log(t d^{2}_{ij}-y_{i,j})=-\sum_{e\in\scrE}\log(d^{2}_{ij}-\frac{1}{t}y_{i,j})-\abs{\scrE}\log(t).
\]
This gives a logarithmically homogeneous barrier $F(\bX,t)=f(\scrP(\bX,t))$, where $\scrP(\bX,t)=[\scrD(\bX);t]\in\R^{\abs{\scrE}}\times\R$. We thus arrive at a formulation of the form of problem \eqref{eq:Opt}, which reads explicitly as
\begin{equation}\label{eq:SDPMixing2}
  \begin{split}
	  \min_{\bX,t}\quad &\{g(\bX):=\inner{\bI_{n\times n}-\frac{1}{n} \1 \1^\top ,\bX}\}\\
	  \text{s.t.}\quad &\scrP(\bX,t)\in\setK\\
	  & \bX\in\setX,t=1.
  \end{split}
\end{equation}

\begin{table}[t]
 \centering
 \begin{tabular}{lll|r}
	\hline
	Dataset &
	\# Nodes &
	\# Edges & SDPT3 Value  \\
	\hline\hline
	1& 100 & 1000 & 15.62\\
	2 & 100 & 2000 & 7.93\\
	3 & 200 & 1000 & 72.32\\
	4 & 200 & 4000 & 17.84\\
	5 & 400 & 1000 & 388.52\\
	6 & 400 & 8000 & 38.37\\
	7 & 800 & 4000 & 333.25\\
	8 & 800 & 16000 & 80.03\\
	\hline
 \end{tabular}
 \caption{Mixing datasets characteristics.}
 \label{tbl:MixingDataSetsInfo}
\end{table}

We generated random connected undirected graphs of various sizes, and for each edge $\{i,j\}$ in the graph we generated a random $d_{ij}^2$ uniformly in $[0,1]$. Table~\ref{tbl:MixingDataSetsInfo} provides the size of each dataset, and the value obtained by solving the Mixing problem SDP using CVX with the SDPT3 solver.
In order to compare to the $\CGAL$ algorithm, we used the normalisation guidelines specified in \cite{YurTroFerUdeCev21}, so that the norm of the coefficients vector of each constraint will be equal to each other, and the norm of the constraint matrix and the objective matrix equals 1.

All datasets were run for both $\CG$ and $\LCG$ options with the following choice of parameters $\eta_0\in\{0.5\Omega_g,1\Omega_g,2\Omega_g\}$
and $\sigma\in\{0.9,0.5,0.25\}$. \sh{All methods were stopped when either they reached $10000$ iterations or $3600$ seconds running time, the earlier of the two conditions.}
Since $\CGAL$ does not generate feasible solutions, we computed the deviation from feasibility of its outputted solution $\bX^{\CGAL}$ by
$$\text{Relative Feasibility Gap}=\gamma(\bX^{\CGAL})\equiv\max_{\{i,j\}\in \scrE}\frac{d_{ij}^2-\scrD(\bX^{\CGAL})_{ij}}{d_{ij}^2}.$$ We then corrected the solution to a feasible solution $\tilde{\bX}^{\CGAL}=\bX^{\CGAL}/(\gamma(\bX^{\CGAL})+1)$, however in our numerical experiments it turned out that $\gamma$ was extremely large, which resulted in $\tilde{\bX}^{\CGAL}$ being very close to the zero matrix with very poor performance. Therefore, we used the following alternative iterative correction method: We set $\hat{\bX}^{\CGAL,1}=\bX^{\CGAL}$, and at each iteration $k$ we identify $\{i,j\}\in \scrE$ which leads to the maximum value for $\gamma(\hat{\bX}^{\CGAL,k})$ set $i^*=\argmax_{l\in\{i,j\}} X_{ll}$ and set
$$\hat{\bX}^{\CGAL,k+1}_{ij}=\begin{cases}
\hat{\bX}^{\CGAL,k+1}_{ij}, & i,j\neq i^*.\\
\hat{\bX}^{\CGAL,k+1}_{ij}/(\gamma(\hat{\bX}^{\CGAL,k})+1), & i = i^* \text{ or } j= i^*.
\end{cases}.$$
We terminate the algorithm when $\gamma(\hat{\bX}^{\CGAL,k})=0$. The value of the returned $\hat{\bX}^{\CGAL}$ is then computed. \sh{As in the previous section, one of our performance measures is the \emph{relative optimality gap} $\frac{\abs{g(\bX)-g(\tilde{\bX}^{\SDPT3})}}{\abs{g(\tilde{\bX}^{\SDPT3})}}$ 
between the (corrected) solution of each method and the $\SDPT3$ solution. Specifically, for each dataset $D$, method $M$ and parameter combination $P$, we record the objective value $g^{M,P,D}_i$ (or, for $\CGAL$, the corrected objective value) \gi{after iteration $i$}, the corresponding relative optimality gap $r^{M,P,D}_i=(g^{M,P,D}_i-g(\tilde{\bX}^{\SDPT3}))/g(\tilde{\bX}^{\SDPT3})$, the time $t^{M,P,D}_i$ \gi{elapsed since}
the start of the run and, only for $\CGAL$, the relative feasibility gap $\gamma^{\CGAL,D}_i$.}

Figures \ref{fig:Mixing_GapVsIter}-\ref{fig:Mixing_GapVsTime} illustrate the results for the tested $\sigma$ and $\eta_0$ values. \sh{Figure~\ref{fig:Mixing_GapVsIter} displays, for each iteration $i$, the best relative optimality gap $\min_{j\leq i}\{r^{M,P,D}_j\}$ up to that iteration. 
Figure~\ref{fig:Mixing_GapVsTime}  displays for each time $t$ the best relative optimality gap $\min_{j}\{r^{M,P,D}_i:t_j^{M,P,D}\leq t\}$ obtained up to time $t$.
Table \ref{tbl:MixingDataSetsResults} reports the performances of the tested methods at time points $t\in \{100,500,1000\}$ seconds, in order to examine their algorithmic behavior under different time constraints. Specifically, given a data set $D$, a method $M$, parameters $P$ and a time point $t$\gi{, for $\CG$ and $\LCG$} we compute $(j^*,P^*) =
\argmin_{j,P}\{r^{M,P,D}_\gi{j}\vert t_j^{M,P,D}\leq t\}$, so the table contains: the parameter values $P^*=(\eta_{0},\sigma)$, the corresponding objective function value (Obj. Value) $g_{j^*}^{M,P^*,D}$, the relative optimality gap (Opt. Gap) $r_j^{M,P,D}$, and the number of iterations (Iter) performed until that time, given by $\max\{j\vert t_{j}^{M,P^*,D}\leq t\}$. For $\CGAL$, since there is no parameter choice, we retrieve $j^*=\argmin_{j}\{r^{\CGAL,D}_i\vert t_j^{\CGAL,D}\leq t\}$, and the table records the corresponding objective value (Obj. Value) $g^{\CGAL,D}_{j^*}$, the relative optimal gap (Opt. Gap) $r^{\CGAL,D}_{j^*}$, the relative feasibility gap (Feas Gap) $\gamma^{\CGAL,D}_{j^*}$, as well as the number of iterations (Iter) given by $\max\{j\vert t_{j}^{\CGAL,D}\leq t\}$. The lowest relative optimality gap for each data set and time is marked in bold.}
   
\begin{figure}[h!]
  \centering
  \includegraphics[height=0.91\textheight]{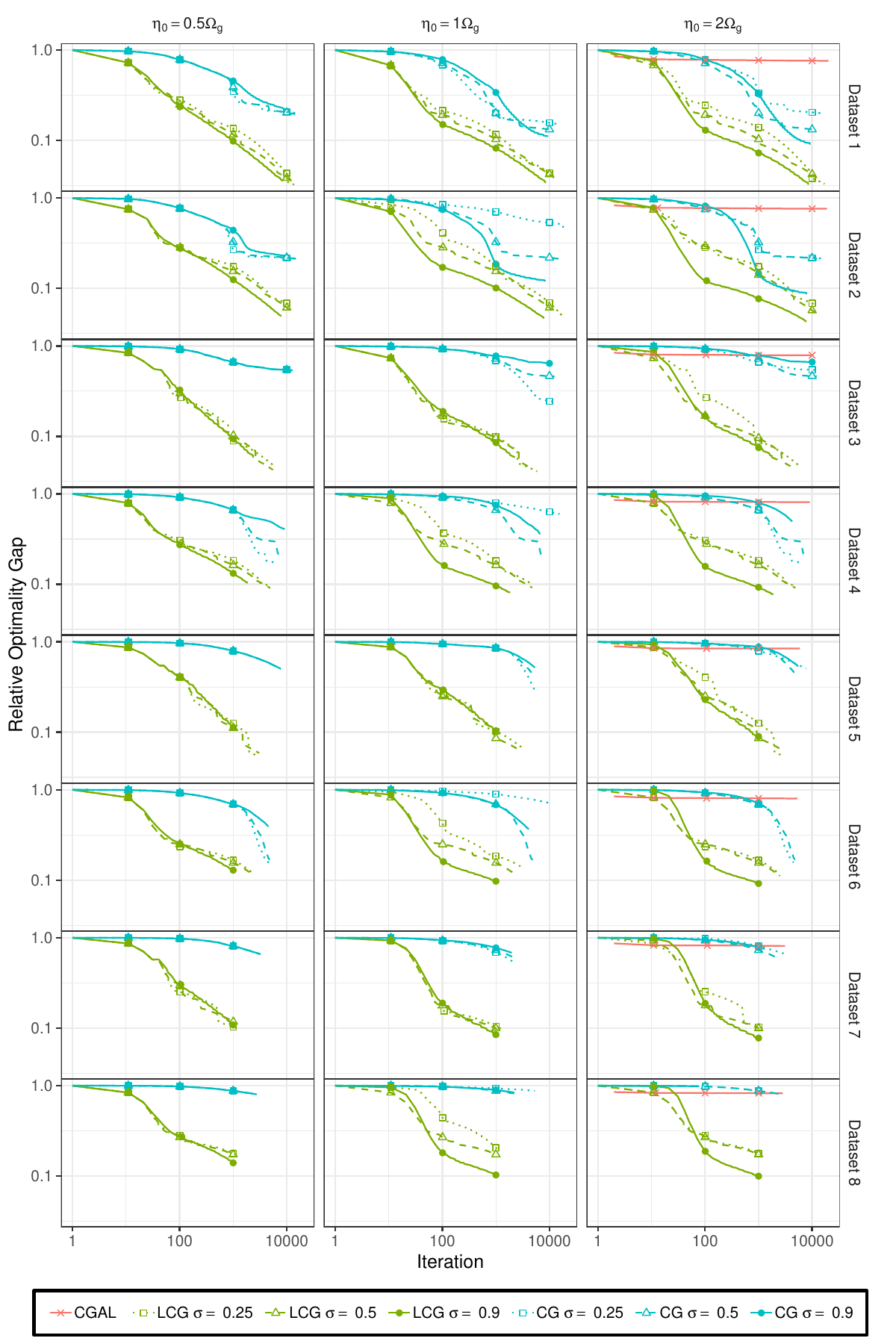}
  \caption{Mixing datasets relative \sh{optimality} gap vs. iteration for various parameter choices.}\label{fig:Mixing_GapVsIter}
\end{figure}

\begin{figure}[h!]
  \centering
  \includegraphics[height=0.91\textheight]{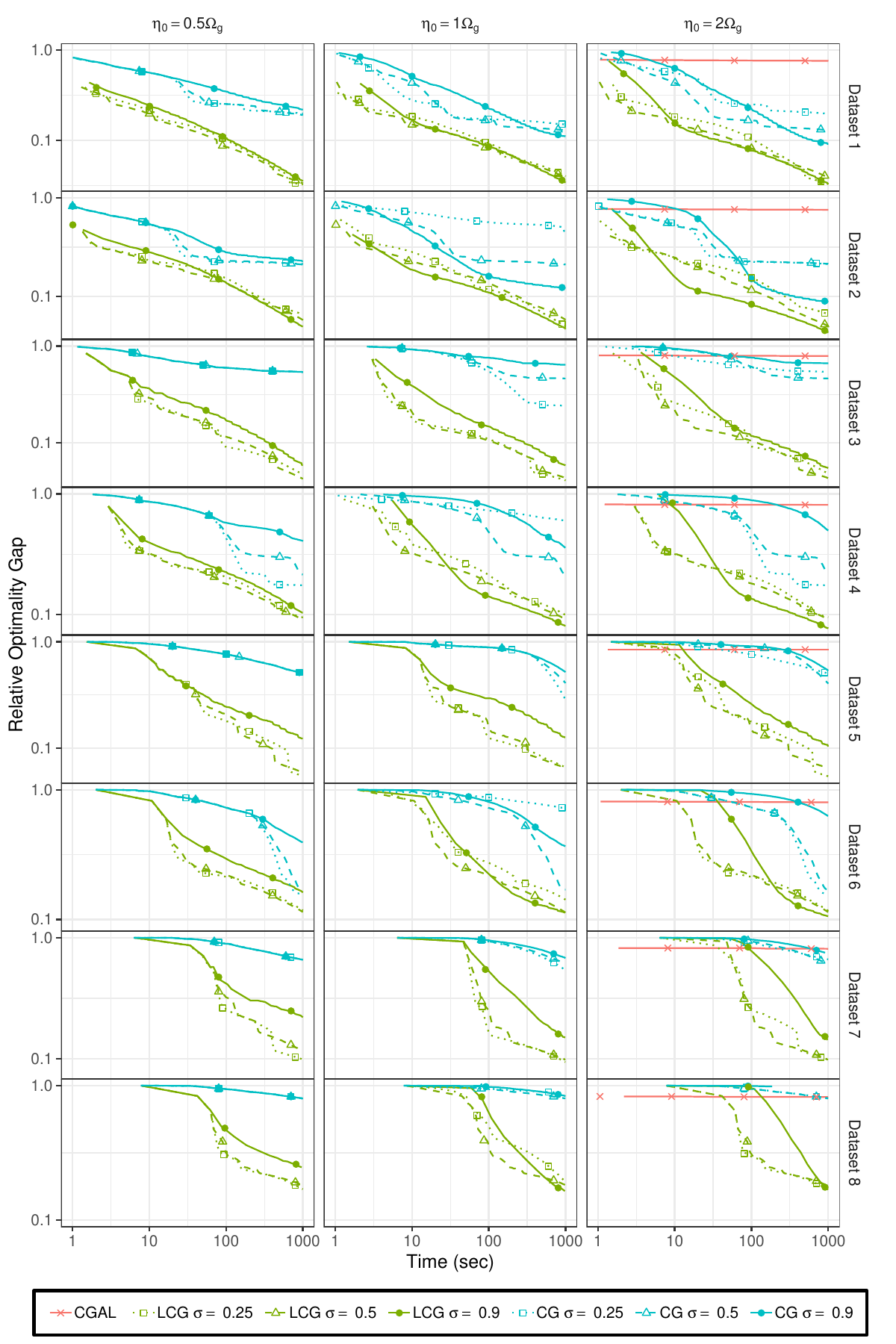}
  \caption{Mixing datasets relative \sh{optimality} gap vs. time for various parameter choices.}\label{fig:Mixing_GapVsTime}
\end{figure}

\begin{table}[t]\caption{Mixing datasets results. For each dataset, time, and algorithm provides the best parameter choice for $\eta_0$ \sh{(in multipliers of $\Omega_g$)} and $\sigma$,  the function value obtained, the relative optimality gap \sh{(Opt. Gap)}, the iteration, and \sh{for $\CGAL$} the relative feasibility gap \sh{(Feas. Gap) of the solution with the best objective value obtained until that time. The best objective value and relative optimality gap for each row are marked in bold.}}
   \label{tbl:MixingDataSetsResults}
   \centering
   \setlength{\tabcolsep}{3pt}
   \resizebox{\textwidth}{!}{
   \begin{tabular}{lr|rrrrr|rrrrr|rrrr}
	   \hline
	   Data & Time & \multicolumn{5}{|c|}{$\CG$} & \multicolumn{5}{|c|}{$\LCG$} & \multicolumn{4}{|c}{$\CGAL$}\\
	   Set&(Sec) &\multicolumn{1}{|l}{$\eta_0$}  &\multicolumn{1}{l}{ $\sigma$} & \multicolumn{1}{l}{Obj.} & \multicolumn{1}{l}{Opt.} & \multicolumn{1}{l|}{Iter} & \multicolumn{1}{|l}{$\eta_0$}  & \multicolumn{1}{l}{ $\sigma$} & \multicolumn{1}{l}{Obj.} & \multicolumn{1}{l}{Opt.} &  \multicolumn{1}{l|}{Iter} &  \multicolumn{1}{l}{Obj.} & \multicolumn{1}{l}{Opt.} & \multicolumn{1}{l}{Feas.} & \multicolumn{1}{l}{Iter}\\
	   & & ($x\cdot\Omega_g$) & & Value  & Gap \% & &  ($x\cdot\Omega_g$) & & Value  & Gap \%  & &Value & Gap &\% Gap \%\\
	   \hline
	   \multirow{3}{*}{MixData1} & 100 &   2 & 0.5 & 13.03 & 16.6 &  2385 &   2 & 0.9 & {\bf 14.39} & {\bf  7.9} & 809 & 3.68 & 76.4 & 21930.2 &  3090 \\
	   & 500 &   2 & 0.9 & 13.96 & 10.7 &  4969 &   2 & 0.25 & {\bf 14.97} & {\bf 4.1} & 9001 & 3.74 & 76.1 & 12304.4 & 11460 \\
	   & 1000 &   2 & 0.9 & 14.20 & 9.1 &  9286 & 0.5 & 0.25 & {\bf 15.12} & {\bf 3.2} & 13598 & 3.76 & 75.9 & 8873.5 & 19981 \\
	   \hline
	   \multirow{3}{*}{MixData2} & 100 &   2 & 0.9 & {6.71} & {15.3} &   957 &   2 & 0.9 & {\bf 7.26} & {\bf 8.4} & 745 & 1.88 & 76.2 & 19184.7 &  2948 \\
	   & 500 &   2 & 0.9 & 7.17 & 9.5 &  4099 &   2 & 0.9 & {\bf 7.49} & {\bf 5.5} & 3873 & 1.90 & 76.0 & 8729.7 & 10864 \\
	   & 1000 &   2 & 0.9 & 7.23 & 8.8 &  8024 &   2 & 0.9 & {\bf 7.59} & {\bf 4.3} & 7758 & 1.91 & 76.0 & 6066.0 & 18977 \\
	   \hline
	   \multirow{3}{*}{MixData3} & 100 &   1 & 0.25 & 32.86 & 54.6 &  1865 & 0.5 & 0.25 & {\bf 65.70} & {\bf 9.2} & 888 & 15.15 & 79.1 & 44622.3 &  1656 \\
	   & 500 &   1 & 0.25 & 54.34 & 24.9 &  8832 &   1 & 0.25 & {\bf 68.88} & {\bf 4.8} & 3289 & 15.30 & 78.8 & 13031.6 &  6080 \\
	   & 1000 &   1 & 0.25 & 54.87 & 24.1 & 14128 &   1 & 0.25 & {\bf 69.34} & {\bf 4.1} & 5851 & 15.32 & 78.8 & 7902.7 & 10657 \\
	   \hline
	   \multirow{3}{*}{MixData4} & 100 &   2 & 0.25 & 10.28 & 42.4 &  1676 &   2 & 0.9 & {\bf 15.42} & {\bf 13.6} & 166 & 3.33 & 81.3 & 116332.0 &  1440 \\
	   & 500 & 0.5 & 0.25 & 14.69 & 17.7 &  4710 &   2 & 0.9 & {\bf 16.15} & {\bf 9.5} & 921 & 3.36 & 81.2 & 57781.9 &  5203 \\
	   & 1000 & 0.5 & 0.25 & 14.72 & 17.5 &  6939 &   2 & 0.9 & {\bf 16.47} & {\bf 7.7} & 1880 & 3.36 & 81.2 & 41696.7 &  9081 \\
	   \hline
	   \multirow{3}{*}{MixData5} & 100 &   2 & 0.25 & 92.61 & 76.2 &  1289 &   1 & 0.25 & {\bf 340.02} & {\bf 12.5} & 415 & 59.39 & 84.7 & 122971.0 &   946 \\
	   & 500 & 0.5 & 0.9 & 163.46 & 57.9 &  4788 & 0.5 & 0.5 & {\bf 360.70} & {\bf 7.2} & 1701 & 59.31 & 84.7 & 126861.0 &  3397 \\
	   & 1000 &   1 & 0.25 & 277.79 & 28.5 &  5359 &   2 & 0.25 & {\bf 367.26} & {\bf 5.5} & 2937 & 59.90 & 84.6 & 64948.0 &  5939 \\
	   \hline
	   \multirow{3}{*}{MixData6} & 100 &   1 & 0.5 & 10.52 & 72.6 &   751 &   2 & 0.25 & {\bf 30.11} & {\bf 21.5} & 256 & 7.44 & 80.6 & 149166.0 &   811 \\
	   & 500 & 0.5 & 0.25 & 28.98 & 24.5 &  2995 &   2 & 0.9 & {\bf 33.59} & {\bf 12.5} & 232 & 7.60 & 80.2 & 103077.0 &  3004 \\
	   & 1000 & 0.5 & 0.25 & 32.30 & 15.8 &  4661 &   2 & 0.9 & {\bf 34.24} & {\bf 10.8} & 482 & 7.66 & 80.0 & 68409.9 &  5270 \\
	  \hline
	   \multirow{3}{*}{MixData7} & 100 & 0.5 & 0.9 & 35.69 & 89.3 &   496 &   1 & 0.25 & {\bf 276.90} & {\bf 16.9} &  82 & 59.19 & 82.2 & 702281.0 &   478 \\
	   & 500 &   1 & 0.25 & 102.57 & 69.2 &   984 & 0.5 & 0.25 & {\bf 295.47} & {\bf 11.3} & 616 & 62.41 & 81.3 & 239810.0 &  1766 \\
	   & 1000 &   1 & 0.25 & 152.56 & 54.2 &  2101 &   1 & 0.25 & {\bf 301.48} & {\bf 9.5} & 1447 & 62.78 & 81.2 & 185699.0 &  3124 \\	  	   \hline
\multirow{3}{*}{MixData8} & 100 & 0.5 & 0.9 & 4.49 & 94.4 &   306 & 0.5 & 0.25 & {\bf 55.80} & {\bf 30.3} &  66 & 13.77 & 82.8 & 613444.0 &   431 \\
	   & 500 & 0.5 & 0.9 & 12.03 & 85.0 &  1455 &   2 & 0.25 & {\bf 63.64} & {\bf 20.5} & 566 & 13.99 & 82.5 & 282965.0 &  1593 \\
	   & 1000 & 0.5 & 0.9 & 15.84 & 80.2 &  2751 &   1 & 0.9 & {\bf 66.98} & {\bf 16.3} & 133 & 14.05 & 82.4 & 243686.0 &  2817 \\
	   \hline
   \end{tabular}}
\end{table}

\sh{Figures~\ref{fig:Mixing_GapVsIter} and \ref{fig:Mixing_GapVsTime} indicate that} the line search version $\LCG$ outperforms $\CG$, although in most cases less iterations are performed. One explanation of this phenomenon is that the step size policy employed in $\CG$ is based on global optimization ideas which do not take into account the local structure of the problem. $\LCG$ instead captures these local features of the problem 
\gi{more accurately},
leading to better numerical performance.
By comparison, $\CGAL$ performs quite poorly, with optimality gaps ranging from $70\%$ to $90\%$ roughly; \sh{this is closely linked to large relative feasibility gaps, exceeding $10000\%$.} One possible reason for these results is that $\CGAL$'s convergence guarantees rely on the magnitude of the  dual variables, which in turn depend on the size of the Slater condition gap. Since the right-hand side $d_{ij}^2$ of the constraints may be very small, this gap will generally be small, causing the norm of the dual variables to become large and slowing convergence. This may also explain the contrast with $\CGAL$'s superior performance in the MaxCut example. \sh{Additionally, Table~\ref{tbl:MixingDataSetsResults} does not indicate that there is a preferable choice of parameters for $\CG$ and $\LCG$, since the best configuration is different for different data sets and times. 
}

\subsection{Randomly Scaled SDP}\label{sec:SRS}

Finally, we consider a general SDP problems with inequality constraints, which we use to investigate the sensitivity of the compared methods to problem scaling. The synthetically scaled SDP (SRS) takes the form
\begin{equation}\label{eq:SRS}
   \begin{aligned}
	   &\max&&\inner{\bC,\bX}\\
	   &\text{s.t.}&& \inner{\bA_i,\bX}\leq b_i\quad i=1,\ldots,m\\
	   &&& \tr{(\bX)}\leq 1\\
	   &&& \gi{\bX}\succeq 0.
   \end{aligned}
\end{equation}
\sh{Let $\setQ=\{\by\in\R^{\gi{m}}\vert y_i\leq b_i,i\in \{1,2,\ldots,m\}\}$ and $\setK=\{(\gi{\by},t)\in\R^{m}\times\R\vert\frac{1}{t}\gi{\by}\in\setQ,t>0\}$. This is a closed convex cone with logarithmically homogeneous barrier
\[
f(\gi{\by},t)=-\sum_{i=1}^m\log(t b_i-y_{i})=-\sum_{i=1}^m\log(b_i-\frac{1}{t}y_{i})-m\log(t),
\]
i.e., $F(\bX,t)=f(\scrP(\bX,t))$, where $\scrP(\bX,t)=[\inner{\bA_1,\bX};\inner{\bA_2,\bX};\ldots;\inner{\bA_m,\bX};t]\in\R^{m}\times\R$.
We thus arrive at a formulation of the form of problem \eqref{eq:Opt}, which reads explicitly as
\begin{equation}\label{eq:SRS2}
   \begin{split}
	   \min_{\bX,t}\quad &g(\bX):=\inner{-\bC,\bX}\\
	   \text{s.t.}\quad &\scrP(\bX,t)\in\setK\\
	   & (\bX,t)\in\setX:=\{(\bX',t')\in\scrS^{n}\times\R\vert \bX'\succeq 0,\;\tr(\bX')\leq 1,t'=1\}.
   \end{split}
\end{equation}}

%
\sh{We randomly generate the objective function and constraint coefficients to be PSD matrices with unit Frobenius norm \gi{where, given a matrix $M \in \mathbb{R}^{m_1 \times m_2}$, its Frobenius norm is $\|M\|_\text{F} := \sqrt{\sum_{i =1}^{m_1} \sum_{j=1}^{m_2} M_{ij}^2}$}. Specifically, the constraint matrices for every instance and $i=1,\ldots,m$ are generated by first sampling $\bu_i\in \R^{n}$ with components from a standard Gaussian distribution, rounded to the nearest integer, then setting $\bA_i=\bu_i\bu_i^\top$, and finally normalizing the matrix to Frobenius norm of 1. Similarly, the objective matrix is generated by sampling $\UU_0\in\R^{n\times n}$ with components from a standard Gaussian distribution, rounded to the nearest integer, and $\bC=(\UU_0\UU_0^\top)/\norm{(\UU_0\UU_0^\top)}_F$. Additionally, we generate a PSD reference solution $\bX_0$ by sampling $\bV\in \R^{n\times n}$ with components from a standard Gaussian distribution, then setting $\bX_0=\bV\bV^\top$ and, finally, normalizing the result to have a trace equal to 1.
The right-hand side constraint coefficients are constructed as $b_i = \frac{u_i}{i^{p/2}} \inner{\bA_i, \bX_0}$, with $u_i$ is a random variable uniformaly disributed over $[\frac{1}{2}, 1]$, and $p \in \{0, 1, 2\}$, where $p$ is a parameter that measures the scaling of the constraints. 
We note that $\CGAL$ ran on the normalized problem, using the normalization guidelines specified in \cite{YurTroFerUdeCev21}.}

Similarly to the previous examples, the performance metric to benchmark our experimental results is the relative optimality gap, defined $\frac{|g(\gi{\bX}) - g(\gi{\bX}^{\SDPT3}))|}{|g(\gi{\bX}^{\SDPT3})|}$, where $\gi{\bX}^{\SDPT3}$ is the optimal solution computed by the interior-point solver $\SDPT3$.
We compare our solutions with those generated by $\CGAL$.
However, since $\CGAL$ does not necessarily produce feasible solutions, its reported solution $\bX^{\CGAL}$ is corrected to ensure feasibility.
\sh{Specifically, given a relative feasibility gap $\gamma=\max_{i\in\{1,2,\ldots,m\}}\frac{\max\{\{\inner{\bA_i,\bX}-\bb_i\},0\}}{\bb_i}$, the corrected solution was given by $\tilde{\bX}^{\CGAL}=\bX^{\CGAL}/(1+\gamma)$, and the relative optimality gap is obtained by using $\tilde{\bX}^{\CGAL}$.}

\sh{
\gi{For each value of $p$, we ran $\CGAL$, $\LCG$ and $\CG$ on 30 realizations of problems of size $n=100$ and $m=100$.}
\sh{All datasets were run for both $\CG$ and $\LCG$ with parameters $\eta_0=2\Omega_g$
and $\sigma\in\{0.25,0.9\}$. and all methods were stopped when either after $1000$ seconds of running time.}
For every $p\in\{0,1,2\}$, realization $R\in\{1,2,\ldots,30\}$ of the data and algorithm $M$, we record the objective value $g^{M,p,R}_i$ (corrected as explained above for $\CGAL$) obtained after the iteration, the corresponding relative optimality gap $r^{M,p,R}_i$, the time since the start of the run $t^{M,p,R}_i$ and, for $\CGAL$, the relative feasibility gap $\gamma^{\CGAL,p,R}_i$.}

\begin{table}[H]
\caption{Computational results for randomly scaled SDP datasets. For each time point, value of the parameter $p$, and algorithm ($\CG$ and $\LCG$ with $\sigma \in \{0.25, 0.9\}$, and $\CGAL$), the table provides the objective function value (Obj. Value), the relative optimality gap (Opt. Gap), the iteration count and, only for $\CGAL$, the relative feasibility gap (Feas. Gap) of the solution with the best objective value obtained up to that time. For each value of $p$, the metrics displayed in the table are aggregated across 30 datasets, and the best relative optimality gap achieved is marked in bold.}
\label{tbl:srs_tab}
\centering
\setlength{\tabcolsep}{3pt}
\resizebox{1\textwidth}{!}{
\gi{
\begin{tabular}{lr|lrr|lrr|lrr|lrr|lrrr}
\hline
Data & Time & \multicolumn{3}{|c|}{$\CG$ $\sigma=0.25$} & \multicolumn{3}{|c|}{$\CG$ $\sigma=0.9$} & \multicolumn{3}{|c|}{$\LCG$ $\sigma=0.25$} & \multicolumn{3}{|c}{$\LCG$ $\sigma=0.9$} & \multicolumn{4}{|c}{$\CGAL$} \\ \hline
$p$ & (sec) & \multicolumn{1}{|l}{Obj.} & \multicolumn{1}{l}{Opt.} & \multicolumn{1}{l|}{Iter} & \multicolumn{1}{|l}{Obj.} & \multicolumn{1}{l}{Opt.} & \multicolumn{1}{l|}{Iter} & \multicolumn{1}{|l}{Obj.} & \multicolumn{1}{l}{Opt.} & \multicolumn{1}{l|}{Iter} & \multicolumn{1}{|l}{Obj.} & \multicolumn{1}{l}{Opt.} & \multicolumn{1}{l|}{Iter} & \multicolumn{1}{|l}{Obj.} & \multicolumn{1}{l}{Opt.} & \multicolumn{1}{l}{Feas.} & \multicolumn{1}{l}{Iter} \\ & & Value & Gap \%& & Value & Gap \%& & Value & Gap \%& & Value & Gap \%& & Value & Gap \%& Gap \%&  \\ \hline\hline
\multirow{3}{*}{0} & 100 & -186.45 & 0.671 & 2067 & -186.78 & 0.4978 & 1467 & -186.79 & 0.4908 & 1543 & -186.99 & 0.3875 & 1305 & -187.67 & \textbf{0.0059} & 0.0183 & 13349 \\
& 500 & -187.19 & 0.2788 & 6665 & -187.34 & 0.198 & 6120 & -187.28 & 0.2272 & 6095 & -187.4 & 0.1666 & 6201 & -187.7 & \textbf{0.0009} & 0.0049 & 50070 \\
& 1000 & -187.33 & 0.2027 & 13071 & -187.46 & 0.1308 & 11662 & -187.4 & 0.1653 & 11918 & -187.5 & 0.1139 & 12133 & -187.7 & \textbf{0.0004} & 0.0041 & 88240 \\
\hline
\multirow{3}{*}{1} & 100 & -143.26 & 1.3039 & 1131 & -143.56 & 1.0994 & 1020 & -143.81 & 0.9277 & 1099 & -144.01 & \textbf{0.786} & 783 & -141.99 & 1.7272 & 2.1609 & 13266 \\
& 500 & -144.38 & 0.5311 & 3976 & -144.53 & 0.423 & 3841 & -144.58 & 0.3932 & 3626 & -144.66 & 0.3374 & 3427 & -144.26 & \textbf{0.2718} & 0.6039 & 50306 \\
& 1000 & -144.62 & 0.3621 & 7458 & -144.73 & 0.2838 & 7314 & -144.72 & 0.2955 & 7014 & -144.8 & 0.235 & 6790 & -144.83 & \textbf{0.1209} & 0.213 & 87925 \\
\hline
\multirow{3}{*}{2} & 100 & -107.75 & 1.5629 & 767 & -108.18 & 1.1652 & 686 & -108.31 & 1.0371 & 612 & -108.64 & \textbf{0.7336} & 569 & -32.44 & 70.2786 & 265.9191 & 14138 \\
& 500 & -108.82 & 0.5716 & 2868 & -109.02 & 0.3879 & 2724 & -108.98 & 0.4199 & 2727 & -109.14 & \textbf{0.276} & 2603 & -65.05 & 40.3478 & 74.8299 & 50504 \\
& 1000 & -108.98 & 0.4222 & 5519 & -109.15 & 0.2695 & 5283 & -108.64 & 0.3322 & 5374 & -109.22 & \textbf{0.2037} & 5191 & -82.82 & 24.0898 & 35.2035 & 90162 \\
\hline
\end{tabular}}
}
\end{table}

Table \ref{tbl:srs_tab} 
displays the performance of $\CGAL$, $\CG$ and $\LCG$ after time points $t\in \{100,500,1000\}$ seconds for different values of the parameter $p$, in order to study algorithmic behaviour under different time constraints.
\sh{Specifically, for each 
method $M$, parameter value $p$, realization $R$, and time $t$,  we compute $
\gi{(j^*)^{M,p,R}}
=\argmin_{j}\{r^{M,p,R}_\gi{j}:t_j^{M,p,R}\leq t\}$ and $k^{M,p,R}=\max\{j:t_j^{M,p,R}\leq t\}$. 
The table displays the average 
objective function value (Obj. Value) given by $\frac{1}{30}\sum_{R} g^{M,p,R}_{(j^*)^{M,p,R}}$, the average relative optimality gap (Opt. Gap) $\frac{1}{30}\sum_{R} r^{M,p,R}_{(j^*)^{M,p,R}}$, as well as the number of iterations (Iter) until that time  given by $\frac{1}{30}\sum_{R}k^{M,p,R}
$ and, for $\CGAL$, the average feasibility gap (Feas. Gap) $\frac{1}{30}\sum_R \gamma^{M,p,R}$.}
Table \ref{tbl:srs_tab} shows that $\CGAL$ benefits from inexpensive iterations, allowing it to close the optimality gap more quickly than $\CG$ and $\LCG$ in favorable settings ($p=0$).
However, as $p$ increases to 2, although $\CGAL$ still performs more iterations than our algorithms, it fails to close the infeasibility gap even after 1000 seconds.
This leads to a significantly larger optimality gap, as our calculations penalize infeasible solutions.
In contrast, $\CG$ and $\LCG$ demonstrate robust performance, with little variation across different values of $p$. \sh{In particular $\LCG$ reaches an average optimality gap smaller than $1\%$ after 100 seconds for all values of $p$, while the average optimality gap of  $\CG$ increases by a factor of 2-2.5 when moving from $p=0$ to $p=2$, but still maintaining less than $1\%$ average optimality gap after 500 seconds for all values of $p$.}

\sh{The performance of the average relative optimality gap with respect to iteration and time for each value of $p$ is visually depicted in Figures~\ref{fig:SRS_gap_iter_supp} and \ref{fig:SRS_gap_time_supp}, respectively. In these figures, it can clearly be observed that the performance of $\CG$ and $\LCG$ is robust to the value of $p$, while $\CGAL$'s performance deteriorates significantly as deteriorates as $p$ gets larger.}
\gi{
In order to create Figure \ref{fig:SRS_gap_iter_supp} we first compute, for each iteration $i$, algorithm $M$, parameter value $p \in \{0, 1, 2\}$ and realization $R \in \{1, \dots, 30\}$, the best relative optimality gap $\hat{r}^{M,p,R}_i = \min_{j\leq i}\{r^{M,p,R}_j\}$ up to iteration $i$.
We then plot its average $\frac{1}{30}\sum_{R} \hat{r}^{M,p,R}_i$ across all realizations, for each iteration $i$.
Instead, Figure \ref{fig:SRS_gap_time_supp} is obtained by first computing, for each time $t$, algorithm $M$, parameter value $p$ and realization $R$, the best relative optimality gap \sh{$r^{M,p,R}_{(j^*)^{M,p,R}}$ obtained up to time $t$ (as defined above)}, and then by averaging these values across all realizations.}
\sh{These figures elucidate that while $\CGAL$ performance deteriorates as $p$ increases, the performance of $\CG$ and $\LCG$ is robust to the scaling of the problem, .
}

\begin{figure}[ht!]
\centering
\includegraphics[scale=0.37]{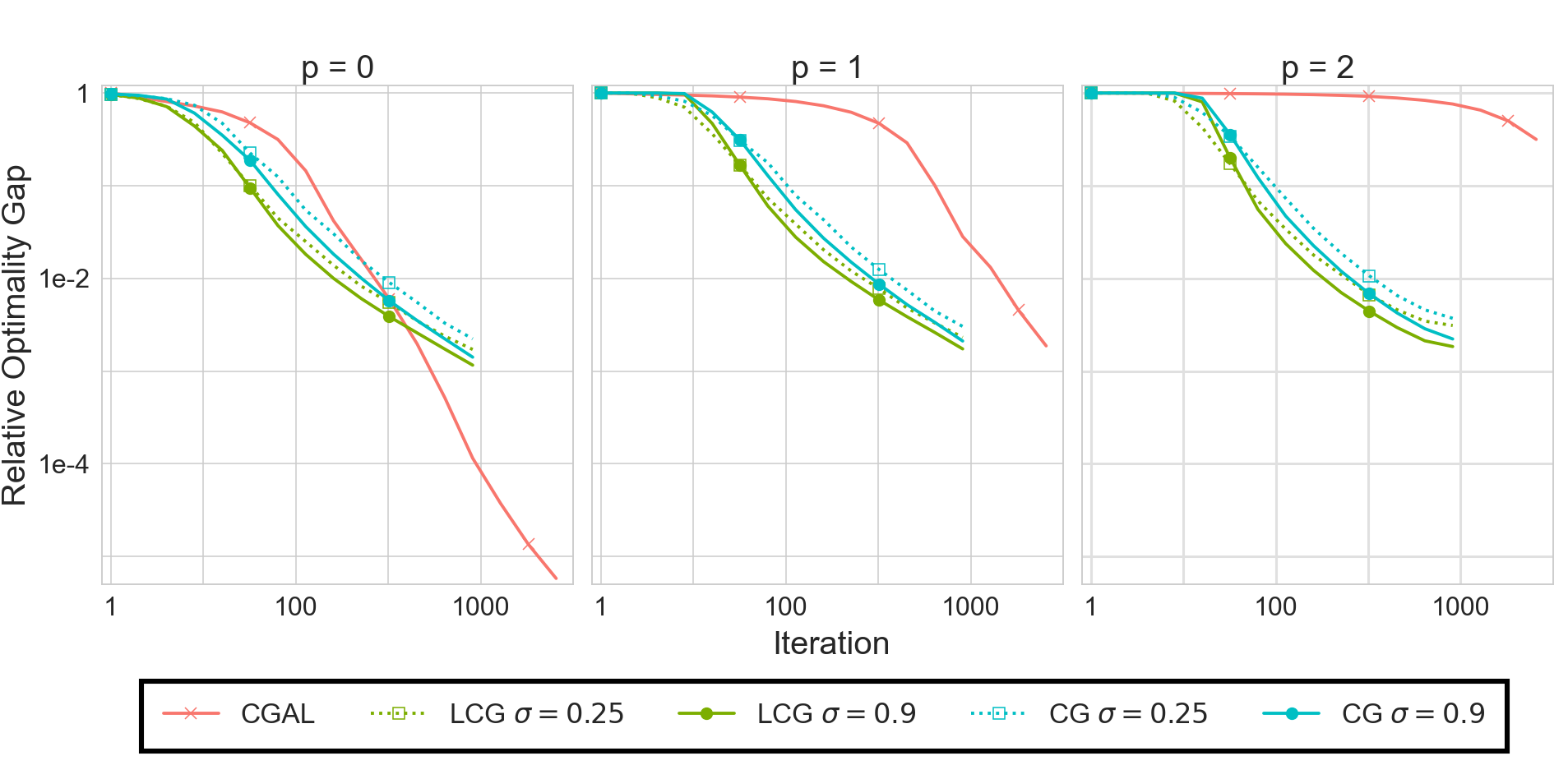}
\caption{Randomly scaled SDP problem: relative optimality gap vs. iterations for several parameter choices (averaged over 30 instances for each $p$)}
\label{fig:SRS_gap_iter_supp}
\end{figure}

\begin{figure}[ht!]
\centering
\includegraphics[scale=0.37]{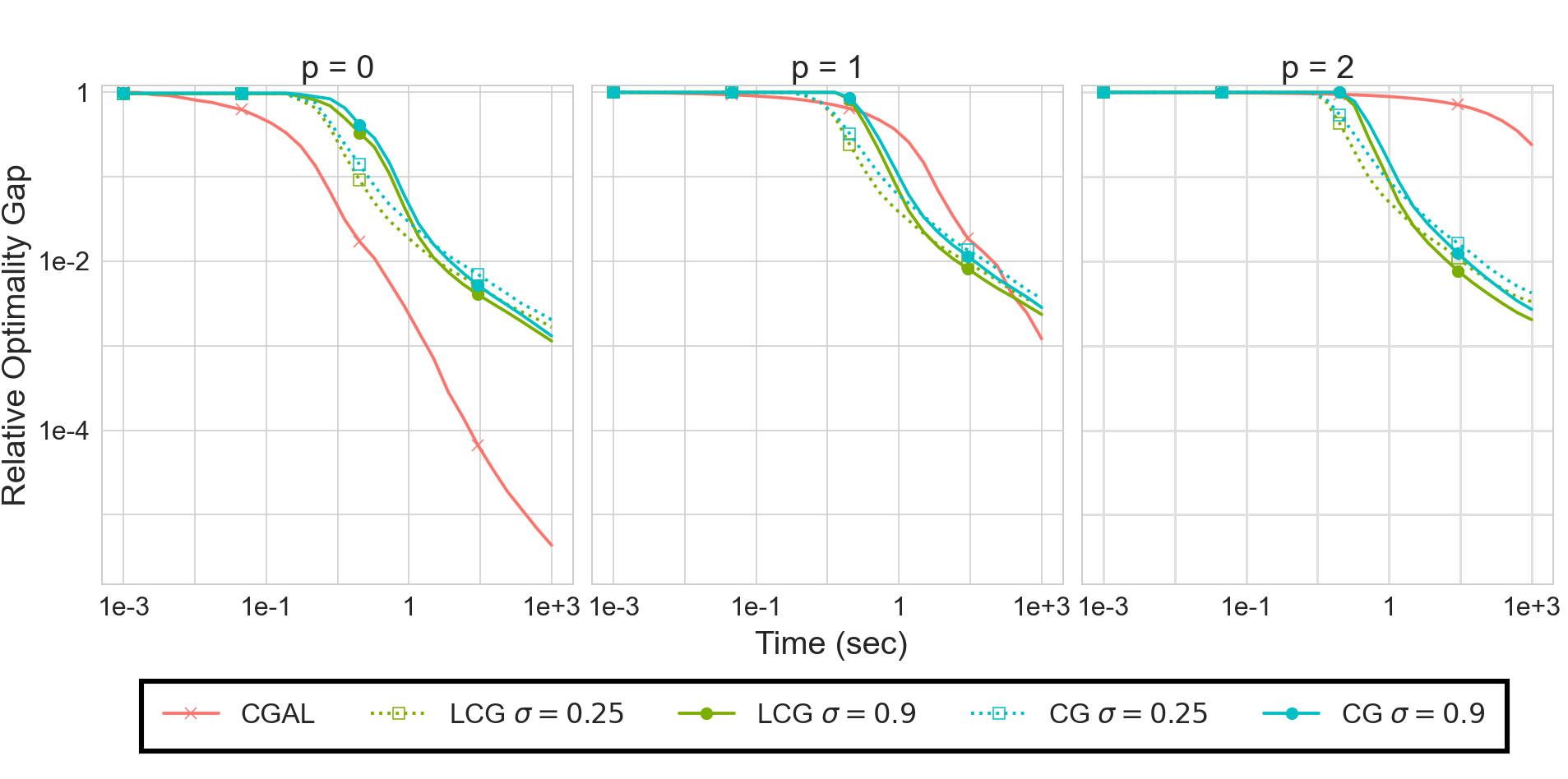}
\caption{Randomly scaled SDP problem: relative optimality gap vs. time for several parameter choices (averaged over 30 instances for each $p$)}
\label{fig:SRS_gap_time_supp}
\end{figure}

\end{document}